\documentclass[12pt,reqno]{amsart}

\usepackage[french, english]{babel}
\usepackage[normalem]{ ulem }
\usepackage{soul}
\usepackage{eucal}
\usepackage{flafter}
\usepackage{amscd,amsthm}
\usepackage{amsfonts}
\usepackage{graphicx}
\usepackage{mathtools} 
\mathtoolsset{showonlyrefs}
\usepackage{hyperref}

\usepackage{amssymb, amsmath, amsthm, amsfonts,amstext, graphicx}

\usepackage{mathptmx}       
\usepackage{helvet}         

\usepackage{bm}
\newcommand*{\BF}[1]{\ifmmode\bm{#1}\else\textbf{#1}\fi}
\usepackage{stackengine}
\newcommand\ocirc[1]{\ensurestackMath{\stackon[1pt]{#1}{\mkern2mu\circ}}}

\usepackage{courier} 
\usepackage{listings}
\usepackage{color}
 
\definecolor{codegreen}{rgb}{0,0.6,0}
\definecolor{codegray}{rgb}{0.5,0.5,0.5}
\definecolor{codepurple}{rgb}{0.58,0,0.82}
\definecolor{backcolour}{rgb}{0.95,0.95,0.92}
 
\lstdefinestyle{mystyle}{
    backgroundcolor=\color{backcolour},   
    commentstyle=\color{codegreen},
    keywordstyle=\color{magenta},
    numberstyle=\tiny\color{codegray},
    stringstyle=\color{codepurple},
    basicstyle=\footnotesize,
    breakatwhitespace=false,         
    breaklines=true,                 
    captionpos=b,                    
    keepspaces=true,                 
    numbers=left,                    
    numbersep=5pt,                  
    showspaces=false,                
    showstringspaces=false,
    showtabs=false,                  
    tabsize=2
}
 
\usepackage{ulem}

\usepackage{footnote}

\newtheorem{theorem}{Theorem}[section]
\newtheorem{lem}[theorem]{Lemma}

\newtheorem{prop}[theorem]{Proposition}
\newtheorem{defn}{Definition}[section]

\newtheorem{rem}{Remark}[section]
\newtheorem{exa}{Example}[section]
\numberwithin{equation}{section}

\renewcommand{\a}{{\alpha}}

\newcommand{\e}{\varepsilon}
\newcommand{\de}{\delta}
\newcommand{\fa}{\varphi}
\newcommand{\ga}{\gamma}
\renewcommand{\k}{{\BF k}}
\newcommand{\la}{\lambda}

\newcommand{\si}{\sigma}

\newcommand{\De}{\Delta}
\newcommand{\Ga}{\Gamma}

\newcommand{\Om}{\Omega}
\newcommand{\lan}{\langle}
\newcommand{\ran}{\rangle}

\newcommand{\ab}{{{\BF a}_0}}
\newcommand{\G}{{\mathcal G}}
\newcommand{\PZ}{{\mathbb Z}_+}
\newcommand{\V}{\Gamma}

\def\T{{\mathbb T}}
\def\N{{\mathbb{N}}}
\def\T_N{{\T_N}}
\def\T{{\mathbb{T}}}

\def\M{{\mathcal{M}}}
\allowdisplaybreaks

\def\R{{\mathbb R}}

\def\Bg{{\mathbf{g}}}
\def\Bk{{\BF{k}}}

\def\Y{{\mathcal Y}}
\def\aa{{\BF a}}
\def\W{{\mathcal W}}
\def\A{{\mathcal A}}
\def\E{{\mathbb E}}
\renewcommand\P{{\mathbb P}}
\def\M{{\mathcal M}}

\newcommand{\HW}{\widetilde{H}}
\newcommand{\I}{\mathcal{I}}
\newcommand{\B}{\mathcal{B}}
\newcommand{\K}{\mathcal{K}}
\def\M{{\mathcal M}}
\def\K{{\mathcal K}}
\def\D{{\mathcal D}}
\def\I{{\mathcal I}}
\def\F{{\mathcal F}}
\def\z{{\BF z}}
\def\t{{\BF t}}

\def\y{{\BF y}}
\def\Z{{\mathbb Z}}

\newcommand\ve{\varepsilon}

\newcommand{\mc}[1]{{\mathcal #1}}

\newcommand{\bb}[1]{{\mathbb #1}}

\title[Derivation of coupled KPZ-Burgers equation]{Derivation of coupled KPZ-Burgers equation \\
from multi-species zero-range processes}

\author{C. Bernardin}
\address{Universit\'e C\^ote d'Azur, CNRS, LJAD\\
Parc Valrose\\
06108 NICE Cedex 02, France}
\email{{\tt cbernard@unice.fr}}

\author{T. Funaki} 
\address{Department of Mathematics\\
School of Fundamental Science and Engineering\\
Waseda University\\
3-4-1 Okubo, Shinjuku-ku\\
Tokyo 169-8555, Japan}
\email{{\tt funaki@ms.u-tokyo.ac.jp}}

\author{S. Sethuraman}
\address{Department of Mathematics\\
University of Arizona\\
621 N. Santa Rita Ave.\\
Tucson, AZ 85750, USA}
\email{{\tt sethuram@math.arizona.edu}}

\date{\today}

\begin{document}

\begin{abstract}
We consider the fluctuation fields of multi-species weakly-asymmetric zero-range interacting particle systems in one dimension, where the mass density of each species is conserved.  Although such fields have been studied in systems with a single species, the multi-species setting is much less understood.  Among other results, we show that, when the system starts from stationary states, with a particular property, the scaling limits of the multi-species fluctuation fields, seen in a characteristic traveling frame, solve a coupled Burgers SPDE, which is a formal spatial gradient of a coupled KPZ equation.  
\end{abstract}

\maketitle

\smallskip
\noindent {\it{Keywords: }}{interacting, particle system, zero-range, multi-species, weakly asymmetric, fluctuation, field, Burgers, coupled, KPZ}

\smallskip
\noindent {\it 2010 Mathematics Subject Classification.} 60K35, 82C22, 60H15, 60G60, 60F17


\section{Introduction}
The purpose of this article is to consider scaling limits of the empirical mass fluctuation fields of weakly asymmetric multi-species zero-range processes on a sequence of tori $\T_N = \{0,1,\ldots, N-1\}$, as an approximation of ${\mathbb Z}$, where the mass density of each species is conserved, starting from stationary initial conditions.  Although such limits have been worked out with respect to single species zero-range and other particle models, there appears to be little work on fluctuations for multi-component systems.  
In particular, in a certain characteristic frame of reference, the single component fluctuation limit, starting from several types of initial conditions, through a variety of techniques, has been shown to satisfy a `scalar' KPZ-Burgers equation, an object at the center of much recent study:  For instance, \cite{ACQ}, \cite{Assing}, \cite{BG}, \cite{Corwin_Tsai}, \cite{Dembo_Tsai}, \cite{Funaki_Quastel}, \cite{GJ}, \cite{GJS}, \cite{GPS}, \cite{Gub_J}, \cite{Gubinelli_Imkeller_Perkowski},  \cite{GP18}, \cite{GPnew}, \cite{H}, \cite{Kupiainen}, \cite{Ortmann_Quastel_Remenik}, \cite{Sasamoto_Spohn}, and reviews \cite{Corwin_review}, \cite{KK}, \cite{Quastel_Spohn}, and references therein.
In this context, it is a natural question to investigate the connections between the fluctuations when there is more than one conservation law and `coupled' KPZ-Burgers systems.

On the other hand, recently, in coupled driven systems with several conservation laws, the works \cite{PSSS} and \cite{Spohn_Stolz} have proposed a suite of universality classes for the scaling limits of the associated fluctuations.
 In this respect, one might ask in what sense do the fluctuations of the weakly-asymmetric systems with several species, that we study, relate to these works.

In a nutshell, our main result, Theorem \ref{secondorder_thm}, is that, when the characteristic speed of each component is the {\it same}, the mass fluctuation limit of the multi-species process, seen in traveling frame with this speed, satisfies a type of martingale problem for a coupled KPZ-Burgers equation, shown to have unique solution in \cite{GPnew}.  
Our condition, called the `Frame condition (FC)', that the characteristic speeds of all the fluctuation field components be the same is a condition on the the density of the stationary state and as well the rate function of the zero-range process.  In particular, not every stationary state satisfies this condition, and part of the results is to determine when this condition holds.  This is the first microscopic derivation of a singular coupled SPDE system. 

The methods of the article are to develop martingale representations of the fluctuation fields and to close equations by averaging nonlinear rate terms in terms of the fields themselves through a multi-component Boltzmann-Gibbs principle.  Limit points of the scaled fields are shown to satisfy a multi-component `energy' formulation, which extends that in the `scalar' case \cite{GJ}, \cite{GJS}.  Combining with the recent work \cite{GPnew} (see also \cite{GP18} in the `scalar' setting), which shows uniqueness of solutions of the `energy' framework, we may identify the process.

 Although we work on tori $\T_N$, all of our results and proofs translate straightforwardly to systems in the infinite volume ${\mathbb Z}$.  However, since the uniqueness of the martingale problem for the general coupled KPZ-Burgers equations has only been shown in finite volume $\T=[0,1)$ \cite{GPnew}, we have chosen to work on the discrete tori $\T_N$.  We remark, however, when the multi-component particle system is a `multi-color' zero-range system, where particles may be one of $n$ colors but behave the same mechanistically, uniquness criteria on ${\mathbb Z}$ for the associated martingale formulation is known, and we explain in this case in Subsection \ref{multi_color_kpz} that our main result, Theorem \ref{secondorder_thm}, can be stated as an infinite volume result on ${\mathbb Z}$.  

Since all components of the coupled SPDE that we derive (cf. \eqref{gen_OU_sec}) have the same drift, a regime not considered in \cite{PSSS} or \cite{Spohn_Stolz}, the question of which universality class its fluctuations belong to is perhaps complementary to the discussion in these papers.  In this respect, we detail, in Section \ref{example_section}, in `multi-color' systems, when the general SPDE reduces to the SPDE \eqref{eq:5.2_reduced}, a certain structure in the limit equation.  Moreover, with respect to the general SPDE \eqref{gen_OU_sec}, we show it satisfies a `trilinear' condition \eqref{eq:7.tri} in Section \ref{trilinear} (cf. \cite{F18}, \cite{FH17}), and comment that a more varied equation structure looks possible.
Some calculations are given in Section \ref{ND_section} when $n=2$, as well as questions for further work when $n\geq 3$ in this direction.

To be complete, we also consider the mass fluctuation limits of the multi-species process, in a fixed, unmoving frame, which satisfy linear SPDE, Theorem \ref{drift_thm}.

After a brief discussion of the notation used in the article, we now introduce informally the model considered, and make heuristic computations in the following subsections of the Introduction, which convey the ideas leading to our main results.  In Section \ref{models}, we detailed more carefully the models studied, and in Section \ref{results}, we state the results.  In Section \ref{proofs}, proofs are given, relying on a multi-component `Boltzmann-Gibbs' principle, shown in Section \ref{proof_Boltzmann}.  Finally, in Sections \ref{trilinear}, \ref{example_section} and \ref{ND_section}, we show the general SPDE limit satisfies a `trilinear' condition, discuss the `multi-color' reduced model, and discuss the `non-decoupled' structure of the general limit SPDE.

\section*{Notation}
We use boldface letters to denote vectors: if ${\BF a} \in \R^n$, its components are denoted by $(a^1,\ldots, a^n)$ or $(a^i)_{i=1}^n$. The usual scalar product in $\R^n$ is denoted by $\cdot$ and the corresponding norm by $\| \cdot \|$. The $L^1$ norm of $\BF a$ is denoted by $|\BF a| =|a^1|+\cdots+|a^n|$. The entries of a matrix $M$ of size $p\times n$ are denoted by $M_{ij}$, $1\le i \le p$, $1\le j \le n$, and we write $M=(M_{ij})_{1\le i\le p,  1\le j \le n}$. If $M$ is diagonal then we write $M={\rm{diag}}((M_{ii})_{1\le i \le n})$. The continuous torus $[0,1)$ is denoted by $\T$ and the space of $\R^n$-valued smooth functions on $\T$ by ${\mathcal D} (\T)^n$. If $H,G: \T \to \R^n$ are $\R^n$-valued functions, their $L^2$-scalar product is denoted by $(H,G)_{L^2 (\T)^n}$ and the corresponding norm by $\|\cdot\|_{L^2 (\T)^n}$.

 If $\Psi :=(\Psi^1, \ldots, \Psi^p): {\BF a} \in \R^n \to \Psi ({\BF a}) \in \R^p$ is a differentiable function, we denote the $j$th partial derivative of $\Psi^i$ by $\partial_{a^j} \Psi^i$ and by $\nabla \Psi$ the $p\times n$ matrix with entries given by $\Big( \partial_{a^j} \Psi^i \Big)_{1\le i \le p, 1 \le j \le n}$. 

Given two real-valued functions $f$ and $g$ depending on the variable $u \in \bb R^d$ we will write $f(u) \approx g(u)$ if there exists a constant $C>0$ which does not depend on $u$ such that for any $u$, $C^{-1} f(u) \le g(u) \le C f(u)$ and $ f(u) \lesssim g(u)$ if for any $u$, $f(u) \le C g(u)$. We write $f = O (g)$ (resp. $f=o(g)$) in the neighborhood of $u_0$ if $| f| \lesssim | g|$ in the neighborhood of $u_0$ (resp. $\lim_{u \to u_0} f(u)/g(u) =0$). Sometimes it will be convenient to 
make precise the dependence of the constant $C$ on some extra parameters and this will be done by the standard notation $C(\lambda)$ if $\lambda$ is the extra parameter.

Usually, the notation $E_\mu$ denotes the expectation with respect to the probability measure $\mu$ and the notation $\bb E_\mu$ is reserved for expectation with respect to the law $\bb P_\mu$ of the paths of the involved stochastic processes starting from $\mu$.  The subscript is omitted when the context is clear.

\subsection{Summary of derivations}

Informally, the single component zero-range process consists of a collection of particles on $\T_N$ which interact infinitesimally only with those on their own spatial location.  More precisely, if there are $k$ particles at $x\in \T_N$, then with rate $g(k)$ one of these particles jumps, moving to a location $y$ with probability $p(y-x)$.  In the multi-component process, there are several types of particles, say $n\ge 1$.  For $1\leq i\leq n$, when there are $k_i$ particles of type $i$ at location $x$, a type $i$ particles jumps with rate $g_i(k_1,\ldots, k_n)$, moving to $y$ again with chance $p(y-x)$.  

The weakly-asymmetric process $\{\BF{\alpha}_t^N\, ; \, t\ge 0\} = \{(\alpha^{1,N}_t,\ldots, \alpha^{n,N}_t)\, ;\, t\ge 0\}$ corresponds to when the jump probability $p$ is in the form $p(1) = 1/2 + c/N^{\gamma}$, $p(-1)=1/2-c/N^{\gamma}$ and $p(z)=0$ otherwise.  Here, $N$ is a scaling parameter, $c\in {\mathbb R}$ and $\gamma> 0$.  We will study the process in diffusive time scale. Accordingly, the configuration $\a^N_t : x\in \T_N \to \a^N_t(x)\in \Z_+^n $ specifies the number of particles of each type $1\leq i\leq n$ at each location $x\in \T_N$ at time $N^2t$ for $t\geq 0$. Observe that the number of particles of each type is preserved during the time evolution. Therefore we expect that the system has a family of invariant measures indexed by type particle densities. In Subsection \ref{invariant_section}, we discuss that the multi-component systems possesses a family of invariant product measures $\nu_\ab$, indexed by component densities $\ab=(a^1_0,\ldots, a^n_0)$, when the rates $(g_1,\ldots, g_n)$ satisfy a compatibility relation.  

\subsubsection{Hydrodynamics}

Before developing the fluctuation limits, it will be useful to consider the hydrodynamics or law of mass averages of the process. 
We take $\ga=1$ such that $p(1)= \frac12+\frac{c}N$ and $p(-1)= \frac12-\frac{c}N$.
Consider a $n$-dimensional valued macroscopic empirical profile process $\{{\BF X}_t^N \; ;\; t\ge 0\} = \{(X_{t}^{i,N})_{i=1}^n\; ; \; t\ge 0\}$ defined for any time $t\ge 0$ by
$$
X_{t}^{i,N}(u) = \sum_{x=1}^{N-1} \alpha_{t}^{N,i}(x) \; {\BF 1}_{\big[\tfrac{x}N, \tfrac{x+1}{N}\big) } (u), \quad u \in \T.
$$

The hydrodynamic limit for the $n$-species system can be stated as follows: for each time $t\ge 0$, $({\BF X}_{t}^{N})_N$ converges \footnote{in the sense that for any smooth function $H:\T \to \R$, $\Big(\int_\T H(u) \, X_{t}^{i,N} (u) \, du\Big)_N$ converges in probability to $\int_\T H(u) \, a^i (t, u) du$.} to ${\BF a} (t,\cdot)$ where ${\BF a}(t,\cdot)=(a^1(t,\cdot), \ldots, a^N (t, \cdot))$ and the limit densities $a^i (t, \cdot)$ are the solution of the $n$ coupled system:
\begin{align}\label{eq:3.1}
\partial_t a^i = \tfrac12 \Delta(\tilde g_i({\BF a})) + 2c \nabla(\tilde g_i ({\BF a})),
\quad 1\le i \le n,
\end{align}
where $\Delta=\partial_u^2$, $\nabla=\partial_u$ and  $\tilde g_i$ is an averaged or homogenized function reflecting the microscopic jump rates
defined by
\begin{equation}
\label{tilde g eq}
\tilde g_i({\BF a}):=E_{\nu_{\BF a}} \left[g_i(\a(0))\right], \quad i=1,\ldots,n.
\end{equation}

One can derive the hydrodynamic equation \eqref{eq:3.1} at least heuristically as follows.
Take a test function $G\in {\mathcal D} (\T)$. Then, by Dynkin's formula and Taylor's expansion, we have
\begin{align*}
&\E \Big[ \int_{\T} X_t^{i,N} (u) G(u) du \Big] -\E \Big[\int_\T X_0^{i,N} (u) G(u) du \Big]\\
& = \frac1{2N} \sum_x \int_0^t \E[g_i(\a_s (x))] (\Delta G) (\tfrac{x}N) \, ds + \frac{2c}N \sum_x \int_0^t \E[  g_i(\a_s (x))]  (\nabla G) (\tfrac{x}N) \, ds \; + \; o(1).
\end{align*}
By local ergodicity (cf. \cite{G}),
one can replace any expression like $N^{-1} \sum_x F(\tfrac xN )\E[g_i({\BF \alpha}_s (x)]$ by its ensemble mean $\int_\T \tilde g_i ({\BF X}_s^{i,N} (u)) F(u) du$. Assuming that $\{{\BF X}^{N}\}_N$ converges to some macroscopic profile ${\BF a}$, we get that this profile satisfies a weak formulation of (\ref{eq:3.1}). This argument can be made precise by following the `entropy' method say in Chapter 5 of \cite{KL}.

\begin{rem} If $\gamma>1$ hydrodynamic equations are still in the form (\ref{eq:3.1}) but with $c=0$ while if $\gamma<1$, it is expected that in the time scale $tN^{1+\gamma}$ it is given by a particular weak solution of (\ref{eq:3.1}) without the diffusive term (system of conservation laws).
\end{rem}

\subsubsection{Linear fluctuations}
We discuss the stationary fluctuations from the hydrodynamic limit, again when $\gamma=1$.  For any $N\ge 1$, we assume ${\BF \alpha}_0^N \overset{\text{law}}{=} \nu_{\ab}$ at $t=0$ with some $\ab\in [0,\infty)^n$ and
consider the fluctuation field:
\begin{equation*}
\begin{split}
&\{{\BF{\Y}}_t^N\; ; \; t\ge 0\}  = \{(\Y_t^{i,N})_{i=1}^n\; ; \; t\ge 0\},\\
&\Y_t^{i,N}(du) = \frac1{\sqrt{N}} \sum_{x\in\T_N}\big(\a_{t}^{i,N}(x) -a_{0}^i\big)
\de_{\frac{x}N}(du), \quad u \in \T,
\end{split}
\end{equation*}
where $\delta_z$ denotes the Dirac mass at $z \in \T$. The limit in law $\{\Y_t\; ; \; t\ge 0\}  = \{(\Y_{t}^i)_{i=1}^n\; ; \; t\ge 0\}$ of $\{\Y^N_t \; ; \; t\ge 0\}_{N\ge 1}$ is a solution of the following linear SPDE:
\begin{align}  \label{eq:4.1}
\partial_t \Y = \tfrac12 Q(\ab) \Delta  \Y +2c Q(\ab) \nabla \Y
+ q(\ab) \nabla \dot{\W}_t,\quad t\ge 0, 
\end{align}
where $\dot{\W} = \{\dot{\W}^i \}_{i=1}^n$ is a space-time Gaussian white noise
with $n$-components and covariance
$$
\E[ \dot{\W}_t^i(u)\dot{\W}_s^j(v)] = \de^{ij} \de(t-s) \de(u-v),
$$
and
$Q(\aa)$ and $q(\aa)$ are $n\times n$ matrices such that
\begin{align*}
Q(\aa) & = \big(Q_{ij}(\aa)\big)_{1\le i,j\le n}
= \left( \partial_{a^j} (\tilde g_i(\aa)) \right)_{1\le i,j\le n},\\
q(\aa)&={\rm diag}\left(\big(q_{ii} (\aa)\big)_{1\le i\le n}\right)
= {\rm diag}\left( \left( \sqrt{\tilde g_i(\aa)}\right)_{1\le i\le n}\right) .
\end{align*}

The limit noises in \eqref{eq:4.1} are obtained in the following way.
Let $G\in {\mathcal D}(\T)$ and $M^{N,i}(G)$ be the martingale term of the semi-martingale $\Y^{N,i}(G)$.  Then, by It\^o's formula and Taylor's expansion, we have
\begin{align*}
\frac{d}{dt} \lan M^{i,N}(G)\ran_t
& = \frac1N \sum_{x\in \T_N} g_i ({\BF \alpha}^N_t (x)) \left( \nabla G (\tfrac{x}N)\right)^2 + o(1)\\
& \to  \tilde g_i(\ab) \int_\T(\nabla G (u))^2 du \quad \text{ as } \; N\to\infty,
\end{align*}
since $\a_t^N\overset{\text{law}}{=} \nu_{\ab}$ for all $t\ge 0$.
Moreover, for $i\not= j$, one can see the martingales $M^{N,i}$ and $M^{N,j}$ are uncorrelated and therefore independent in the limit, cf.\ \eqref{eq:4.cross} below.

Heuristically, the drift term in the limit \eqref{eq:4.1} is obtained as follows (cf. \cite{Spohnbook}).
Make the Taylor expansion in the hydrodynamic equation:
\begin{align*}
& a^i \, (=a^i(t,u)) = a_{0}^i + \frac1{\sqrt{N}} \Y^i+ \cdots \\
& \tilde g_i(\aa) = \tilde g_i(\ab)
+ \frac1{\sqrt{N}} \sum_{j=1}^2 \partial_{a^j} (\tilde g_i(\aa))\Big|_{\ab}  \cdot \Y^j+
 \cdots.
\end{align*}
Insert these expansions in the hydrodynamic equation
with a small noise term:
\begin{align*}
\partial_t a^i = \tfrac12 \Delta (\tilde g_i(\aa)) + 2c \nabla (\tilde g_i(\aa))
+ \frac1{\sqrt{N}} \text{(noise)}
\end{align*}
For example, since $\ab$ is a constant,
$$
\partial_t a^i = \frac1{\sqrt{N}} \partial_t \Y^i + \cdots.
$$
Multiplying both sides by $\sqrt{N}$, we obtain the limit SPDE \eqref{eq:4.1}.  We note this type of argument and result is well-known, following say the method in Chapter 11 of \cite{KL}.

\subsubsection{Nonlinear fluctuations and coupled
KPZ-Burgers equations}
\label{subsubsec:NFKPZ}

We now take $\ga=\frac12$, so that $p(1)=\frac12+\frac{c}{\sqrt{N}},
p(-1)=\frac12-\frac{c}{\sqrt{N}}$.  
In other words, $c$ is replaced by $c\sqrt{N}$ so that the hydrodynamic equation
for the $i$th particle looks like
\begin{align}  \label{eq:5.1}
\partial_t a^i = \tfrac12 \Delta(\tilde g_i(\aa)) + 2c\sqrt{N}
 \nabla (\tilde g_i(\aa))
+ \frac1{\sqrt{N}} \text{(noise)}.
\end{align}
To cancel  the diverging factor $\sqrt{N}$ in the second term, we introduce the moving frame
with speed $2c\la N^{\frac32}$ (cf.\ \cite{GJS} p.298) at the microscopic level with suitably
chosen $\la := \la(\ab)$.
$$
\Y_t^{i,N} := \frac1{\sqrt{N}}\sum_x (\a_{t}^{i,N}(x) -a_0^i) 
\, \de_{\tfrac{x}N- \tfrac{2c\la N^{3/2}t}N}(du).
$$
The frame should have a common speed for all species $1\leq i\leq n$ so that a non-trivial limit can be found.

At the macroscopic level, we expand $a^i(t,u)$ along with the center $2c\la\sqrt{N}t$
of the moving frame as
$$
a^i(t,u+2c\la\sqrt{N}t) = a_0^i+ \frac1{\sqrt{N}} \Y_t^i+ \cdots.
$$
More precisely, make the Taylor expansion up to the second order terms:
\begin{align*}
& a^i = a_0^i + \frac1{\sqrt{N}} \Y^i+ \cdots \\
& \partial_t a^i = \frac1{\sqrt{N}} \partial_t \Y^i + 2c\la \sqrt{N} \nabla a^i +\cdots\\
& \hskip 5mm = \frac1{\sqrt{N}} \partial_t \Y^i + 2c\la \nabla \Y^i +\cdots\\
& \tilde g_i(\aa) = \tilde g_i(\ab)
+ \frac1{\sqrt{N}} \sum_{j=1}^n \partial_{a^j} (\tilde g_i(\aa))\Big|_{\ab} \cdot \Y^j
\\
&\hskip 13mm+ \frac1{2N} \sum_{j,k=1}^n \partial_{a^j}\partial_{a^k}
 (\tilde g_i(\aa))\Big|_{\ab} \cdot \Y^j \Y^k + \cdots.
\end{align*}
Noting $\nabla ( \tilde g_i(\ab))=0$, we insert these
expansions to the hydrodynamic equation \eqref{eq:5.1}.  Then, multiplying the both 
sides by $\sqrt{N}$, we obtain:
\begin{align*}
\partial_t \Y^i =&  \tfrac12 
\sum_{j=1}^n \partial_{a^j} (\tilde g_i(\aa))\Big|_{\ab} \cdot \Delta \Y^j \\
& + 2c\sqrt{N} \sum_{j=1}^n \partial_{a^j} (\tilde g_i(\aa))\Big|_{\ab} \cdot \nabla \Y^j
-2c \la \sqrt{N}  \nabla \Y^i \\
& + c \sum_{j,k=1}^n \partial_{a^j}\partial_{a^k}
 (\tilde g_i(\aa))\Big|_{\ab} \cdot  \nabla (\Y^j \Y^k) + \text{noise}.
\end{align*}

The second line is a diverging term.  However, this line vanishes, if one can 
choose $\ab$ such that the following `Frame condition' holds:
\begin{itemize}
\item [(FC)] For any $i\in \{1,\ldots,n\}$,
\begin{equation*} 
\la =\la (\ab) := \partial_{a^i} (\tilde g_i( \aa))\Big|_{\ab}, \quad {\rm and} \quad
\partial_{a^j} (\tilde g_i(\aa))\Big|_{\ab} =0 \; \text{ if } i\not= j.
\end{equation*}
\end{itemize}

In this way, we derive a coupled KPZ-Burgers equation
\begin{align}
\label{eq:1.5-b}
\partial_t \Y^i =  \tfrac12  \sum_{j=1}^n \partial_{a^j} (\tilde g_i(\aa))\Big|_{\ab} \cdot \Delta \Y^j \; +\;  c \sum_{j,k=1}^n \partial_{a^j}\partial_{a^k}  (\tilde g_i(\aa))\Big|_{\ab} \cdot  \nabla (\Y^j \Y^k) \; +\;  \text{noise}.
\end{align}

The noise in \eqref{eq:1.5-b} turns out to be the same as in the linear fluctuations \eqref{eq:4.1}. Part of our results will be to characterize more carefully the Frame condition on the density (FC) on $\ab$ in
Proposition \ref{prop:4.2} below.

\begin{rem}
\rm
One can choose $\la$ as
$$
\la = \la (\ab)+ \frac{\la_1}{\sqrt{N}},
$$
where $\la (\ab)$ is defined in the condition (FC) 
and $\la_1 \in \R$ is arbitrary.  Then, we have an
additional drift term $-2c\la_1 \nabla \Y^i$ in the right hand side of \eqref{eq:1.5-b}. 

We comment that if one considers $\la_1=\la_1(i)$, dependent on $1\leq i\leq n$, then formally drifts of $-2c\la_1(i)\nabla \Y^i$ are introduced.  However, unless $\la_1(i)\equiv \la_1$, fixed for $1\leq i\leq n$, we will not be able to derive rigorously \eqref{eq:1.5-b} with the drifts $-2c\la_1(i)\nabla\Y^i$.  The difficulty stems from having to view $\Y^i$ in several different reference frames at once.  One cannot capture it seems the scaling limits of $\Y^{N,i}_t$, while in the reference frame seen in $\Y^{N,j}_t$ with velocity $2c \Big(\lambda (\ab) + \frac{\lambda_1(j)}{\sqrt{N}}\Big)\sqrt{N}$, when $\lambda_1(i)\neq \lambda_1(j)$ (cf. the later proof of tightness in Subsection \ref{tightness} and identification of limits in Subsection \ref{identification}).  
\end{rem}

\section{Model  }
\label{models}

We now detail the weakly asymmetric $n$-species zero-range process.  The configurations of particles with $n$-species on $\T_N$ are denoted by
${\BF \alpha} = (\a^i)_{i=1}^n \in  \Om^n$, where $\Om = \PZ^{\T_N}$ is the
configuration space of particles of a single species, and $\PZ=\{0,1,2,\ldots\}$.  Here, 
$\a^i=(\a^i(x))_{x\in\T_N}$ and $\a^i(x)\in \PZ$
describes the number of particles of the $i$th species at $x\in \T_N$.  The jump rate 
of the $i$th species particles is given by $g_i({\BF\alpha}(x))$ where ${\BF \alpha}(x)=(\a^j(x))_{j=1}^n \in \PZ^n$.
Weak asymmetry is introduced as follows. Once a jump is going to happen, 
the probability of jumping to the right is $p(1)= \frac12+ \frac{c}{N^\ga}$ 
and that to the left is $p(-1)= \frac12- \frac{c}{N^\ga}$, for $c\in \R$ and $\gamma>0$.
The parameter $\gamma$ controls the strength of the weak-asymmetry. It should be taken as $\ga=1$ to show the hydrodynamic limit (\ref{eq:3.1}) and the
linear fluctuation result (\ref{eq:4.1}) in the Euler time scale, while we will fix it as $\ga = \frac12$ to derive the second-order KPZ-Burgers fluctuations (\ref{eq:1.5-b}) in the diffusive time-scale.

Hence we introduce a diffusive time change $t\mapsto N^2 t$ for the microscopic 
system.  Then, the action of the generator $L_N$ of the system on functions $f:\Om^n \to \R$ is defined by
\begin{equation} 
 \label{eq:2.LN}
 \begin{split}
L_N f({\BF \alpha})
&=  N^2 \sum_{x\in \T_N} p(1) \sum_{i=1}^n g_i({\BF \alpha}(x))
 \left\{ f({\BF \alpha}^{x,x+1;i})-f({\BF \alpha})\right\} \\
& + N^2 \sum_{x\in\T_N} p(-1) \sum_{i=1}^n g_i({\BF \alpha} (x))
\left\{ f({\BF \alpha}^{x,x-1;i})-f({\BF \alpha})\right\}, 
\end{split}
\end{equation}
where ${\BF \alpha}^{x,y;i}:=((\alpha^{x,y;i})^j)_{j=1}^n\in \Om^n$ denotes the configuration of particles
after one particle of $i$th species in the configuration ${\BF \alpha}=(\a^i)_{i=1}^n$ jumps from $x$ to $y$.
That is, when $\a^i(x) \ge 1$, we define $({\alpha}^{x,y;i})^i(x) = \a^i(x)-1$, $({\alpha}^{x,y:i})^i(y) = \a^i(y)+1$
and  $({\alpha}^{x,y;i})^j =\a^j$ for $j\not= i$.

For $\Bk=(k^1,\ldots,k^n)\in \PZ^n$ and $1\le j \le n$ we denote $\Bk_{j, \pm}$ the configuration{\footnote{If $k^j =0$, $\Bk_{j, -}$ is not defined.} } 
$$\Bk_{j,\pm} =(k^1,\ldots,k^{j-1},k^j\pm1,k^{j+1},\ldots,k^n).$$
It is natural to assume that the jump rates $g_i$ satisfy the non-degeneracy condition:
\begin{itemize}  
\item [(ND)]
$g_i(\Bk)=0 \Leftrightarrow k^i=0\quad \text{ and }\quad
g_i^* := \inf_{\substack{|\Bk|\ge 0\\ k^i>0}} \; g_i(\Bk)>0,
\ \ {\rm for} \ \ i=1,\ldots,n$.
\end{itemize}
We will also impose the linear growth condition:
\begin{itemize}
\item[(LG)]
$\max_{1\le i, j \le n} \sup_{\Bk\in \PZ^n} |g_i(\Bk_{j,+})-g_i(\Bk)| < \infty$.
\end{itemize}

Such a Lipschitz assumption is not necessary to define the process on $\T_N$ since it is then a countable state Markov chain but useful to construct the process on $\Z$.  See \cite{Andjel} for the single species construction on $\Z$ say; the multi-species zero-range construction on $\Z$ straightforwardly follows the same argument. We will denote the generator of the infinite-volume dynamics by $L$. In the current paper, (LG) is used in particular to verify a spectral gap estimate in Lemma \ref{spec_gap_lemma}. 
It also implies that $g_i$ is at most linear in its growth, which is useful in defining some expectations.

\subsection{Invariant measures and compatibility condition}
\label{invariant_section}

To identify product invariant measures, we will
assume the jump rate  $ (g_i )_{i=1}^n$ satisfies the following
compatibility condition:
\begin{itemize}
\item[(INV)]
$\forall i\ne j,\quad \forall \Bk \in \Z_+^n \text{\; such that $k^i>0$ and $k^j>0$, \; } \quad \cfrac{g_i(\Bk)}{g_i(\Bk_{j,-})}
= \cfrac{g_j(\Bk)}{g_j(\Bk_{i,-})}.$
\end{itemize}
Then we can define the following product 
\begin{equation*}
\Bg!(\Bk)= \prod_{\ell=1}^{|\Bk|} g_{i(\ell)}(\Bk_\ell),
\end{equation*}
along any increasing path $\Bk_0= \boldsymbol{0} \to \cdots \to \Bk_\ell \to \cdots \to
\Bk_{|\Bk|}=\Bk$ connecting $\boldsymbol{0}$ and $\Bk$ in $\PZ^n$ such that $|\Bk_\ell|=\ell$ and $(\Bk_\ell-\Bk_{\ell-1})^j = \de_{i(\ell),j}$ for $1\le j \le n$, $1\le \ell\le |\Bk|$. Note that, because of the condition (INV), $\Bg!(\Bk)$ does not depend on the choice of the increasing path $\{\Bk_\ell\}$, and so is well-defined. See \cite{GS03} or \cite{DSZ}, p.\ 798 (when $n=2$) for this derivation on the torus, which extends to $n\geq 2$.

Under the condition (INV), the invariant measures on $\Omega^n =(\Z_+^n)^{\T_N}$ are product measures over $\T_N$ (but the corresponding marginals are not product) with one site marginal given by
\begin{equation} 
\label{eq:invariant}
\bar\nu_{\boldsymbol\fa}(\Bk) = \frac1{Z_{\boldsymbol\fa}} 
\frac{\boldsymbol\fa^{\Bk}}{\Bg!(\Bk)},
\end{equation}
where $\boldsymbol\fa = (\fa^1,\ldots,\fa^n) \in (0,\infty)^n$ is a parameter, sometimes called the fugacity, and
\begin{equation*}
\boldsymbol\fa^\Bk= [\fa^1]^{k^1}\cdots[\fa^n]^{k^n}, \quad Z_{\boldsymbol\fa}= \sum_{\Bk\in \PZ^n} \frac{\boldsymbol\fa^{\Bk}}{\Bg!(\Bk)}, 
\end{equation*}
In other words, $\bar\nu_{\boldsymbol\fa}$ is a probability measure on $\Om^n$ such that $
\bar\nu_{\boldsymbol\fa}({\BF \alpha} (x)=\Bk)$ is given by the formula \eqref{eq:invariant} for each $x\in \T_N$ and {\footnote{With some abuse of notation we omit the dependance on $N$ in $\bar\nu_{\boldsymbol\fa}$ and use the same notation for the marginal and the product measure. We write ${\bar \nu}_{\boldsymbol\fa}$ also for those on the infinite-volume configuration space $({\mathbb Z}_+^{\mathbb Z})^n$ instead of $\Omega^n \equiv ({\mathbb Z}_+^{\mathbb T_N})^n$.}} 
$\{{\BF \alpha}(x)\}_{x\in\T_N}$ are independent.

Extension of the invariance of these product measures for the infinite-volume process on $\Z$ is immediate as functions depending only on a finite number of sites are dense in the domain of the infinite-volume generator $L$, defined as $L_N$ with $\mathbb T_N$ replaced by $\mathbb Z$.

The measure $\bar\nu_{\boldsymbol\fa}$ is well-defined for $\boldsymbol\fa$
in the proper domain of $Z_{\boldsymbol\fa}$:
$$
{{\rm{Dom}}}_Z := \{\boldsymbol\fa \in (0,\infty)^n\; ;\;  Z_{\boldsymbol\fa}<\infty\}.
$$
To ensure ${{\rm{Dom}}}_Z$ is non-trivial in the sense that it contains a
neighborhood of ${\BF 0}=(0,\ldots,0)$, we assume the following:
\begin{itemize}
\item[(ORI)]
$\fa_* := \liminf_{|\Bk|\to\infty} \Bg!(\Bk)^{\frac1{|\Bk|}} >0$.
\end{itemize}
We now change the parametrization of the invariant measures in terms of density $\BF a$  instead of the fugacity $\boldsymbol\fa$: $\boldsymbol\fa:=\Phi (\BF a)$ is chosen so that
$$\nu_{\BF a} := \bar\nu_{\boldsymbol\fa}$$
has mean $\BF a
= (a^1,\ldots, a^n)$ and $a^i \ge 0$ for $1\leq i\leq n$.  That is,
\begin{equation}  \label{eq:aveg-a}
a^i \equiv a^i (\boldsymbol\fa) := E_{\bar \nu_{\boldsymbol\fa}}[\a^i(0)], \quad i=1,\ldots,n.
\end{equation}
To define $\Phi$ more explicitly let us denote the map $R: \boldsymbol\fa \to \BF a$, taking fugacity to its associated density, with the proper domain
$$
{{\rm{Dom}}}_R := \{\boldsymbol\fa \in {{\rm{Dom}}}_Z\; ; \; a^i(\boldsymbol\fa)<\infty, \, i=1,\ldots, n\}.
$$
By the relation (cf. \eqref{tilde g eq}),
\begin{equation}  \label{eq:aveg}
\tilde g_i(R(\boldsymbol\fa)) = E_{\bar\nu_{\boldsymbol\fa}}[g_i({\BF \alpha}(0))] = \varphi^i,
\quad i=1,\ldots,n,
\end{equation}
we see that $\boldsymbol\fa$ and $\BF a$ are in $1:1$ correspondence on their proper domains.
The function $\Phi: \boldsymbol\ab \to \boldsymbol\fa$ appearing above is the inverse of $R$. 

The fugacity $\boldsymbol\fa$ is sometimes written in terms of  the chemical potential $$\boldsymbol\lambda=(\lambda^1,\ldots,\lambda^n):={\BF \lambda} (\BF \fa)$$ defined by
$\varphi^i=e^{\lambda^i}$ with $\lambda^i \in {\mathbb R}$ for $1\leq i\leq n$. With some abuse of notation we write also ${\BF \lambda} ({\BF a})$ instead of ${\BF \lambda} (\Phi({\BF a}))$.

Let also $\Gamma(\BF a)={\rm cov}(\nu_{\BF a})$ be the covariance matrix of ${\BF \alpha}(0)$ under $\nu_{\BF a}$. 
When the density $\BF a$ belongs to $(0,\infty)^n$, the measure $\nu_{\BF a}$ gives positive measure to all $n$-tuples $(k^1,\ldots, k^n)\in \PZ^n$.  Then, the variance of any linear combination of $(\alpha^i(0))_{i=1}^n$ cannot vanish.  Hence, the matrix $\Gamma({\BF a})$ is positive-definite, symmetric and therefore invertible.

\begin{lem}  
\label{lem:2.1}
For any $\boldsymbol\fa \in \ocirc{{\rm{Dom}}}_R$, that is the interior of the set ${\rm{Dom}}_R$, we have
$$
(\nabla R )( \boldsymbol\fa) = \Gamma(R(\boldsymbol\fa))\; {\rm diag}\left( \left( \frac{1}{\tilde g_i(R(\boldsymbol\fa))}\right)_{1\le i \le n} \right).
$$
Accordingly, for any $\BF a \in \ocirc{{\rm{Dom}}}_\Phi$, we have
$$
(\nabla \Phi) ({\BF a}) = {\rm diag} \left( \Big(\tilde g_i( {\BF a})\Big)_{1\le i\le n}  \right) \; \Gamma({\BF a})^{-1}$$
and
$$(\nabla \lambda) ({\BF a}) = \Gamma({\BF a})^{-1}.$$

\end{lem}

\begin{proof}
This is shown in \cite{DSZ}, p.799, or \cite{GS03}, (2.20) noting 
$d\lambda_i/d\varphi_i = 1/\varphi_i$ since $\varphi_i=\log\tilde g_i(\BF a)$.

For the reader's convenience, we give a derivation of the last statement as it is short.  
See also \eqref{eq:7.3} below.  Consider the moment generating function
$M({\BF \lambda}) = E_{\nu_{\BF a}}[e^{{\BF \lambda}\cdot({\BF \alpha}(0)-\BF a)}]$,
where the chemical potential ${\BF \lambda}={\BF \lambda}(\z)$ is associated to density $\z$.
Then,
\begin{align*}
\z-{\BF a} &= \nabla_{\BF \lambda} M({\BF \lambda}(\z))/M({\BF \lambda}(\z))\\
&=\left( E_{\nu_{\BF a}}\Big [(\alpha^j (0)-a^j )e^{{\BF \lambda}({\BF z})\cdot ({\BF \alpha} (0)-{\BF a})}\Big]/M({\BF \lambda} ({\BF z} ))\right)_{j=1}^n.
\end{align*} 
Then, for any $j \in \{1, \ldots, n\}$, we have that
\begin{equation*}
\begin{split}
&1= \sum_{r=1}^n \partial_{z^j} \lambda^r({\BF z}) \Big\{E_{\nu_{\z}}[(\alpha^j(0)-a^j)(\alpha^r(0)-\a^r)]- E_{\nu_\z}[(\alpha^j(0)-\a^j)]]E_{\nu_\z}[(\alpha^r(0)-a^r)]\Big\}\\
&= \sum_{r=1}^n \partial_{z^j}\lambda^r(\z) \Gamma(\z)_{jr}=\Big[ \Gamma (\z) \, {\nabla} {\BF \lambda} (\z) \Big]_{jj}.
\end{split}
\end{equation*}
When $k\neq j$, we have
$0 = \sum_{r=1}^n \partial_{z^k} \lambda^r (\z) \,  \Gamma(\z)_{jr}= \Big[ \Gamma (\z) \, \nabla {\BF \lambda} (\z) \Big]_{jk}$.

\end{proof}

\subsection{Markov process}
\label{Markov}
Under the previous conditions on the rates $(g_i)_{i=1}^n$ and jump probability $p$, a process $\{ {\BF \alpha}_t^N\; ; \; t \ge 0\} =\{ (\a_t^{i,N})_{i=1}^n\; ; \; t \ge 0\}$ may be constructed on $\T_N$ or $\Z$ with a family of invariant product measures $\{\nu_\ab: \ab\in {{\rm{Dom}}}_\Phi\}$.  Moreover,
with respect to a fixed $\nu_\ab$, the process can be associated to a Markov semigroup on $L^2(\nu_\ab)$ and Markov generator $L_N$ with a core of $L^2(\nu_\ab)$ functions.   The adjoint $L_N^*$ can be seen to be the generator with respect to reversed jump probability $p^*(\cdot) = p(-\cdot)$.  Moreover, the measure $\nu_\ab$ is invariant with respect to the adjoint process, and is reversible when $p = p^* = s$.   See \cite{S_extremal} for more on these details in the single species case, with the multi-species extension being straightforward.  

In the sequel, a local function is one which depends only on a finite number of occupation variables $({\BF \alpha} (x))_{x \in \Z}$. Define, for local $f$, the function, 
$$\tilde{f}({\BF a}) = E_{\nu_{\BF a}}[f],$$
when the expectation makes sense.

\subsection{Spectral gap}
\label{specgap_section}

The mixing properties of the system will play a role in the analysis.  For $\ell \ge 1$, consider the localized, ergodic process, corresponding to the symmetric nearest-neighbor jump probability $s(\pm 1)=(p(\pm 1)+p(\mp 1))=1/2$ and $s(x)=0$ for $x\neq \pm 1$, on the interval $\Lambda_\ell = \{x\in \Z: |x|\leq \ell\}$ with a fixed number of $\k =(k^1,\ldots, k^n)$ particles of each type, and generator
$$S_{\k,\ell} f(\BF \alpha) \ = \ \frac{1}{2}\sum_{\stackrel{|x-y|=1}{x,y\in \Lambda_\ell}}\sum_{i=1}^ng_i(\BF \alpha (x))\big\{f({\BF \alpha}^{x,y;i}) - f({\BF \alpha})\big\}.$$
For this Markov process, the canonical measure $\nu_{\k,\ell} = \nu_\ab(\cdot| \sum_{x\in \Lambda_\ell}{\BF \alpha} (x) = \k)$ is reversible and invariant.  Let $l_{\k,\ell}$ be the spectral gap, that is, the second smallest eigenvalue of $-S_{\k,\ell}$ (with $0$ being smallest).  Denote $W(\k,\ell) = l^{-1}_{\k,\ell}$ and recall the Poincar\'e's inequality
$${\rm Var}_{\nu_{\k,\ell}}(f) \ \leq W(\k,\ell) D_{\k,\ell}(f)$$
where $D_{\k,\ell}$ is the canonical Dirichlet form
$$D_{\k,\ell}(f) \ = \ \frac{1}{4}\sum_{\stackrel{|x-y|=1}{x,y\in \Lambda_\ell}} \sum_{i=1}^n E_{\nu_{\k,\ell}}\big[g_i({\BF \alpha}(x))\big\{f({\BF \alpha}^{x,y;i})-f({\BF \alpha})\big\}^2\big].
$$

We will suppose the following condition which guarantees sufficient mixing for our purposes:
\begin{itemize}
\item[(SG)] There is a constant $C:=C(\ab,n)$ such that 
$\sup_{\ell\geq 2}E_{\nu_\ab} \left[ W\left(\sum_{x\in \Lambda_\ell}{\BF \alpha}(x), \ell\right)^2\right]  \ \leq \ C\ell^4$.
\end{itemize}
The condition (SG) is one on the rates $(g_i)_{i=1}^n$, useful in the proof the Boltzmann-Gibbs principle in Proposition \ref{gbg_L2}. There is a large class of rates for which this condition holds. 

Consider the following lower bound criterion:
\begin{itemize}
\item[(LB)] Suppose there exists ${\BF m}_0\in \N^n$ and $\epsilon_0>0$ so that
$\inf_{1\leq i\leq n}\inf_{{\BF m}\in \PZ^n} \{g_i({\BF m} + {\BF m}_0) - g_i({\BF m}) \} \geq \epsilon_0$.
\end{itemize}
Under (LB), the domains ${\rm{Dom}}_R = {\rm{Dom}}_\Phi=[0,\infty)^n$.

\begin{lem}
\label{spec_gap_lemma}
Under assumptions (LG) and (LB), we have $W(\k,\ell)\leq C({\BF m}_0, \epsilon_0,  n)\ell^2$ uniformly in $\k\in \PZ^n$.
\end{lem}
\begin{proof}
Write the variance
\begin{eqnarray}
\label{first line}
{\rm Var}_{\nu_{\k,\ell}}(f) &=& E_{\nu_{\k,\ell}}[ E_{\nu_{\k,\ell}}[f^2({\BF \alpha})|\alpha^2(\cdot),\ldots, \alpha^n(\cdot)]] - E_{\nu_{\k,\ell}}[ (E_{\nu_{\k,\ell}}[f({\BF \alpha})|\alpha^2(\cdot),\ldots, \alpha^n(\cdot)])^2]\nonumber\\
&&\ \ + {\rm Var}_{\nu_{\k,\ell}}(h(\alpha^2,\ldots, \alpha^n)),
\end{eqnarray}
where $h(\alpha^2,\ldots, \alpha^n) = E_{\nu_{\k,\ell}}[f({\BF \alpha})|\alpha^2(\cdot),\ldots, \alpha^n(\cdot)]$.

With respect to the first line of the above display, by fixing $\alpha^2(x)=a^2(x),\ldots, \alpha^n(x)=a^n(x)$ for $x\in \Lambda_{\ell}$, we may consider the process governing the first type particles $\alpha^1$ with space inhomogeneous rate function
$g(\alpha^1;\cdot):=g^1(\alpha^1, a^2(\cdot),\ldots, a^n(\cdot))$ and symmetric nearest-neighbor jump probability.  The product measure, indexed over $\Lambda_\ell$, with marginal probability at site $x$ with $k$ particles, proportional to 
${\varphi_1^k}/{g(1)\cdots g(k)}$,
 is invariant.  These measures are $\nu_{\k,\ab}$ conditioned on $\alpha^i(x) = a^i(x)$ for $x\in \Lambda_\ell$ and $2\leq i\leq n$.   Because of (LB), we see that
$$\inf_{k\geq 1, x\in \Lambda_{\ell}}\Big\{ g(k+k^1_0, a^2(x), \ldots a^n (x) ) - g(k, a^2(x), \ldots, a^n (x) )\Big\} >\epsilon_0.$$

For such inhomogeneous processes, localized on $\Lambda_\ell$ with $k^1$ particles, it is known
that the associated inverse spectral gap quantity satisfies $W(k^1,\ell)\leq C\ell^2$ with constant $C=C({\BF m}_0, \epsilon_0, n)$.  See Theorem 2.3 and preceding discussion in \cite{Jankowski} for a statement and proof which extends the argument for the `homogeneous' zero-range process in \cite{LSV}; see also \cite{CP} for an argument under an assumption more restrictive than (LB).  Hence, we may bound the first line in (\ref{first line}) by $C\ell^2 D_{\k,\ell}(f)$.

The second line in \eqref{first line} is now similarly bounded by an induction argument.  Finally, we can recover that $W(\k, \ell) \leq nC\ell^2$ in terms of a constant $C=C({\BF m}_0, \epsilon_0, n)$.
\end{proof}

\section{Results}
\label{results}

We first define spaces needed to state the main theorems. We recall that ${\mathcal D} (\T)$ is the space of smooth functions. For a fixed $0<T<\infty$ and $n\geq 1$, denote by $D([0,T], {\mathcal D}^{\prime} (\T)^n)$ and $C([0,T], {\mathcal D}^{\prime} (\T)^n)$ the function spaces of $c\grave{a}dl\grave{a}g$ and continuous maps respectively from $[0,T]$ to ${\mathcal D}^{\prime} (\T)^n$.  The bracket $\langle \cdot,\cdot\rangle$ denotes  the dual pairing between ${\mathcal D}^{\prime} (\T)^n$ and ${\mathcal D} (\T)^n$, but also between other pairs of spaces when the context is clear.  We equip these spaces with the uniform weak topology: a sequence $\{\mc Z_\cdot^N\}_{N\ge 1}$ converges to a path $\mc Z_\cdot$ if for all $H \in \mathcal D (\bb T)^n$, we have
\begin{equation*}
\lim_{N \to \infty} \; \sup_{0\le t \le T} \; \Big\| \mc Z^N_t (H)  -  \mc Z_t (H)  \Big\|=0,
\end{equation*}
where $\|\cdot\|$ is the usual Euclidean norm on $\R^n$.

Throughout the article, the initial configuration ${\BF\alpha}^N_0$ will be distributed according to $\nu_{{\BF a_0}}$ with respect to a density ${\BF a_0}\in [0,\infty)^n$. 

We now discuss two types of fluctuation results.  First, we describe the so-called `linear' fluctuations, before going to the `nonlinear' ones, leading to coupled KPZ-Burgers equations.

\subsection{Linear fluctuations}

Let now $\bar\Y^N_\cdot:=(\bar\Y^{i,N}_\cdot)_{i=1}^n \in D([0,T], {\mathcal D}^{\prime} (\T)^n)$ be the density fluctuation field, acting on functions $H=(H^i)_{i=1}^n \in {\mathcal D} (\T)^n$, given by
$$\bar\Y^N_t(H) \ = \ \left( \cfrac{1}{\sqrt{N}}\sum_{x\in \T_N} H^i\left(\tfrac{x}{N}\right)\big(\a^{i,N}_{t}(x) - a_0^i\big) \right)_{i=1}^n.$$

By the central limit theorem, for each fixed $t\geq 0$, $\bar\Y^{N}_t$ converges in distribution to $\dot W_0= (\dot W_0^i)_{i=1}^n$, the Gaussian distribution taking values in ${\mathcal D}^{\prime} (\T)^n$ corresponding to covariances, for $1\leq i,j\leq n$,
\begin{equation}
\label{standard_covariance}
{\rm Cov}\big( \dot W^i_0(G), \dot W^j_0(H)\big) \ = \Gamma_{ij} ({\BF a}_0) \, \int_\T G(u)H(u) du
\end{equation}
 where we recall that $\Gamma_{ij} ({\BF a}_0) = E_{\nu_{\BF a_0}}\big[(\alpha^i(0)-a_0^i)(\alpha^j(0)-a^j_0)\big]$. 
%
Recall also we denote the standard $\R^n$-space-time white noise $\dot\W_t$ as the ${\mathcal D}^\prime (\T)^n$-valued Gaussian process with covariance function given by
\begin{equation}  \label{eq:3.2-b}
{\rm Cov}\left( \dot\W^i_t(G), \dot\W^j_s(H)\right) \ = \ \de^{ij} \delta(t-s)
\int_\T G(u)H(u)du
\end{equation}
where $\de^{ij}$ is the standard Kronecker symbol, $\delta$ the Dirac mass at $0$ and $H,G \in \mathcal D (\T)^n$ are test functions.
\begin{theorem}
\label{drift_thm}
Suppose $\gamma = 1$.  Starting from initial measure $\nu_{\BF a_0}$, the sequence $\bar\Y^{N}_\cdot$, as $N\uparrow\infty$, converges weakly in the uniform topology on $D([0,T], {\mathcal D}^{\prime} (\T)^n)$ to the unique process $\Y_\cdot$, solving
\begin{equation}
\label{gen_OU_drift}
\partial_t \Y_t \ = \ \tfrac{1}{2}Q({\BF a_0})\Delta\Y_t + 2c Q({\BF a_0})\nabla \Y_t + q({\BF a_0})\nabla \dot\W_t,
\end{equation}
where, initially $\Y_0$ is distributed as $\dot{W}_0$, and
\begin{align*}
Q({\BF a}) & = \big(Q_{ij}({\BF a})\big)_{1\le i,j\le n}
= \left( \partial_{a^j} (\tilde g_i({\BF a})) \right)_{1\le i,j\le n} \ \ {\rm and }\\
q({\BF a}) &= {\rm diag}\left(\big(q_{ii}({\BF a})\big)_{1\le i\le n}\right)
= {\rm diag}\left( \big(\sqrt{\tilde g_i({\BF a})}\big)_{1\le i\le n}\right).
\end{align*}
\end{theorem}

The stochastic PDE \eqref{gen_OU_drift} is a cross-diffusion system so that its well-posedness
is non-trivial.  We show the well-posed and the invariance of the distribution of $\dot{\W}_0$, though the latter 
is obvious by the limiting procedure.  
Recall 
$$Q({\BF a_0}) \equiv \nabla\Phi({\BF a_0})= {\rm diag}((\tilde g_i ({\BF a}_0))_{1\le i\le n} )\Ga({\BF a}_0)^{-1}$$ from 
Lemma \ref{lem:2.1}. This plays a role of Einstein's relation as it will be clear from Lemma \ref{lem:3.2}. To simplify notation we note ${\rm diag}(\tilde g_i)= {\rm diag}((\tilde g_i ({\BF a}_0))_{1\le i\le n} )$ and $\Gamma=\Ga({\BF a}_0)$.  Then, one can rewrite 
\eqref{gen_OU_drift} into the stochastic PDE for $\mathcal{Z}_t = \Ga^{-\frac12} \Y_t$:
\begin{equation}
\label{eq:3.4}
\partial_t \mathcal{Z}_t \ = \ \tfrac{1}{2}A\Delta\mathcal{Z}_t + 2c A\nabla \mathcal{Z}_t + 
\Ga^{-\frac12} {\rm diag}(\sqrt{\tilde g_i}) \nabla \dot\W_t,
\end{equation}
where $A= \Ga^{-\frac12}{\rm diag}(\tilde g_i)\Ga^{-\frac12}$ is a symmetric matrix.  
However,
$\Ga^{-\frac12}{\rm diag}(\sqrt{\tilde g_i}) \dot\W_t \overset{\text{law}}{=} A^{\frac12}\dot\W_t$,
since the covariance coincides:  for $H\in {\mathcal D} (\T)^n$,
\begin{align*}
\E\Big[\Big\langle \Ga^{-\frac12}{\rm diag}(\sqrt{\tilde g_i})  \W_t\, ,\,   H \Big\rangle^2\Big]
& = t \left\|{\rm diag}(\sqrt{\tilde g_i})\Ga^{-\frac12}H\right\|_{L^2(\T)^n}^2 \\
& = t (AH, H)_{L^2(\T)^n} \\
& = \E[\langle {A}^{1/2} \W_t, H\rangle^2].
\end{align*}
Thus, $\mathcal{Z}_t$ satisfies the following stochastic PDE in law:
\begin{equation}
\label{eq:3.5}
\partial_t \mathcal{Z}_t \ = \ \tfrac{1}{2}A\Delta\mathcal{Z}_t + 2c A\nabla \mathcal{Z}_t + 
A^{\frac12} \nabla \dot\W_t.
\end{equation}
Since $A$ is symmetric, it is diagonalizable as $A e_i = \mu^i e_i$ with an orthonormal basis
$\{e_i\}_{i=1}^n$ of $\R^n$ and eigenvalues $(\mu^i)_{i=1}^n \in \R^n$.  Set $z_t^i(x) :=
\mathcal{Z}_t(x)\cdot e_i$, then we have $n$-independent $\R$-valued stochastic PDEs
\begin{equation}
\label{eq:3.6}
\partial_t z_t^i \ = \ \tfrac{1}{2}\mu^i \Delta z_t^i + 2c \mu^i \nabla z_t^i + 
\sqrt{\mu^i} \, \nabla \dot w_t^i,
\end{equation}
where $\{\dot w_t^i \equiv \dot w_t^i(x) := \dot \W_t(x)\cdot e_i\}_{i=1}^n$
are $n$ independent space-time white noises. Moreover, since ${\mc Y}_0$ is distributed as $\dot{W}_0$, ${\mc Z}_0= \Gamma^{-1/2} {\mc Y}_0$ is distributed like a standard $\R^n$-valued space white-noise, say $\dot\W_0$, or equivalently the $\{z_0^i\}_{i=1}^n$ are $n$ independent standard $\R$-valued space white-noises.  

It is now standard to show the next lemma.  For completeness, we give the proof below.
\begin{lem} 
 \label{lem:3.2}
The process $\{z_t\; ; \; t\ge 0\}$ has the distribution of the spatial white noise $\dot \W_0$ as its invariant measure. Thus, the invariant measure of $\{\Y_t \; ; \; t\ge 0\}$ is the distribution of $\Ga^{\frac12}\dot \W_0$,  which is nothing but $\dot{W}_0$ whose covariance is given by \eqref{standard_covariance}.
\end{lem}

\begin{proof}
First assume $c=0$ and denote the solution of \eqref{eq:3.6} with $c=0$ by $z_t^{0,i}$.
Let $\mathcal{A} = \tfrac{1}{2}\mu^i \Delta$ and $\mathcal{C} = \sqrt{\mu^i} \nabla$
be the operators acting on ${\mathcal D}^{\prime} (\T)$.  For $\fa\in {\mathcal D} (\T)$,
we set $\fa_t = e^{t\mathcal{A}}\fa$.  Let $T>0$. Then, we see that the process $\{M_t\; ; \; t \in [0,T]\}$ defined by
$$
M_t := \exp\left\{ \sqrt{-1} \lan z_t^{0,i},\fa_{T-t}\ran
+ \frac12 \int_0^t \lan \mathcal{C}\mathcal{C}^* \fa_{T-s},\fa_{T-s}\ran ds\right\},
\quad 0\le t \le T,
$$
is a martingale.  Indeed, by applying It\^o's formula, we have
$d \lan z_t^{0,i}, \fa_{T-t}\ran = \lan d w_t^i, \mathcal{C}^* \fa_{T-t}\ran$
and this implies $dM_t = \sqrt{-1}M_t \lan d w_t^i, \mathcal{C}^* \fa_{T-t}\ran$,
cf. Proposition 6.1 in \cite{F91-1}.  Thus, we obtain $\E[M_t]=\E[M_0]$ and, in particular, taking $T=t$, we have
\begin{align*}
\E\left[ e^{\sqrt{-1} \lan z_t^{0,i}, \fa\ran} \right]
e^{\frac12 \int_0^t \lan \mathcal{C}\mathcal{C}^* \fa_{s}, \fa_{s}\ran ds}
= E\left[ e^{\sqrt{-1} \lan z_0^{0,i}, \fa_t\ran} \right] = 
e^{-\frac12 \|\fa_t\|_{L^2(\T)}^2},
\end{align*}
if $z_0^{0,i}$ is distributed as $\dot \W_0^i$.  However, since $\mathcal{C}\mathcal{C}^*
= -\mu^i\De = -2 \mathcal{A}$,
\begin{align*}
\int_0^t \lan \mathcal{C}\mathcal{C}^* \fa_{s}, \fa_{s}\ran ds
= - \int_0^t\lan 2\mathcal{A} e^{2s\mathcal{A}}\fa,\fa\ran ds
= -\|\fa_t\|_{L^2(\T)}^2 + \|\fa\|_{L^2(\T)}^2.
\end{align*}
This shows that
$$
\E\left[ e^{\sqrt{-1} \lan z_t^{0,i}, \fa\ran} \right]
= e^{-\frac12 \|\fa\|_{L^2(\T)}^2}.
$$
In particular, the distribution of $\dot \W_0^i$ is invariant for $\{z_t^{0,i}\; ; \; t \ge 0\}$.
When $c\not=0$, since $z_t^i(x) = z_t^{0,i}(x+2c \mu^i t)$ (mod $1$ in $x$),
we also have the invariance of $\{z_t^i\; ; \; t\ge 0\}$.
\end{proof}

\subsection{Nonlinear fluctuations}

To probe second-order effects, we will like to absorb the drift in \eqref{gen_OU_drift}, and observe the fluctuation field moving with a common velocity $\lambda$.  Such an idea will make sense when all the drifts in \eqref{gen_OU_drift} are the same, that is Frame condition (FC)
holds.

\begin{prop}  \label{prop:4.2}  
The Frame condition (FC)
is equivalent for $\ab$ to be such that
 $$\V_{ij} ({\bf a}_0)=0 \ {\rm \  for \  } i\not= j \quad {\rm  and} \quad  \frac{\tilde g_i ({\BF a}_0) }{\V_{ii} ({\BF a}_0)} =\lambda \ {\rm is \ constant \ in \ }i,$$
 where we recall $\V({\BF a}_0)$ is the covariance matrix $\V ({\BF a_0})= {\rm cov}(\bar\nu_\fa)\big|_{{\BF a}=\ab}$.
\end{prop}

\begin{proof}
This is immediate, since (FC) is equivalent to ${\nabla}\Phi(\ab) = \la I$
and this is further rewritten as 
${\rm diag} \Big((\tilde g_i (\ab))_{1\le i\le n}\Big)\Gamma^{-1} = \lambda I$
by Lemma \ref{lem:2.1}.
\end{proof}

We remark, as a consequence of Proposition \ref{prop:4.2}, under the Frame condition (FC), that the matrix $\Gamma= {\rm cov}(\bar\nu_\fa)\big|_{a={\BF a_0}}$ is diagonal.

We now consider the fluctuation field moving with a `characteristic' velocity $\lambda = \lambda({\BF a_0})$ according to the Frame condition (FC).
Define $\Y^{N}_\cdot\in D([0,T],{\mathcal D}^{\prime} (\T)^n)$, in terms of its action on $H=(H^i)_{i=1}^n \in {\mathcal D} (\T)^n$, as
$$\Y^{N}_t(H) \ = \left( \ \frac{1}{\sqrt{N}}\sum_{x\in \T_N}  H^i\left(\tfrac{x}{N} - \tfrac{2c \lambda tN^{2}}{N^{\gamma+1}}\right)\big(\a^{i,N}_{t}(x) - a_0^i \big) \ \right)_{i=1}^n.$$
In the following, we suppose that $\gamma=1/2$ in order to get a non-trivial limit of $\Y^{N}_\cdot$.  In this case, the asymmetry is strong enough so that a `quadratic' term is recovered in the limit. Formally, under Frame condition (FC),
the limits of $\{\Y^{N}_\cdot\}_{N\ge 1}$ satisfy a type of (ill-posed) KPZ-Burgers equation:  For $i=1,\ldots, n$,
\begin{align}
\label{gen_OU_sec}
\partial_t \Y^i_t &= \ \frac12 \partial_{a^i}\tilde{g}_i(\ab)\,  \Delta \Y^i_t 
 + c\sum_{j,k=1}^n \partial_{a^j}\partial_{a^k}\tilde{g}_i(\ab)\, \nabla (\Y^j_t\Y^k_t) + \sqrt{\tilde{g}_i(\ab)}\, \nabla\dot\W^i_t.
\end{align}

\begin{rem}\rm
Nevertheless, such an equation, as we show in Section \ref{trilinear}, satisfies a `trilinear' condition, and therefore, as an SPDE, can be understood in terms of the theory developed in \cite{FH17}, and shown to have $\R^n$-valued spatial white noise multiplied by $\Gamma^{1/2}$ as an invariant measure.
\end{rem}

However, from the view of microscopic dynamics, to derive this equation as a scaling limit, as in \cite{GJ} and \cite{GJS}, we have to interpret it in the framework of  `$L^2$-energy' martingale formulation of \eqref{gen_OU_sec}.  
Let $\iota: \T \to [0,\infty)$ be given by $\iota(z) = (1/2) 1_{[-1,1]}(z)$ and, for $\varepsilon>0$, let
$\iota_\varepsilon(z) = (2\varepsilon)^{-1} 1_{|z|\leq \varepsilon}$.
Let $G_\varepsilon: \T \to [0,\infty)$ be a smooth 
approximating function in ${\mathcal D} (\T)$ of $\iota_\varepsilon$ such that 
\begin{equation}
\label{G_approx}
\|G_\varepsilon\|^2_{L^2(\T)} \leq 2\|\iota_\varepsilon\|^2_{L^2(\T)}=\varepsilon^{-1}\ \  {\rm and \ \ }\lim_{\varepsilon\downarrow 0}\varepsilon^{-1/2}\|G_\varepsilon -\iota_\varepsilon\|_{L^2(\T)}  \ = \ 0.
\end{equation}
Such approximating functions can be found by convoluting $\iota_\varepsilon$ with smooth kernels.  
For $x\in \T$, let $\tau_x$ denote the shift so that $(\tau_x \a) (z) = \a(z+x)$ and $(\tau_x G_\varepsilon ) (z) = G_\varepsilon(x+z)$.

With respect to a process $\Y_\cdot \in C([0,T]; {\mathcal D}^{\prime} (\T)^n)$, define $\A_\cdot^\varepsilon\in C([0,T], {\mathcal D}^{\prime} (\T)^{n})$, for $\varepsilon>0$, by its coordinate action on $H\in {\mathcal D} (\T)$:  For each $1\leq i\leq n$ and $t\in [0,T]$, we have
\begin{align*}
\label{A_varep_def}
\A_{t}^{i,\varepsilon}(H) & =\  \sum_{j,k=1}^n\partial_{a^j}\partial_{a^k}\tilde g_i({\BF a_0})\int_0^t \int_{\T} (\nabla H)  (u) \Y^j_s(\tau_{-u} G_\varepsilon)  \Y^k_s(\tau_{-u} G_\varepsilon)  \, du\,  ds.
\end{align*}
We say the process $\Y_\cdot \in D\big([0,T], {\mathcal D}^\prime (\T)^n\big)$ satisfies an {\em $L^2$ energy condition} if,
for $H\in {\mathcal D} (\T)$ and $1\leq i\leq n$, $\big({\mathcal A}_\cdot^{i,\varepsilon}(H)\big)_{\ve>0}$ is a `uniformly $L^2$ Cauchy' sequence, as $\varepsilon\downarrow 0$, with respect to the space of random trajectories equipped with the complete metric $d(x_\cdot, y_\cdot) = \E[\sup_{t\in [0,T]}(x_t - y_t)^2]^{1/2}$, that is,  
\begin{equation}
\label{en.cond}
\lim_{\varepsilon_1,\varepsilon_2\downarrow 0}\, \sup_{1\le i\le n}\, \E\left[ \sup_{t\in [0,T]}\Big(\A_{t}^{i,\varepsilon_1}(H) - \A_t^{i,\varepsilon_2}(H)\Big)^2\right] \ = \ 0,
\end{equation}
and the limit does not depend on the specific smoothing family $\{G_\varepsilon\}$.  The limit process $\big(\A^{i} \big)_{i=1}^n$ belongs to $C([0,T], {\mathcal D}^\prime (\T)^n)$ (see \cite{Walsh1}), and is defined by the uniformly $L^2$ Cauchy limit
$$\forall i\in \{1, \ldots, n\}, \quad \forall H\in{\mc D} (\T), \quad \A^{i}_{t}(H) \ := \ \lim_{\varepsilon\downarrow 0} \A^{i,\varepsilon}_{t}(H).$$

\begin{defn}
\label{frac_energy_defn}
\rm
We will say that $\Y_\cdot$ is a {\em multi-species stationary energy solution} of \eqref{gen_OU_sec} if the following holds.
\begin{itemize}
\item[(i)] For each fixed $t\in [0,T]$, $\Y_t$ is a spatial white noise with $n$-components and covariance function given by \eqref{standard_covariance}.
\item[(ii)] The process $\Y_\cdot$ satisfies the $L^2$-energy condition \eqref{en.cond}: For each $1\leq i\leq n$, there is a process $\A^{i}_\cdot\in C([0,T], {\mathcal D}^{\prime} (\T))$ whose action on $H\in {\mathcal D}^{\prime} (\T)$ is the uniformly $L^2$ Cauchy limit $\A^{i}_\cdot(H)$.  We also impose that $\A_\cdot(H)$ has zero quadratic variation.
 \item[(iii)] There is a
process $\M_\cdot\in C([0,T], {\mathcal D}^{\prime} (\T)^n)$, such that, for $H\in {\mathcal D} (\T)$ and $1\leq i\leq n$, $\M^{i}_\cdot(H)$ is a continuous martingale with respect to the filtration generated by $\Y_\cdot$ with quadratic variation $\langle \M^{i}(H)\rangle_t \ =\  \tilde{g}_i({\BF a_0})t\|\nabla H \|^2_{L^2}$.  Also, $M^i_\cdot$ and $M^j_\cdot$ are independent, when $1\leq i\neq j\leq n$.  Hence, for each $H\in {\mathcal D} (\T)$, by Levy's theorem, $\M(H)$ is an $n$-dimensional Brownian motion.
Moreover,
for $t\in [0,T]$ and $1\leq i\leq n$,
\begin{align*}
\M^i_t(H) &= \Y^i_t(H)-\Y^i_0(H) \\
&\ \ \ \ \ - \frac{1}{2} \int_0^t \partial_{a^i} \tilde{g}_i(\ab)\, \Y^i_s( \Delta H) ds
 - c\A^i_t(H).
\end{align*}
\item[(iv)] 
The time-reversed processes $\widehat\Y_\cdot = \Y_{T-\cdot}$ and $\widehat\A_\cdot = \A_T - \A_{T-\cdot}$ satisfy, for each $1\leq i\leq n$ and $H\in {\mathcal D}^{\prime} (\T)$, that
\begin{align*}
\widehat{\M}^i_t(H) &= 
\widehat\Y^i_t(H) - \widehat\Y_0(H)\\
&\ \ \ \ \  -  \frac{1}{2} \int_0^t \partial_{a^i} \tilde{g}_i(\ab)\, \widehat \Y^i_s(\Delta H) ds + c\widehat\A^i_t(H)
\end{align*}
is a continuous martingale in the filtration generated by $\widehat \Y_\cdot$, and $\widehat{\M}_0(H)=0$.  Also, $\widehat{\M}_\cdot(H)$ has the same quadratic and cross variation properties as $\M_\cdot$ in item (iii).
 \end{itemize}
\end{defn}

\begin{rem}\rm
\label{GPrem}
We remark when the drift coefficients $c\partial_{a^k}\partial_{a^j}\tilde{g}_i(\ab)$ vanish for all $1\leq i,j\leq n$, for instance when $c=0$, then \eqref{gen_OU_sec} represents a well-posed stochastic heat equation which has a unique solution. More generally, by Remark 4.13 in \cite{GPnew}, there is a unique process $\Y_\cdot$ satisfying the `multi-species energy solution' criteria in Definition \ref{frac_energy_defn}.
\end{rem}

We now come to the main result of the article.

\begin{theorem}
\label{secondorder_thm}
Suppose $\gamma = 1/2$ and the Frame condition (FC)
holds.  Starting from initial measure $\nu_{\BF a_0}$, the sequence $\{\Y^{N}_\cdot\; ; \;  N\geq 1\}$ converges
with respect to the
uniform topology on $D([0,T], {\mathcal D}^{\prime} (\T)^n)$, 
to the unique multi-species energy solution $\Y_\cdot$ of \eqref{gen_OU_sec}.

\end{theorem}

\begin{rem}
\label{KPZ_rmk}
\rm
Although we have worked on the tori $\T_N$, there are straightforward extensions of the theorems to ${\mathbb Z}$ and to an infinite volume limit.  In particular, Theorem \ref{drift_thm} will hold on ${\mathbb R}$ instead of $\T$.  Also, Theorem \ref{secondorder_thm} will extend to a limit point characterization--that is, the sequence $\{\Y^N_\cdot\}$ is tight with respect to the uniform topology, and any limit point satisfies the multi-species energy solution martingale problem on ${\mathbb R}$.  However the difficulty now is that it is not yet shown that there is a unique solution to this infinite volume martingale problem.  If there were such a unique solution, then Theorem \ref{secondorder_thm} would extend straightforwardly to ${\mathbb R}$.
\end{rem}

\begin{rem}
\rm
In Section \ref{example_section}, we consider multicolored systems where $g_i(\k) = g(|\k|)[k_i/|\k|]$ for $1\leq i\leq n$ in terms of a rate $g$.  In such systems, we show there are specifications when the Frame condition (FC)
holds with respect to a density $\ab$ and the coefficients $(c/2)\partial_{a^j,a^k}\tilde{g}_i(\ab)$ do not all vanish.  In this setting, the SPDE \eqref{gen_OU_sec} simplifies to \eqref{eq:5.2_reduced}, and we discuss that a unique process, even on $\R$, can be associated to it.
\end{rem}

\section{Proofs}
\label{proofs}

The arguments for Theorems \ref{drift_thm} and \ref{secondorder_thm} adapt the `hydrodynamics' scheme of \cite{GJS}, with some new features, to the multi-species context, developing the stochastic differential of ${\bar \Y}^{N}$ and $\Y^{N}$ into drift and martingale terms, before analyzing their limits.  Since the arguments of the two theorems are similar, to simplify the discussion, we only prove in detail Theorem \ref{secondorder_thm}, the most involved.

In Subsection \ref{Assoc_mart_section}, various generator actions are computed in general.  Then, in Subsection \ref{BG_statement}, a general `Boltzmann-Gibbs' principle is stated which will help close equations.  In Subsection \ref{tightness}, tightness of the processes in Theorem \ref{secondorder_thm} is shown.  In Subsection \ref{identification}, we identify several features of the limit points.   
Finally, in Subsection \ref{main_proof_subsec}, we finish the proof of Theorem \ref{secondorder_thm}.

\subsection{Stochastic differentials}
\label{Assoc_mart_section}

For $H\in {\mathcal D} (\T)$, $x\in \T_N$, and $N\geq 1$, define the scaled operators:
\begin{eqnarray*}
\Delta_{x}^{N}H &=&{N^2} \left\{ H\left(\frac{x+1}{N}\right) + H\left(\frac{x-1}{N}\right) - 2H\left(\frac{x}{N}\right)\right\},\\ 
\nabla_{x}^{N}H &=& \frac{N}{2}\left\{H\left(\frac{x+1}{N}\right) - H\left(\frac{x- 1}{N}\right)\right\}.
\end{eqnarray*}
Define, for $\gamma=1/2$ and $s\geq 0$,
\begin{equation}
\label{shifted_H}
H_{s}(\cdot)   \ = \  H\Big( \cdot - \frac{1}{N}\Big\lfloor\frac{2c\lambda sN^2}{\sqrt{N}}\Big\rfloor\Big) \quad  {\rm and} \quad \HW_{s}(\cdot)  \ = \  H\Big( \cdot - \frac{1}{N} \frac{2c\lambda sN^2}{\sqrt{N}}\Big), 
\end{equation}
functions seen in frames along $N^{-1}\T_N$ and $\T$ respectively which will be useful, where $\lfloor S \rfloor$ denotes the integer part of $S\in \R$. Now, for $H=(H^i)_{i=1}^n \in {\mathcal D}(\T)^n$ and $s\in [0,T]$,  let 
$$F(s, {\BF \alpha}^N_s;H) \equiv \Big( F^i (s, {\BF \alpha}_s^N, H) \Big)_{i=1}^n= \Y^{N}_s(H)$$
and write, for $1\leq i\leq n$,
\begin{eqnarray*}
L_N F^i (s, {\BF \alpha}^N_s; H) & = & \frac{1}{2\sqrt{N}}\sum_{x\in \T_N}g_i({\BF \alpha}^{N}_s(x)) \Delta^{N}_{x}\HW^i_{s} + \ 2c \sum_{x\in \T_N} g_i({\BF \alpha}^N_s(x))\nabla^{N}_x\HW^i_{s}.\end{eqnarray*}
Also,
\begin{eqnarray*}
\partial_s F^i(s,{\BF \alpha}^{N}_s; H) & = &{-2c\lambda}\sum_{x\in \T_N} \nabla \HW^i_{s}\left(\frac{x}{N}\right)\left(\alpha^{i,N}_s(x)-a_0^i\right),
\end{eqnarray*}
where $\partial_s$ acts only on the first coordinate of $F^i$. By Dynkin's formula, the process $\{ {\mathcal M}_t^N (H) \; ;\; t\in [0,T]\}$ defined by
\begin{eqnarray*}
\M^{N}_t (H)  &:=& \ F(t, {\BF \alpha}^N_t; H) - F(0, {\BF \alpha}^N_0; H)  - \int_0^t \Big\{ {\partial_s}F(s,{\BF \alpha}^N_s; H) + L_N F(s, {\BF \alpha}^N_s; H)\Big\} ds
\end{eqnarray*}
is a martingale.  
We write
\begin{equation}
\label{mart_decomposition}
\M^{N}_t(H)
\ = \ \Y^{N}_t(H) - \Y^{N}_0(H) - \I^{N}_t(H) -\B^{N}_t(H) -\K^{N}_t(H).
\end{equation}
Here, the $i$th components of the fields above, noting the Frame condition (FC) in Subsection \ref{subsubsec:NFKPZ},
are the following:
\begin{eqnarray*}
&&\mathcal{I}^{i,N}_t(H)  = 
 \int_0^t \frac{1}{2\sqrt{N}}\sum_{x\in \T_N} \big(g_i({\BF \alpha}^N_s(x))-\tilde{g}_i({\BF a_0})\big)\Delta^{N}_{x}H^i_{s} ds,
\\
&&\mathcal{B}^{i,N}_t(H)  =  {2c} \int_0^t \sum_{x\in \T_N} \Big(g_i({\BF \alpha}^N_s(x))-\tilde{g}_i({\BF a_0})\\
&&\ \ \ \ \ \ \ \ \ \ \ \ \ \ \ \ \ \ \ \ \ \ \ \ \ \ \ \ \ \ \ \ \ \ \ \ \ \ \ \ \ \ - \sum_{j=1}^n\partial_{a^j}\tilde g_i(\ab) (\alpha^{j,N}_s(x)-a_0^j)\Big)\nabla^{N}_xH^i_{s} ds,\\
&&\K^{i,N}_t(H) = \int_0^t\Big[\frac{1}{2\sqrt{N}}\sum_{x\in \T_N}\kappa^{N,1}_{x}(H^i,s)\big(g_i({\BF \alpha}^{N}_s(x))-\tilde{g}_i({\BF a_0})\big)\\ 
&& \ + {2c}\sum_{x\in \T_N}\kappa^{N,2}_x(H^i,s)\Big(g({\BF \alpha}^{N}_s(x))-\tilde{g}_i({\BF a_0}) - \sum_{j=1}^n \partial_{a^j}\tilde g_i({\BF a_0})(\alpha^{j,N}_s(x)-a_0^j)\Big)\Big]ds.
\end{eqnarray*}
We use the same notation ${\mathcal M}_t^N (H), {\mathcal Y}_t^N (H), {\mathcal I}_t^N (H), \mathcal B_t^N (H)$ and $\mathcal K_t^N (H)$for $H \in {\mc D} (\mathbb T)$ by replacing $H^i$ with $H$ in these formulas. Centering constants were introduced in $\I^{i,N}_t$ and $\B^{i,N}_t$ for free, as $\sum_x\Delta^{N}_{x}H^i_{s} = \sum_x\nabla^{N}_xH^i_{s}=0$.   By Taylor expansion, uniformly in $x$, we have 
\begin{eqnarray*}
&&\kappa^{N,1}_{x}(H^i,s) \ = \ \Delta^{N}_{x} \big(\HW^i_{s} - H^i_{s}\big) \ = \ O(N^{-1}),\\
&&\kappa^{N,2}_x(H^i,s) \ = \ \nabla^{N}_{x} \big(\HW^i_{s} - H^i_{s}\big) \ = O( N^{-1}).
\end{eqnarray*}
Also, the process
$\{ (\M^{i,N}_t(H))^2 - \langle \M^{i,N}(H)\rangle_t\; ; \; t\in [0,T]\}$ is a martingale with quadratic variation
\begin{eqnarray*}
\langle \M^{i,N}(H) \rangle_t & = & \int_0^t  \frac{1}{N}\sum_{x\in \T_N} \sum_{\ve =\pm 1} p (\ve) g_i({\BF \alpha}^{N}_s(x))  \Big[ N \big\{ \HW_s^i (\tfrac{x +\ve}{N}) - \HW_s^i(\tfrac{x}{N} )\big\}  \Big]^2\, ds
\end{eqnarray*}
where we recall that
\begin{equation*}
p(\ve) = \cfrac{1}{2} + \cfrac{c \ve}{\sqrt N}. 
\end{equation*}

In addition, we have the bound, following an expansion of an exponential martingale as in Section 3.1 of \cite{GJS}:
\begin{equation}
\label{quad_var_bound}
\quad \quad \E_{\nu_{\BF a_0}}\big[\big(\M^{i,N}_t(H)-\M^{i,N}_s(H)\big)^4\big]  \ \leq \ C(c, {\BF a_0}, g, H,n)\Big\{
|t-s|^2 + \frac{1}{N^{3/2}}|t-s|\Big\}. 
\end{equation}

\subsection{Boltzmann-Gibbs principle} \label{BG_statement}

Define, for $\zeta\in \Omega={\mathbb Z}_+^{\mathbb T_N}$ (resp. $\zeta \in \Omega^n$)  and $\ell\geq 1$, that
$$\zeta^{(\ell)}(x) \ := \ \frac{1}{2\ell+1}\sum_{y\in \Lambda_\ell}\zeta(x+y)$$
where we recall that $\Lambda_{\ell} = \{x\in\T_N: |x|\leq \ell\}$.   
For $h:\T_N\rightarrow\R$, $T>0$ and $c\in {\mathbb R}$, define $\bar h_{c, s}:[0,T]\times \T_N\rightarrow \R$ by $\bar h_{c,s}(x) = \bar h(x - \lfloor cs\rfloor)$. The density $\ab$ in the following is generic, and need not satisfy the Frame condition (FC).

\begin{theorem}
 \label{gbg_L2}
Let $\ell_0 \ge 1$, $f:\Omega^n\rightarrow \R$ be a local $L^5(\nu_{\ab})$ function supported on sites in $\Lambda_{\ell_0}$ such that $\tilde{f}(\ab)=0$ and $\nabla \tilde{f} (\ab)=0$.
There exists a constant $C=C(\ab,\ell_0,n)$
such that, for $T>0$, $\ell\geq \ell_0$ and $h:\T_N\rightarrow \T$,
\begin{align*}
& \E_{\nu_\ab}\Big[ \sup_{0\leq t\leq T}\Big( \int_0^t  \sum_{x\in\T_N} \Big(f(\tau_x{\BF \alpha}^{N}_s)\\
 &\ \ \ \ \ \ \ \ \ \ - \frac{1}{2}\sum_{j,k=1}^n \partial_{a^j}\partial_{a^k}\tilde{f}(\ab)\Big\{\Big(\big(\alpha^{j,N}_s\big)^{(\ell)}(x)-a_0^j\Big)\Big(\big(\alpha^{k,N}_s\big)^{(\ell)}(x)-a_0^j\Big)\\
 &\ \ \ \ \ \ \ \ \ \ \ \ \ \ \ \ \ \ \ \ \ \ \ \ \ \ \ \ \ \ \ \ \ \ \ \ \ \ \ \ \ \ \ \ \ \ \ \ -\frac{\Gamma_{jk} (\ab)}{2\ell+1} \Big\}
 \bar h_{c,s}(x)ds\Big)^2 \Big]\\
&\ \ \ \ \ \ \ 
    \leq \ C\|f\|^2_{L^5(\nu_{\ab})}\bigg(\frac{T\ell }{N}\Big(\frac{1}{N}\sum_{x\in \T_N}h^2(x)\Big) + \frac{T^2N^2}{\ell^{3}}\Big(\frac{1}{N}\sum_{x\in{\T_N}}|h(x)|\Big)^2\bigg).
\end{align*}

On the other hand, when only $\tilde{f}(\ab)=0$ is known,
\begin{align*}
& \E_{\nu_\ab}\Big[ \sup_{0\leq t\leq T}\Big( \int_0^t  \sum_{x\in{\T_N}} \Big( f(\tau_x {\BF \alpha}^{N}_s) 
  - \sum_{i=1}^n \partial_{a^i}\tilde{f}(\ab)\Big\{\big(\alpha^{i,N}_s\big)^{(\ell)}(x) - a_0^i\Big\} 
\bar h_{c,s}(x)ds\Big)^2\Big]\\
&\ \ \ \ \ \ \ \ 
    \leq \ C\|f\|^2_{L^5(\nu_{\a})}\bigg(\frac{T \ell^2}{N}\Big(\frac{1}{N}\sum_{x\in \T_N}h^2(x)\Big) + \frac{T^2N^2}{\ell^{2}} \Big(\frac{1}{N} \sum_{x\in{\T_N}}|h(x)|\Big)^2\bigg).
\end{align*}
\end{theorem}

The first estimate is used for ${\mathcal B}_t^{i,N} (H)$ in \eqref{tightness_B} and then in the proof of Proposition \ref{stat_lemma1}, while the second is used for $\mathcal I^{i,N}_t (H)$. We prove this theorem in Section \ref{proof_Boltzmann}.

\subsection{Tightness}
\label{tightness}
We prove tightness of the fluctuation fields, starting from the invariant measure $\nu_{\ab}$, using the Boltzmann-Gibbs principle.

\begin{prop}
\label{stationary_tightness}
The sequences $\{\Y^{N}_t: t\in [0,T]\}_{N\geq 1}$, $\{\M^{N}_t: t\in [0,T]\}_{N\geq 1}$, $\{\I^{N}_t: t\in [0,T]\}_{N\geq 1}$, $\{\B^{N}_t: t\in [0,T]\}_{N\geq 1}$, $\{\K^{N}_t: t\in [0,T]\}_{N\geq 1}$ and $\{\langle \M^{N}\rangle_t: t\in [0,T]\}_{N\geq 1}$, when starting from invariant measure $\nu_{\ab}$, are tight in the uniform topology on $D([0,T], {\mathcal D}^{\prime} (\T)^n)$.
\end{prop}
\begin{proof}
 By Mitoma's criterion \cite{Mitoma}, to prove tightness of the sequences with respect to uniform topology on $D([0,T], {\mathcal D}^{\prime} (\T)^n)$, it is enough to show, for each $1\leq i\leq n$, tightness of $\Y^{i,N}_\cdot (H)$, $\M^{i,N}_\cdot(H)$, $\I^{i,N}_\cdot (H)$, $\B^{i,N}_\cdot (H)$, $\K^{i,N}_\cdot (H)$ and $\langle \M^{i,N} (H)\rangle_\cdot$, with respect to the uniform topology for all $H\in {\mathcal D} (\T)$.  Note that all initial values vanish, except $\Y^{i,N}_0(H)$. 

Tightness of $\{\Y^{i,N}_t(H)\; ; \; t\in [0,T]\}$, in view of the decomposition $\Y^{i,N}_t(H) = \Y^{i,N}_0(H) + \I^{i,N}_t(H)  + \B^{i,N}_t(H) + \K^{i,N}_t(H)+\M^{i,N}_t(H)$, will follow from tightness of each term.  The tightness of $\Y^{i,N}_0(H)$, given that we begin under $\nu_{\ab}$, is immediate.

For the martingale term, we use Doob's inequality, stationarity of $\nu_{\BF a_0}$, estimate \eqref{quad_var_bound}, and dividing the time interval into subintervals of size $\delta^{-1}$  to obtain
\begin{eqnarray*}
&&\P_{\nu_\ab}\Big( \sup_{\stackrel{|t-s|\leq \delta}{0\leq s,t\leq T}} |\M^{N}_t(H) - \M^{N}_s(H)|>\varepsilon\Big)\\
&& \ \ \ \ \ \ \lesssim \  3 T \varepsilon^{-4}\delta^{-1} \E_{\nu_\ab}\big[ \big(\M_\delta^{N}(H)\big)^4\big] \ \leq \ C(c, {\BF a_0}, \varepsilon, g, H,T)\Big\{\delta + \frac{1}{N^{3/2}}\Big\}.
\end{eqnarray*}
Tightness now straightforwardly follows.

We now prove tightness for $\{\B^{i,N}_t(H)\; ; \; t\in[0,T]\}$ through the Kolmogorov-Centsov criterion.
Let
$$
V_i(\BF \alpha)  \ = \ g_i(\BF \alpha) - \tilde g_i({\BF a_0}) - \sum_{j=1}^n \partial_{a^j}\tilde g_i({\BF a_0})(\alpha^j(0) -a_0^j). $$
Then, for $H\in {\mc D} ({\mathbb T})$, 
$$\B^{i,N}_t(H) \ =\   2c\int_0^t \sum_{x\in\T_N} \big(\nabla^N_x H_{s}\big) \, \tau_x V_i(\BF\alpha^N_s)ds.$$
By its form, $\tilde V_i({\BF a_0}) = 0$ and $\partial_{a^j} \tilde V_i({\BF a_0}) = 0$ for $1\leq j\leq n$.  Also, for $1\leq j,k\leq n$, $\partial_{a^j}\partial_{a^k}\tilde V_i({\BF a_0})=\partial_{a^j}\partial_{a^k}\tilde g_i({\BF a_0})$.
   
By invoking Theorem \ref{gbg_L2}, with $\ell_0=1$ and $h(x) =\nabla^N_x H $, and translation-invariance of $\nu_{\BF a_0}$, 
for $\ell\geq 1$, we have
\begin{equation}
\label{tightness_B}
\begin{split}
&\E_{\nu_{\BF a_0}} \Big[\Big(\B^{i,N}_t(H)\\\
&
  - c\sum_{x \in \T_N} \sum_{j,k=1}^n \int_0^t \partial_{a^j}\partial_{a^k}\tilde{g}_i({\BF a_0})\Big\{\Big(\big(\alpha^{j,N}_s\big)^{(\ell)}(x)-a^j_0\Big)\Big(\big(\alpha^{k,N}_s\big)^{(\ell)}(x)-a^j_0\Big) \\
  &\ \ \ \ \ \ \ \ \ \ \ \  \ \ \ \ \ \ \ \ \ \ \ \ \ \ \ \ \ \ \ \ \ -\frac{\Gamma_{j,k}( \ab) }{2\ell+1} \Big\}
\nabla^N_x H_{s} \, ds\Big)^2 \Big]\\
&\  \ \ \ \ \ \ \ \ \ \ \ \ \leq \ \ C({\BF a_0},c, g_i) \Big\{ \frac{t\ell}{N} + \frac{t^2N^2}{\ell^{3}}\Big\} \Big[\Big(\frac{1}{N}\sum_{x\in \T_N} (\nabla^N_xH)^2 \Big) + \Big(\frac{1}{N}\sum_{x\in\T_N}|\nabla^N_xH| \Big)^2\Big]. 
\end{split}
\end{equation}
On the other hand, noting $E_{\nu_{\BF a_0}}[|\BF \alpha^\ell(0) - \ab|^4]\leq C\ell^{-2}$, the square of the $L^2(\nu_{\BF a_0})$ norm of the integral term above is bounded by
$$C({\BF a_0}, c, g, H)\frac{t^2N^2}{\ell^2}\Big(\frac{1}{N}\sum_x \big|\nabla^N_xH  \big|\Big)^2.$$
Hence, for $\ell >1$, we have
$$\E_{\nu_{\BF a_0}} \big[(\B^{i,N}_t(H))^2\big] \ \leq \ C(c,{\BF a_0}, g,H) \Big [\frac{t\ell}{N} + \frac{t^2N^2}{\ell^2}\Big].$$ 
By stationarity we have also 
$$\E_{\nu_{\BF a_0}} \big[(\B^{i,N}_t(H) -\B^{i,N}_s (H) )^2\big] \ \leq \ C(c,{\BF a_0}, g,H) \Big [\frac{|t-s| \ell}{N} + \frac{|t-s|^2N^2}{\ell^2}\Big].$$ 
Then, if $\ell$ is taken as $\ell = |t-s|^{1/3}N>1$, we conclude $\E_{\nu_{\BF a_0}}\big[(\B^{i,N}_t(H) - \B^{i,N}_s(H))^2\big] \leq C(c, {\BF a_0},g,H) |t-s|^{4/3}$. However, when $|t-s|^{1/3}N \leq 1$, we have
that
\begin{equation}
\label{4-3_estimate}
\begin{split}
\E_{\nu_{\BF a_0}}\Big[(\B^{i,N}_t(H) - \B^{i,N}_s(H))^2\Big] & \leq  C(c,{\BF a_0}, g) (t-s)^2 N^2 \Big(\frac{1}{N}\sum_x |\nabla_x^N H |\Big)^2 \\ 
&\leq \;  C(c, {\BF a_0}, g,H) |t-s|^{4/3}. 
\end{split}
\end{equation}
This shows tightness of $\B_\cdot^{i,N}(H)$.
 
Finally, the argument for $\I^{i,N} (H)$,  $\langle \M^{i,N}(H)\rangle $ and $\K^{i,N} (H)$, given their forms, are simpler and can be done using invariance of the product measure $\nu_{\BF a_0}$ by squaring all terms, with a Kolmogorov-Centsov right-hand estimate of order $|t-s|^2$, completing the proof.
\end{proof}

\subsection{Identification of limit points}\label{identification}

With tightness (Proposition \ref{stationary_tightness}) in hand, we now identify the limit points of $\{\Y^{N}_t \; ; \;  t\in [0,T]\}_{N\ge 1}$ and its parts in decomposition \eqref{mart_decomposition}.  Let $Q^N$ be the distribution of $$\left\{\Big(\Y^{N}_t, \M^{N}_t, \I^{N}_t, \B^{N}_t, \K^{N}_t, \langle \M^{N}\rangle_t\Big) \; ;\; t\in [0,T]\right\},$$
and let $N'$ be a subsequence where $\{Q^{N'}\}_{N'}$ converges to a limit point $Q$.  Let also $\Y_t$, $\M_t$, $\I_t$, $\B_t$, $\K_t$ and $\D_t$ be the respective limits in distribution of the components.
Since tightness is shown in the uniform topology on $D([0,T], {\mathcal D}^{\prime} (\T)^n)$, we have that $\Y_t$, $\M_t$, $\I_t$, $\B_t$, $\K_t$ and $\D_t$ have a.s. continuous paths.

Let now $G_\varepsilon:\T \rightarrow [0,\infty)$ be a smooth 
function for $0<\varepsilon\leq 1$ which approximates $\iota_\varepsilon: z \in \T \to  \varepsilon^{-1}1(|z|\leq \varepsilon)$ as in the definition of energy solution, see Definition \ref{frac_energy_defn} and before it. That is, $\|G_\varepsilon\|^2_{L^2(\T)}\leq 2\|\iota_\varepsilon\|^2_{L^2(\T)}=\varepsilon^{-1}$ and $\lim_{\varepsilon\downarrow 0}\varepsilon^{-1/2}\|G_\varepsilon - \iota_\varepsilon\|_{L^2(\T)}=0$.
Define, for $1\leq j,k\leq n$, and $H\in {\mc D} (\T)$, 
$$\A^{j,k,\varepsilon,N}_{t}(H) \ := \ \int_0^t \frac{1}{N}\sum_{x\in \T_N} (\nabla^N_x H)\big[\tau_x  \Y^{j,N}_s(G_\varepsilon) \, \tau_x\Y^{k,N}_s(G_\varepsilon)\big]ds.$$

Since for fixed $0<\varepsilon\leq 1$ and $0\le t < t'\le T$ the map 
$$\pi_\cdot \mapsto \int_t^{t'} ds \int du\big(\nabla H(u)\big)\big\{\pi^j_s(\tau_{-u} G_\varepsilon)\pi^k_s( \tau_{-u} G_\varepsilon)\big\}$$
is continuous in the uniform topology on $D([0,T]\, , \, {\mathcal D}^{\prime} (\T)^n)$, we have subsequentially in distribution that for any time $t\in [0,T]$
\begin{align*}
&\lim_{N'\uparrow\infty}\sum_{j,k}\partial_{a^j}\partial_{a^k}\tilde g_i({\BF a_0})\A^{j,k,\varepsilon,N'}_{t}(H) \\ 
&\ \ \ \ \ =\ \sum_{j,k}\partial_{a^j}\partial_{a^k}\tilde g_i({\BF a_0})\int_0^t ds\int du \big(\nabla H(u)\big)\big\{\Y^j_s(\tau_{-u} G_\varepsilon)\Y^k_s(\tau_{-u}G_\varepsilon)\big\} \  =: \ \A^{i,\varepsilon}_{t}(H).
\end{align*}

\begin{prop}
\label{stat_lemma1}
We have, for each $1\leq i\leq n$, $H \in {\mc D} (\T)$ and $t\in [0,T]$, that
\begin{eqnarray*}
&&\limsup_{N\uparrow\infty}\E_{\nu_{\BF a_0}} \Big[ \Big|\B^{i,N}_t(H) - c\sum_{j,k=1}^n\partial_{a^j}\partial_{a^k}\tilde g_i({\BF a_0})\A^{j,k,\varepsilon,N}_{t}(H)\Big|^2\Big] \\
&&\ \ \ \ \ \  \leq \ C(c,{\BF a_0}, g,T) \Big(\varepsilon + \varepsilon^{-1}\|G_\varepsilon - \iota_\varepsilon\|^2_{L^2(\T)}\Big)  \Big[\|\nabla H\|^2_{L^2(\T)} + \|\nabla H\|^2_{L^1(\T)}\Big].
\end{eqnarray*}
Then, on the underlying common probability space,
$\{{\mathcal A}^{i,\varepsilon}_{\cdot}(H)\}_{\ve >0}$ is a uniformly $L^2$ Cauchy sequence, as $\varepsilon\downarrow 0$, as specified in \eqref{en.cond}.  Therefore, the limit
$ \lim_{\varepsilon\downarrow 0} \A^{i,\varepsilon}_{\cdot}(H) =: \A^i_\cdot(H)\in C([0,T], \R)$, does not depend on the specific family $\{G_\varepsilon\}$, and is stationary, $\A^i_{t}(H) - \A^i_s(H)\stackrel{d}{=} \A^i_{t-s}(H)$ for $0\leq s\leq t\leq T$.  Also, a.s. on the underlying common probability space, for $0\leq t\leq T$, we have $c\A^i_t(H) = \B^i_t(H)$.  Moreover, the process $\A^i_\cdot(H)$ has zero quadratic variation.

In addition, for each $1\leq i\leq n$,
\begin{eqnarray*}
&&\lim_{N\uparrow\infty} \E_{\nu_{\BF a_0}}\Big[\Big |\I^{i,N}_t(H) - \frac{1}{2}\sum_j\partial_{a^j}\tilde g_i({\BF a_0})\int_0^t \Y^{j,N}_s(\Delta H)ds\Big|^2\Big] \ = \ 0\\
&&\lim_{N\uparrow\infty} \E_{\nu_{\BF a_0}}\Big[\Big| \langle \M^{i,N}(H)\rangle_t - \frac{\tilde g_i({\BF a_0})}{2}t\|\nabla H\|^2_{L^2(\T)}\Big|^2\Big] \ = \ 0\\
&&\lim_{N\uparrow\infty} \E_{\nu_{\BF a_0}}\Big[\Big| \K^{i,N}_t(H) \Big |^2\Big] \ = \ 0.
\end{eqnarray*}

Then, in $L^2$, with respect to the underlying common probability space, we have 
$$\K^i_t(H)= 0,\ \ \I^i_t(H) \ = \ \frac{1}{2}\partial_{a^i}\tilde g_i({\BF a_0})\int_0^t\Y^i_s(\Delta H)ds \ \ {\rm and \ \ }\D^i_t(H) \ =\  \frac{\tilde g_i({\BF a_0})}{2}t\|\nabla H\|^2_{L^2(\T)}.$$
Also, 
$\M^i_\cdot(H)$ is a continuous martingale with quadratic variation $\D^i_\cdot(H)$.  Moreover, the cross variation between components $\M^i_\cdot(H)$ and $\M^j_\cdot (H)$ vanish for $1\leq i\neq j\leq n$, and therefore, by Levy's theorem, $\M_\cdot$ is a version of the noise in \eqref{gen_OU_sec}. 

\end{prop}

\begin{proof} Suppose the limit display for $\B^i_\cdot (H)$ holds.  Then, by Fatou's lemma, valid for weakly converging random variables through Skorohod's representation theorem, that
$$\E_{\nu_{\BF a_0}}\big[\big|\B^i_t(H) - c \A^i_t(H)
\big|^2\big]\leq \lim_{\varepsilon\downarrow 0} C(c,{\BF a_0},g, H)t\varepsilon =0,$$
and so $c\A^i_t(H)=\B^i_t(H)$ a.s.  Moreover, since by the tightness estimate \eqref{4-3_estimate} in the proof of Proposition \ref{stationary_tightness}, we have $\E_{\nu_\ab}\big[(\B^I_t(H)-\B^i_s(H))^2\big]\lesssim |t-s|^{4/3}$, and therefore $\E_{\nu_{\ab}}\big[(\A^i_t(H)-\A^i_s(H))^2\big]\lesssim |t-s|^{4/3}$, we conclude the process $\A^i_\cdot(H)$ has zero quadratic variation.

We now show the limit with respect to $\{\B^{i,N}_t(H)\}_{N\ge 1}$.  Note, for $\ell= \lfloor \varepsilon N\rfloor$, that
\begin{align*}
&\sum_{x\in \T_N}(\nabla^N_x H_{s})\Big(\big(\alpha^{j,N}_s\big)^{(\ell)}(x) - a^j_0\Big)\Big(\big(\alpha^{k,N}_s\big)^{(\ell)}(x)-a^k_0\Big)\\
& \ =  \sum_{x\in \T_N}(\nabla^N_x H_{s})\Big(\frac{1}{2\ell+1}\sum_{|z|\leq \ell}(\alpha^{j,N}_s(z+x) - a^j_0)\Big)\Big(\frac{1}{2\ell +1}\sum_{|z|\leq \ell}(\alpha^{k,N}_s(z+x) - a^k_0)\Big)\\
& \ =  \frac{1+O(N^{-1})}{N}\sum_{x\in \T_N} (\nabla^N_x H)\big[\tau_x \Y^{j,N}_s(\iota_\varepsilon)\big]\big[\tau_x \Y^{k,N}_s(\iota_\varepsilon)\big].
\end{align*}
Here, the shift by $N^{-1}\lfloor c\lambda sN^2/(2\sqrt{N})\rfloor$ in $\tau_x\Y^{N,j}_s\cdot \tau_x\Y^{N,k}_s$ was transferred from $\nabla^N_x H_{s}$ (cf. \eqref{shifted_H}). Then, with $\ell = \lfloor \varepsilon N\rfloor$, by Theorem \ref{gbg_L2}, as in the bound \eqref{tightness_B}, we have
\begin{eqnarray*}
&&\limsup_{N\uparrow\infty}\E_{\nu_{\BF a_0}}\Big[\sup_{0\leq t \leq T}\Big(\B^{i,N}_t(H)\\
&&\ \ \ \ \ \ \ \ \ \ \ \ \ \ \ \ \ \ \ \ - c\sum_{j,k}\partial_{a^j}\partial_{a^k}\tilde g_i({\BF a_0})\int_0^t \frac{1}{N} \sum_{x\in \T_N} (\nabla^N_x H)\tau_x \; \Y^{j,N}_s(\iota_\varepsilon)\tau_x\Y^{k,N}_s(\iota_\varepsilon) ds\Big)^2\Big]\\
&&\  = \ \limsup_{N\uparrow\infty}\E_{\nu_{\BF a_0}}\Big[\sup_{0\leq t\leq T}\Big( \B_t^{i,N}(H) \\
&&\ \ \ \ \ \ \ \  - c\sum_{j,k}\partial_{a^j}\partial_{a^k}\tilde g_i({\BF a_0})\int_0^t  \sum_{x\in \T_N} (\nabla^{N}_x H_{s})\Big\{\big((\alpha^{j,N}_s)^{(\ell)}(x) - a^j_0\big)\big((\alpha^{k,N}_s)^{(\ell)}(x)-a^k_0\big)\\
&&\ \ \ \ \ \ \ \ \ \ \ \ \ \ \ \ \ \ \ \ \ \ \ \ \ \ \ \ \ \ \ \ \ \ \ \ \ \  - \frac{\Gamma_{j,k} (\ab)}{2\ell +1}\Big\}ds\Big)^2\Big]\\
&&\  \leq \ \limsup_{N\uparrow\infty}C(c,{\BF a_0}, g, T)\Big( \varepsilon + \frac{1}{\varepsilon^3 N}\Big)\Big[\Big(\frac{1}{N}\sum_{x\in \T_N} \big(\nabla^N_x H\big)^2\Big) + \Big(\frac{1}{N}\sum_{x\in \T_N}\big|\nabla^N_x H\big|\Big)^2\Big].
\end{eqnarray*}
Here, as the sum of $\nabla^N_x H_{s}$ on $x$ vanishes, we introduced the centering constant $(2\ell +1)^{-1}\Gamma_{j,k} (\ab)$ in the second line.  
Now, 
\begin{align*}
&\Y^{j,N}_s(\iota_\varepsilon)\Y^{k,N}_s(\iota_\varepsilon) - \Y^{j,N}_s(G_\varepsilon)\Y^{k,N}_u(G_\varepsilon) \\
&\ \ \ \ \ \ \ \ \ \ \ \ \ \ \ \  =\  \big[\Y^{j,N}_s(\iota_\varepsilon) - \Y^{j,N}_s(G_\varepsilon)\big] \Y^{k,N}_s (\iota_\varepsilon) 
+ \Y^{j,N}_s(G_\varepsilon) \big[\Y^{k,N}_s(\iota_\varepsilon) - \Y^{k,N}_s(G_\varepsilon)\big].
\end{align*}
By Schwarz inequality and stationarity of $\nu_{\BF a_0}$, 
$$
\limsup_{N\uparrow\infty} \E_{\nu_{\BF a_0}} \Big[\sup_{0\leq t\leq T}\Big(\int_0^t\frac{1}{N} \sum_{x\in \T_N} (\nabla^{N}_x H)\tau_x\Y^{N,j}_s (\iota_\varepsilon)\tau_x\Y^{N,k}_s (\iota_\varepsilon) ds - \A^{j,k,\varepsilon,N}_{t}(H)\Big)^2\Big]$$
is less than
\begin{eqnarray*}
&&T^2 \limsup_{N\uparrow\infty} \E_{\nu_{\BF a_0}} \Bigg[\Bigg(\frac{1}{N} \sum_{x\in \T_N} |\nabla^{N}_x H|\big|\tau_x\Y^{j,N}_0 (\iota_\varepsilon)\tau_x\Y^{k,N}_0 (\iota_\varepsilon) \\
&&\ \ \ \ \ \ \ \ \ \ \ \ \ \ \ \ \ \ \ \ \ \ \ \ \ \ \ \ \ \ \ \ \ \ \ \ \ \ \ \ \ \ \ \ \ \ \ \ \ \ \ \ \ \ \ - \tau_x \Y^{j,N}_0 (G_\varepsilon)\tau_x\Y^{k,N}_0 (G_\varepsilon)\big| \Bigg)^2\Bigg]\\
&&\ \leq \ C(\ab) T^2 \varepsilon^{-1}\|G_\varepsilon-\iota_\varepsilon\|^2_{L^2(\T)}\Bigg(\frac{1}{N}\sum_{x\in \T_N}\big|\nabla^{N}_x H\big|\Bigg)^2.
\end{eqnarray*}
These estimates with the inequality $(a+b)^2 \leq 2a^2 + 2b^2$ finish the proof of the $\{\B^{i,N}_\cdot\}_{N\ge 1}$ limit.


 
The proof for the limit of $\{\I^{i,N}_\cdot\}_N$ is analogous, since $\Delta H$ is uniformly continuous.  Also,
the arguments for $\{\K^{i,N}_\cdot \}_N$ and $\{\langle \M^{i,N}\rangle_\cdot\}_N$ and identification of limits of $\K^i_\cdot$ and $\D^i_\cdot $, noting their forms, and that the process starts from invariant product measure $\nu_{\BF a_0}$, follow by straightforward $L^2$ calculations.

We now address the martingale convergence. 
By the identification given before, any limit point of the quadratic variation sequence equals $\D_\cdot(H)$.
Also, the limit of martingale sequence, with respect to the uniform topology, $\M_\cdot(H)$, is a continuous martingale.  Now, by the triangle inequality, Doob's inequality and the quadratic variation bound \eqref{quad_var_bound}, for $1\leq i\leq n$,
\begin{eqnarray*}
&& \E_{\nu_{\BF a_0}}\Big[\sup_{0\leq s\leq t} |\M^{i,N}_s(H) - \M^{i,N}_{s -}(H)|\Big]\\
&&\ \ \ \ \ \ \ \ \ \ \ \ \ \leq \ 2\E_{\nu_{\BF a_0}} \Big[ \sup_{s\in [0,t]}|\M^{i,N}_s(H)|^2\Big]^{1/2}\\
&& \ \ \ \ \ \ \ \ \ \  \ \ \  \leq \ 2 \E_{\nu_{\BF a_0}}\Big[\langle \M^{i,N}(H)\rangle_t\Big]^{1/2} \ \leq \ C(c,{\BF a_0}, g,H,T).
\end{eqnarray*}
Then, by Corollary VI.6.30 of
\cite{JS}, $(\M_\cdot^{i,N'}(H), \langle \M^{i,N'}(H)\rangle_\cdot)$ converges in distribution to $(\M^i_\cdot(H), \langle \M^i(H)\rangle_\cdot)$.  Since $\langle \M^{i,N'}(H)\rangle_\cdot$ converges in distribution to $\D^i_\cdot(H)$, we have $\langle \M^i(H)\rangle_t = \tilde g_i({\BF a_0})t\|\nabla H\|^2_{L^2}$ for $0\leq t\leq T$.

Moreover, for $i\neq j$, the cross variations vanish:
\begin{align}  \label{eq:4.cross}
\frac{d}{dt} & \lan \M^{i,N}(G_1), \M^{j,N}(G_2)\ran_t \\
& = \! \! N \left(L_N (\lan \a_t^{i,N},G_1\ran \lan \a^{j,N}_{t},G_2\ran) \right.  \notag  \\
& \hskip 15mm  \left.  - \lan \a_t^{i,N},G_1\ran L_N\lan \a^{j,N}_{t},G_2\ran
- \lan \a^{j,N}_{t},G_2\ran L_N \lan \a_t^{i,N},G_1\ran \right) =0.  \notag
\end{align}
Hence, by Levy's theorem, $\M_\cdot$ is a version of the noise desired.
This finishes the proof.
\end{proof}

\subsection{Proof of Theorem \ref{secondorder_thm}: Nonlinear fluctuations}
\label{main_proof_subsec}

 Let $H\in {\mathcal D} (\T)$, $t\in [0,T]$.  By the decomposition \eqref{mart_decomposition}, identification Proposition \ref{stat_lemma1}, and tightness of the constituent processes $\M^{N}_t$, $\Y_t^{N}$, $\Y^{N}_0$, $\I^{N}_t$, $\B^{N}_t$ and $\K^{N}_t$ in the uniform topology, any limit point of $$(\M^{N}_t, \Y^{N}_t, \Y^{N}_0, \I^{N}_t, \B^{N}_t, \K^{N}_t)$$ is such that $\Y_t$ satisfies the multi-species energy condition \eqref{en.cond} where $\A_\cdot$ has zero quadratic variation.  
 Also, componentwise,
$$\M^i_t(H) \ =\ \Y^i_t(H) - \Y^i_0(H) - \frac{1}{2}\partial_{a^i}\tilde g_i({\BF a_0})\int_0^t\Y^i_s(\triangle H)ds -c\A^i_t(H)$$
is a continuous martingale with respect to the filtration generated by $\Y_\cdot$, with quadratic variation $\tilde g_i(\ab)\|\nabla H\|^2_{L^2}$.  In addition, the cross variations between $\M^i_t$ and $\M^j_t$ vanish.
Moreover, $\Y_0$, as the limit of $\Y^N_0$ is a spatial white noise with covariance \eqref{standard_covariance}.  At this point, we have verified items (i), (ii) and (iii) of Definition \ref{frac_energy_defn} in the specification of the multi-scale energy solution.

We now show item (iv) of Definition \ref{frac_energy_defn}.  The time-reversed process $\hat\a_\cdot = \a_{T-\cdot}$, as noted in the Subsection \ref{Markov}, is a multi-species process with reversed jump probabilities $\hat p(\cdot) = p(-\cdot)$ with now drift $-c/(2\sqrt{N})$.  Define $\hat\Y^N_\cdot = \Y^N_{T-\cdot}$.  All the analysis with respect to the `forward' process done previously applies also to $\hat \Y^N_\cdot$ and to the objects $\hat \M^N_\cdot = \M^N_{T-\cdot}$, $\B^N_\cdot = \B^N_{T-\cdot}$, $\hat \I^N_\cdot = \I^N_{T-\cdot}$ and $\hat\K^N_\cdot = \K^N_{T-\cdot}$.  The process $\hat \Y^N_\cdot$ will satisfy the multi-species energy condition \eqref{en.cond} with respect to a process $\hat\A_\cdot$.  One may compute therefore for a limit point and $1\leq i\leq n$ that 
$$\hat \M^i_t(H) \ =\ \hat \Y^i_t(H) - \hat \Y^i_0(H) - \frac{1}{2}\partial_{a^i}\tilde g_i({\BF a_0})\int_0^t \hat \Y^i_s(\triangle H)ds + c\hat \A^i_t(H).$$
Finally, to complete the item, we note by explicit computation that
\begin{align*}
\lim_{\varepsilon \to 0} \; \hat \A^{i,\varepsilon}_t(H) 
&= \lim_{\varepsilon \to 0}\;  \sum_{j,k}\partial_{a^j}\partial_{a^k}\tilde g_i({\BF a_0})
\int_0^t\int \nabla H(u)  \tau_{-u}\hat\Y^j_s(\iota_\e) \tau_{-u}\hat\Y^k_s(\iota_\e)  du ds\\
&=\lim_{\varepsilon \to 0} \; \sum_{j,k}\partial_{a^j}\partial_{a^k}\tilde g_i({\BF a_0})
\int^T_{T-t}\int \nabla H(u)
\tau_{-u}\Y^j_s(\iota_\varepsilon) \tau_{-u }\Y^k_s(\iota_\varepsilon) du ds \\
& = \ \A^i_T(H)-\A^i_{T-t}(H).
\end{align*}
Therefore, by the unique characterization of the multi-species energy solution in Definition \ref{frac_energy_defn}, stated in Remark \ref{GPrem}, we finish the proof of Theorem \ref{secondorder_thm}. \qed

\section{Proof of Theorem \ref{gbg_L2}:  Boltzmann-Gibbs principle}
\label{proof_Boltzmann}

We adapt the proof of Theorem 3.2 in \cite{GJS}  and Theorem 3.1 in \cite{S_longrange} to the multi-species setting.  For the convenience of the reader, although some of the arguments are similar, we give the pertinent details.  We first recall some preliminary notions and estimates, before proving Theorem \ref{gbg_L2} at the end of the Subsection \ref{subsec:BG-0}.

We will need an `equivalence of ensembles' estimate for the zero-range invariant measure $\nu_\ab$, proved in Subsection \ref{ee_section}.  We remark the density ${\BF a_0}$ here is generic and need not satisfy the Frame condition (FC).

In the following, for $\zeta \in \Omega$ or $\zeta\in \Omega^n$, we will abbreviate $\zeta^{(\ell)}(0) = \zeta^{(\ell)}$ for $\ell\geq 1$.  Recall the covariance of $\alpha^j$ and $\alpha^k$, under $\nu_\ab$, is denoted by $\Gamma_{j,k} (\ab)$ for $1\leq j,k\leq n$.

\subsection{Proof of the Boltzmann-Gibbs principle}
\label{subsec:BG-0}

\begin{theorem}[equivalence of ensembles]\label{2EE}
 Let $f:\Omega^n\rightarrow \R$ be a local $L^5(\nu_{\ab})$ function, supported on sites in $\Lambda_{\ell_0}$, such that $\tilde f(\ab)=0$ and $ \nabla \tilde f(\ab)=0$.  Then,
there exists a constant $C= C(\ab,\ell_0,n)$ such that for $\ell\geq\ell_0$ we have
\begin{align*}
&\Big\| E_{\nu_{\ab}}\big[f(\BF\alpha)| \alpha^{i,(\ell)}: 1\leq i\leq n\big] \\
&\ \ \ \ \ \ \ \ \ \ \ \ \ \ \  - \sum_{j,k=1}^n \Big\{(\alpha^{j,(\ell)}-a^j_0)(\alpha^{k,(\ell)}-a^k_0)- \frac{\Gamma_{j,k} (\ab) }{2\ell +1}\Big\} \partial_{a^j}\partial_{a^k} \tilde f(\ab)
\Big\|_{L^4(\nu_\ab)}\\
&\ \ \ \ \ \  \ \leq \ \frac{C\|f\|_{L^5(\nu_{\ab})}}{\ell^{3/2}}.
\end{align*}

On the other hand, when only $\tilde f(\ab)=0$ is known,
\begin{align*}
&\Big\| E_{\nu_{\ab}}[f({\BF \alpha})| \alpha^{i,(\ell)}: 1\leq i\leq n ]  - \sum_{j=1}^n \big \{ \alpha^{j,(\ell)} - a^j_0\big\}\partial_{a^j}\tilde f(\ab)  
\Big\|_{L^4(\nu_\ab)} \ \leq \ \frac{C\|f\|_{L^5(\nu_{\ab})}}{\ell}.
\end{align*}
\end{theorem}

\begin{proof}
The theorem is proved in Subsection \ref{ee_section}.
\end{proof}

We now define notions of $H_{1}$ and $H_{-1}$ norms which will be useful.
Let $S_N= (L_N + L^*_N)/2$ be the generator of the process with respect to symmetric nearest-neighbor jump probability $s(\pm 1)=1/2$ and $s(x)=0$ for $x\neq \pm 1$.   Define the $H_{1,N}$ semi-norm $\|\cdot\|_{1,N}$ on local, bounded functions by
$$\|f\|_{1,N}^2 \ := \ E_{\nu_{\BF a_0}}\big[f(-S_N)f\big] \ = \ N^2D_{\nu_{\BF a_0}}(f),$$
where $D_{\nu_{\BF a_0}}(f)$ is the grand-canonical Dirichlet form
$$D_{\nu_{\BF a_0}}(f) \ = \ \frac{1}{2}\sum_{i=1}^n\sum_{\stackrel{|x-y|=1}{x,y\in \T_N}} E_{\nu_{\BF a_0}}\Big[g_i(\BF \alpha (x))\big( \nabla^i_{x,y}f(\BF \alpha )\big)^2 \Big].$$
Here, $\nabla^{i}_{x,y}f(\BF \alpha) := f(\BF \alpha^{x,y;i})-f(\BF \alpha)$ for $1\leq i\leq n$ and $x,y\in \T_N$.

Let $H_{1,N}$ be the Hilbert space consisting of the completion of functions with finite $H_{1,N}$ norm modulo norm-zero functions.  Note that local bounded functions are dense in $H_{1,N}$.
Let $\|\cdot\|_{-1,N}$ be the dual semi-norm with respect to the $L^2(\nu_{\BF a_0})$ inner-product given by
$$\|f\|_{-1,N} \ := \ \sup \left\{ \frac{E_{\nu_{\BF a_0}}[f\phi]}{\|\phi\|_{1,N}}\; ; \;  \phi\neq 0 \ {\rm local, \ bounded \ s.t.\  \| \phi \|_{1,N}>0}\right\}.$$
Denote $H_{-1,N}$ as the Hilbert space corresponding to the completion over those functions with finite $\|\cdot\|_{-1,N}$ norm modulo norm-zero functions.

We now state a multi-species form of Proposition 4.1 in \cite{GJS}.
Denote the $\Lambda_\ell$-restricted grand-canonical Dirichlet form, on local, bounded functions, by
$$D_{\nu_{\BF a_0}, \ell} \, (\phi) \ = \ \frac{1}{4}\sum_{i=1}^n\sum_{\stackrel{|x-y|=1}{x,y\in \Lambda_\ell}} E_{\nu_{\BF a_0}}\Big[ g_i(\BF \alpha (x))\big(\nabla^i_{x,y}\phi(\BF \alpha)\big)^2\Big].$$

Recall the inverse of the spectral gap, $W(\k, \ell)$, defined in Subsection \ref{specgap_section}.

\begin{prop} \label{crucial estimate}
Let $r:\Omega^n \to \mathbb R$ be a local $L^4(\nu_{\BF a_0})$ function, supported on sites in $\Lambda_{\ell_0}$, for $\ell_0\geq 1$.
Suppose that
$E_{\nu_{\BF a_0}}[r\, |\, \BF \alpha^{(\ell_0)}]=0$ a.s. Then, for local, bounded functions $\phi$, we have
\begin{equation*}
\big|E_{\nu_{{\BF a_0}}}[ r(\BF \alpha) \phi(\BF \alpha)]\big| \ \leq \ E_{\nu_{\BF a_0}}\Big[ W\Big(\sum_{x\in \Lambda_{\ell_0}}\BF \alpha (x), \ell_0\Big)^2\Big]^{1/4} \; \|r\|_{L^4(\nu_{\BF a_0})} \;  D_{\nu_{\BF a_0},\ell_0}^{1/2}(\phi).
\end{equation*}
\end{prop}

\begin{proof}
Recall the notation for the canonical process in Subsection \ref{specgap_section}.  
Since
 \begin{align*}
 E_{\nu_{\BF a_0}} \Big[r \; \Big| \; \sum_{|x|\leq \ell_0}\a(x)=\k \Big]=E_{\nu_{\k,\ell_0}}[r]=0,
 \end{align*}
 the restriction $r_\k$ of the function $r$ to ${\mathcal G}_{\k,\ell_0}=\{\BF \alpha: \sum_{x\in \Lambda_{\ell_0}}\BF \alpha (x) = \k\}$ is orthogonal to constant functions.  Hence, $r_\k$ belongs to the range of $-S_{\k,\ell_0}$, and the equality $r_\k = -S_{\k,\ell_0} u$ holds for a function $u$ on ${\mathcal G}_{\k,\ell_0}$.
Consider the random variable $\k_0= \sum_{x\in \Lambda_{\ell_0}}\BF \alpha (x)$ and write
\begin{eqnarray*}
\big|E_{\nu_{\BF a_0}}[r\phi]\big| &=& \big|E_{\nu_{\BF a_0}}\big[ E_{\nu_{\BF a_0}}[r\phi|\k_0 ]\big]\big| =  \big|E_{\nu_{\BF a_0}}\big[ E_{\nu_{\k_0, \ell_0}}[r\phi] \big]\big| = \big|E_{\nu_{\BF a_0}}\big[ E_{\nu_{\k_0, \ell_0}}[r_{\k_0} \phi] \big]\big|  \\
&=&\big|E_{\nu_{\BF a_0}}\big[ E_{\nu_{\k_0, \ell_0}}[ (-S_{\k_0,\ell_0} u) \phi ]\big ]\big|\\
&\leq & E_{\nu_{\BF a_0}}\left[ E_{\nu_{\k_0, \ell_0}}\Big[u (-S_{\k_0,\ell_0} u) \Big]^{1/2}\;   E_{\nu_{\k_0, \ell_0}}\Big[\phi (-S_{\k_0,\ell_0} \phi)\Big]^{1/2}\right].
\end{eqnarray*}
To obtain the last line, as, for any $\k$, $-S_{\k,\G}$ is a nonnegative symmetric operator, we used that it has a square root. Since $W(\k,\ell_0)$ is the reciprocal of the spectral gap for $-S_{\k,\ell_0}$, we have, for any $\k$, that
$$E_{\nu_{\k, \ell_0}}\big[r_\k u \big] \ \leq \ W(\k,\ell_0)\; E_{\nu_{\k, \ell_0} }\big[r_{\k}^2\big].$$
Inputting this estimate into the previous display, we obtain
\begin{eqnarray*}
\big|E_{\nu_\ab}[r\phi]\big|&\leq & E_{\nu_\ab}\Big[ [W (\k_0, \ell_0) ]^{1/2} \; E_{\nu_{\k_0, \ell_0}}\big[r_{\k_0}^2 \big]^{1/2} \; D^{1/2}_{\k_0,\ell_0}(\phi)\Big].
\end{eqnarray*}
We now use Schwarz inequality and the fact that
$$E_{\nu_\ab} \Big[ D_{\k_0, \ell_0} (\phi) \Big] = D_{\nu_\ab, \ell_0} (\phi)$$
to finish the proof.
\end{proof}

To simplify notation, for the rest of the section, we will drop the superscript `$N$' and write $\BF \alpha^N = \BF \alpha$.
  
 We now state a useful estimate on the variance of additive functionals, which reduces a dynamic estimate to a static $H_{-1,N}$-estimate.  A proof is found in Proposition 4.3 in \cite{S_longrange}, with a straightforward change of notation; see also
  Appendix 1.6 in \cite{KL} for a similar estimate.  Recall the translation $\tau_z\a$ given by $\tau_z\BF \alpha (x) = \BF \alpha (x+z)$.  For $c\in {\mathbb R}$ and ${R} \in L^2(\nu_\ab)$, let $r: \T_N\times \Omega \rightarrow \R$ be a function such that 
$\|r(z,\cdot)\|_{-1,N} \leq \|{R}\|_{-1,N}$ for each $z\in \T_N$.

\begin{prop}
 \label{KV}
Let $R:\Omega^n \to \mathbb R$ be a mean-zero $L^2(\nu_\ab)$ function, $\tilde{R}(\ab)=0$.  Then, there exists a universal constant $C$ such that
\begin{equation*}
 \mathbb{E}_{\nu_\ab}\Big[\sup_{0\leq t\leq T}\Big(\int_0^t r(\lfloor cs\rfloor, {\BF \alpha}_s) ds \Big)^2\Big] \ \leq \ C T \|R\|_{-1,N}^2.
\end{equation*}
\end{prop}

\medskip

Recall the definition of $\bar h_{c,s}$, for $c\in {\mathbb R}$ and $s\in [0,T]$, above the statement of Theorem \ref{gbg_L2}.

\begin{lem}
\label{globalone_block}
Let $f:\Omega^n \to \mathbb R$ be a local $L^4(\nu_\ab)$ function
supported on sites in $\Lambda_{\ell_0}$ such that
 $\tilde{f}(\ab)=0$. Then, there exists a constant $C=C(\ab)$ such that, for $\ell \geq \ell_0$, and $h:\T_N\rightarrow{\mathbb R}$,
\begin{eqnarray*}
&&\E_{\nu_\ab}\left[\sup_{0\leq t\leq T}\left(\int_0^t \sum_{x\in\T_N}\bar h_{c,s}(x)\Big\{f(\tau_x \BF \alpha_s)-\E_{\nu_{\ab}}[f(\tau_x\BF \alpha_s)|\BF \alpha_s^{(\ell)}(x)]\Big\} ds \right)^2\right] \\
&&\ \ \ \ \ \ \ \ \ \ \  \leq \ CT\frac{\ell^{3}}{N^2} \|f\|^2_{L^4(\nu_\ab)}\sum_{x\in{\T_N}}h^2(x).
\end{eqnarray*}
\end{lem}

\noindent {\it Proof.}  
We may write the integrand function as follows:
$$\sum_{x\in \T_N} \bar h_{c,s}(x) \Big\{f(\tau_x\BF \alpha)-\E_{\nu_{\ab}}[f(\tau_x\BF \alpha) \, \big|\, \BF \alpha^{(\ell)}(x)]\Big\}  \ = \ r(\lfloor cs\rfloor, \BF \alpha)$$
where
$$r(z,\BF \alpha) \ := \ \sum_{x\in \T_N} h(x) \Big\{f(\tau_{z+x}\BF \alpha_s)-E_{\nu_{\ab}}[f(\tau_{z+x}\BF \alpha_s) \, \big| \, \BF \alpha_s^{(\ell)}(z+x)]\Big\}.$$
By translation invariance of $\nu_\ab$,
we have that $\|r(z,\cdot)\|_{-1,N} = \|r\|_{-1,N}$ where
$r(\BF \alpha) = r(0,\BF \alpha)$.
Hence, by Proposition \ref{KV}, we need only a sufficient estimate of the $H_{-1,N}$ norm of $r$ to finish. Using Proposition \ref{crucial estimate}, bound the $H_{-1,N}$ norm as follows:
\begin{equation}
\label{global_eq_1}
\begin{split}
&\Vert r \Vert_{-1,N} =\sup_\phi \left\{ \frac{E_{\nu_\ab} [r \phi] }{\Vert \phi\Vert_{1,N}} \right\}\\
&=
N^{-1}\  
\sup_\phi \left\{D^{-\frac{1}{2}}_{\nu_\ab}(\phi) \; E_{\nu_\ab}\Big[\sum_{x\in\T_N}h(x)\Big\{f(\tau_x \BF \alpha)-E_{\nu_{\ab}}[f(\tau_x \BF \alpha)|\BF \alpha^{(\ell)}(x)]\Big\}  \phi\Big]\right\}\\
& = 
N^{-1} \  
\sup_\phi\sum_{x\in \T_N}  D^{-\frac{1}{2}}_{\nu_\ab}(\phi) \; E_{\nu_\ab}\Big[h(x)  \big( f(\BF \alpha) - E_{\nu_\ab}[f(\BF \alpha)|\BF \alpha^{(\ell)}]\big) \phi(\tau_{-x}\BF \alpha)\Big]\\
& \leq 
N^{-1} \  \ 
\sup_\phi  D^{-\frac{1}{2}}_{\nu_\ab}(\phi)\sum_{x\in \T_N}  |h(x)| \; E_{\nu_\ab}\Big [ W\Big(\sum_{z\in \Lambda_\ell}\BF \alpha (z), \ell\Big)^2\Big]^{\frac{1}{4}} \; \|f\|_{L^4(\nu_\ab)} \;  D^{\frac{1}{2}}_{\nu_\ab,\ell}(\phi(\tau_{-x}\BF \alpha)).
\end{split}
\end{equation}
By translation-invariance of $\nu_\ab$, and counting the number of repetitions,
$$\sum_{x\in \T_N}  D_{\nu_\ab,\ell}(\phi(\tau_{-x}\BF \alpha)) \  \leq \ (2\ell +1)  D_{\nu_\ab}(\phi).$$
Then, by the spectral gap assumption (SG) in Subsection \ref{specgap_section}, and $2ab = \inf_{\kappa>0}[\kappa a^2 + \kappa^{-1}b^2]$, we bound \eqref{global_eq_1} by a constant times 
\begin{eqnarray*}
&&
N^{-1} \ 
\sup_\phi D^{-\frac{1}{2}}_{\nu_\ab}(\phi)\inf_{\kappa>0}\Big\{\kappa \ell^2\|f\|^2_{L^4(\nu_\ab)}\sum_{x\in \T_N} h^2(x) + \kappa^{-1} \ell D_{\nu_\ab}(\phi)\Big\}\\
&&\ \  \ \ \ \ \ \lesssim \left(\cfrac{\ell^{3}}{N^2}\,   \|f\|_{L^4(\nu_\ab)}^2\sum_{x\in \T_N} h^2(x)\right)^{\frac{1}{2}}. \quad \quad \quad \quad \quad \quad \quad \quad \quad \quad \quad \quad \quad \quad \quad \qed
\end{eqnarray*}

We actually use Lemma \ref{globalone_block} only for $\ell=\ell_0$.The size of the box in the conditional expectation is now doubled in the next lemma, which is used for the proof of Lemma \ref{globaltwo-blocks}. 

\begin{lem}
\label{globalrenormalization}
Let $f:\Omega^n \to \mathbb R$ be a local $L^5(\nu_\ab)$ function 
supported on sites in $\Lambda_{\ell_0}$ 
such that $\tilde f(\ab)=0$ and $\nabla \tilde f(\ab)=0$. There exists a constant $C=C(\ab, \ell_0,n)$ such that, for $\ell \geq \ell_0$, and 
$h:\T_N\rightarrow {\mathbb R}$,
\begin{eqnarray*}
&&\E_{\nu_\ab}\left[\sup_{0\leq t\leq T}\Big(\int_0^t\sum_{x\in{\T_N}} \bar h_{c,s}(x)\right.\\
&&\left. \quad\quad \quad \quad  \times \Big\{ \E_{\nu_{\ab}}[f(\tau_x\BF \alpha_s)|\BF \alpha_s^{(\ell)}(x)] -\E_{\nu_{\ab}}[f(\tau_x \BF \alpha_s)|\BF \alpha_s^{(2\ell)}(x)]\Big\} ds \Big)^2\right] \\
&&\ \ \ \ \ \ \ \ \ \ \ \ \ \leq \ CT\|f\|^2_{L^5(\nu_\ab)}\frac{\ell}{N^2}\sum_{x\in{\T_N}}h^2(x).
\end{eqnarray*}

On the other hand, when only $\tilde f(\ab)=0$ is known,
\begin{eqnarray*}
&&\E_{\nu_\ab}\Big[\sup_{0\leq t\leq T}\Big(\int_0^t\sum_{x\in{\T_N}} \bar h_{c,s}(x)\\
&&\ \ \ \ \ \ \ \ \ \ \ \ \ \ \ \ \ \ \ \ \ \ \ \ \ \ \  \times \Big\{E_{\nu_{\ab}}[f(\tau_x\BF \alpha_s)|\BF \alpha_s^{(\ell)}(x)]-E_{\nu_{\ab}}[f(\tau_x\BF \alpha_s)|\BF \alpha_s^{(2\ell)}(x)]\Big\} ds \Big)^2\Big] \\
&&\ \ \ \ \ \ \ \ \ \ \ \ \ \leq \ CT\|f\|^2_{L^5(\nu_\ab)}\frac{\ell^{2}}{N^2}\sum_{x\in{\T_N}}h^2(x).
\end{eqnarray*}
\end{lem}

\begin{proof}  We prove only the first statement since the second has a similar argument.  Define the sigma-field $\F_\ell = \sigma\{\BF \alpha^{(\ell)}, \BF \alpha^c_\ell\}$ where $\BF \alpha^c_\ell = \{\BF \alpha (x): x\not\in \Lambda_\ell\}$ for $\ell\geq 1$.  Since $f$ is supported on sites $\Lambda_{\ell_0}$, and $\nu_\ab$ is a product measure, we have that
\begin{equation*}
E_{\nu_\ab}\big[f(\BF \alpha)\big|\BF \alpha^{(m)}\big] =E_{\nu_\ab}\big[f(\BF \alpha)\big|\BF \alpha^{(m)}, \BF \alpha^c_m\big]
\end{equation*}
and since $\F_{2\ell}\subset \F_\ell$, $\ell\geq \ell_0$ we get that
\begin{eqnarray*}
E_{\nu_\ab}\Big[E_{\nu_\ab}\big[f(\BF \alpha)\big|\BF \alpha^{(\ell)}\big]\big|\BF \alpha^{(2\ell)}\Big] &=&
E_{\nu_\ab}\left[E_{\nu_\ab}\big[f(\BF \alpha)\big|\BF \alpha^{(\ell)}, \BF \alpha^c_\ell\big]\, \Big|\, \BF \alpha^{(2\ell)}, \BF \alpha^c_{2\ell}\right] \\
& = & E_{\nu_\ab}\big[f(\BF \alpha)\big| \BF \alpha^{(2\ell)}, \BF \alpha^c_{2\ell}\big]\ = \ 
E_{\nu_\ab}\big[f(\BF \alpha)\big| \BF \alpha^{(2\ell)}\big].\end{eqnarray*}

We now follow similar steps as in the proof of Lemma \ref{globalone_block} to the last line, with 
$$r(z,\BF \alpha) \ = \ \sum_{x\in \T_N} h(x)  \Big\{E_{\nu_{\ab}}[f(\tau_{z+x}\BF \alpha)|\BF \alpha^{(\ell)}(z+x)]-E_{\nu_{\ab}}[f(\tau_{z+x}\BF \alpha)|\alpha^{(2\ell)}(z+x)]\Big\}$$
and $r(\BF \alpha)=r(0,\BF \alpha)$.
To finish, we claim that the following variance is bounded as 
$$
\Big\|E_{\nu_{\ab}}[f(\BF \alpha)|\BF \alpha^{(\ell)}]-E_{\nu_{\ab}}[f(\BF \alpha)|\BF \alpha^{(2\ell)}]\Big\|_{L^4(\nu_\ab)}^2 \ \lesssim \|f\|^2_{L^5(\nu_\ab)}\ell^{-2}.
$$
Indeed, by the inequality $(a+b+c)^2 \leq 3a^2+3b^2 +3c^2$, the variance is bounded by
\begin{eqnarray*}
&& 3\Big\|E_{\nu_{\ab}}\Big[f(\BF \alpha) - \tfrac{1}{2}\sum_{j,k}\partial_{a^j}\partial_{a^k}\tilde f(\ab)\Big\{(\alpha^{j,(\ell)} - a^j_0)(\alpha^{k,(\ell)}-a^k_0)\\
&&\ \ \ \ \ \ \ \ \ \ \ \ \ \ \ \ \ \ \ \ \ \ \ \ \ \ \ \ \ \  -\tfrac{\Gamma_{j,k} (\ab) }{2\ell +1}\Big\}\Big|\a^{(\ell)}\Big]\Big\|^2_{L^4(\nu_\ab)}\\
&&   +3\Big\|E_{\nu_{\ab}}\Big[f(\a)
 - \frac{1}{2}\sum_{j,k}\partial_{a^j}\partial_{a^k}\tilde f(\ab) \Big\{(\alpha^{j,(2\ell)} - a^j_0)(\alpha^{k,(2\ell)}-a^k_0) \\
 &&\ \ \ \ \ \ \ \ \ \ \ \ \ \ \ \ \ \ \ \ \ \ \ \ \ \ \ \ \ \ -\frac{\Gamma_{j,k} (\ab)}{2(2\ell)+1}\Big\}\Big|\BF \alpha^{(2\ell)}\Big]\Big\|^2_{L^4(\nu_\ab)}\\
&&  + 3\Big\|\tfrac{1}{2}\sum_{j,k}\partial_{a^j}\partial_{a^k}\tilde f(\ab)\Big\{E_{\nu_\ab}\Big[(\alpha^{j,(\ell)} - a^j_0)(\alpha^{k,(\ell)}-a^k_0) -\tfrac{\Gamma_{j,k} (\ab)}{2\ell +1}\Big|\BF \alpha^{(\ell)}\Big] \\
&&\ \ \ \ \ \ \ \ \ \ \ \ \ \ \ \ \
+ E_{\nu_\ab}\Big[(\alpha^{j,(2\ell)} - a^j_0)(\alpha^{k,(2\ell)}-a^k_0) -\tfrac{\Gamma_{j,k} (\ab)}{2(2\ell)+1}\Big|\BF \alpha^{(2\ell)}\Big]\Big\}
\Big\|_{L^4(\nu_\ab)}^2.
\end{eqnarray*}
The first two terms are bounded by Theorem \ref{2EE} of order $O(\|f\|^2_{L^5(\nu_\ab)}\ell^{-3})$.  However, the last term, by a fourth moment bound of $(\alpha^{q,(m)}-a^q_0)^2$ with variously $m=\ell$ and $m=2\ell$ and using that $|\partial_{a^j}\partial_{a^k}\tilde f(\ab)| \lesssim \|f\|_{L^2(\nu_\ab)}$  is of order $O(\|f\|^2_{L^2(\nu_\ab)}\ell^{-2})$.

Combining the two estimates, we obtain that the variance is of order $O(\|f\|^2_{L^5(\nu_\ab)}\ell^{-2})$ as desired.  Replacing now $\|f\|^2_{L^4(\nu_\ab)}$ in the last line of Lemma \ref{globalone_block} by this variance estimate, one recovers the bound in the first statement.
 \end{proof}

The next lemma replaces $\ell_0$ with $\ell \ge \ell_0$ and Lemma \ref{EE_1block} will apply for large $\ell$. 

\begin{lem}
\label{globaltwo-blocks}
Let $f:\Omega^n \to \mathbb R$ be a local $L^5(\nu_\ab)$ function 
supported on sites in $\Lambda_{\ell_0}$ 
such that $\tilde f(\ab)=0$ and $\nabla \tilde f(\ab)=0$.  Then, there exists a constant $C = C(\ab,\alpha,\ell_0,n)$ such that, for $\ell \geq  \ell_0$,  and 
$h:\T_N\rightarrow {\mathbb R}$,
\begin{eqnarray*}
&&\E_{\nu_\ab}\Big[\sup_{0\leq t\leq T}\Big(\int_0^t\sum_{x\in{\T_N}}\bar h_{c,s}(x)\Big\{E_{\nu_{\ab}}[f(\tau_x\BF \alpha)|\BF \alpha^{(\ell_0)}(x)]
-E_{\nu_{\ab}}[f(\tau_x\BF \alpha)|\BF \alpha^{(\ell)}(x)]\Big\} ds \Big)^2\Big] \\
&&\ \ \ \ \ \ \ \ \ \ \ \ \ \leq CT\|f\|^2_{L^5(\nu_\ab)}\frac{\ell}{N^2}\sum_{x\in{\T_N}}h^2(x).
\end{eqnarray*}
On the other hand, when only $\tilde f(\ab)=0$ is known,
\begin{eqnarray*}
&&\E_{\nu_\ab}\Big[\sup_{0\leq t\leq T}\Big(\int_0^t\sum_{x\in{\T_N}}\bar h_{c,s}(x)\Big\{E_{\nu_{\ab}}[f(\tau_x\BF \alpha)|\BF \alpha^{(\ell_0)}(x)]-E_{\nu_{\ab}}[f(\tau_x\BF \alpha)|\BF \alpha^{(\ell)}(x)]\Big\} ds \Big)^2\Big] \\
&&\ \ \ \ \ \ \ \ \ \ \ \ \ \leq CT\|f\|^2_{L^5(\nu_\ab)}\frac{\ell^2}{N^2}\sum_{x\in{\T_N}}h^2(x).
\end{eqnarray*}
\end{lem}

\begin{proof}  As the second statement has a similar proof, we only demonstrate the first display. Write $\ell = 2^{m+1}\ell_0 + r$ where $0\leq r\leq 2^{m+1}\ell_0 -1$.  Then,
\begin{eqnarray*}
&&E_{\nu_{\ab}}[f(\BF \alpha)|\BF \alpha^{(\ell_0)}]-E_{\nu_{\ab}}[f(\BF \alpha)|\BF \alpha^{(\ell)}]
\  = \ E_{\nu_\ab}[f(\BF \alpha)|\BF \alpha^{(2^{m+1}\ell_0)}]  - E_{\nu_\ab}[f(\BF \alpha)|\BF \alpha^{(\ell)}] \\
&& \ \ \ \ \ \ \ \ \ \ + \sum_{i=0}^{m} \big \{E_{\nu_{\ab}}[f(\BF \alpha)|\BF \alpha^{(2^i \ell_0)}]-E_{\nu_{\ab}}[f(\BF \alpha)|\BF \alpha^{(2^{i+1}\ell_0)}]\big\}.
\end{eqnarray*}
Now, by Minkowski's inequality and Lemma \ref{globalrenormalization}, we obtain
that the left-side of the display in the lemma statement is bounded by a constant times
\begin{eqnarray*}
&&\left\{ \Big(\frac{T 2^{m+1}\ell_0}{N^2}\Big)^{1/2} + \sum_{i=0}^{m} \Big(\frac{T 2^{i}\ell_0}{N^2}\Big)^{1/2} \right\}^2 \|f\|^2_{L^5(\nu_\ab)}\sum_{x\in \T_N} h^2(x)\\
&& \ \lesssim \ 
 \frac{T \ell \, \|f\|^2_{L^5(\nu_\ab)} }{N^2}\sum_{x\in \T_N} h^2(x),
 \end{eqnarray*}
to finish the proof.
\end{proof}
The last step is an `equivalence of ensembles' estimate.
\begin{lem}
\label{EE_1block}
Let $f:\Omega^n \to \mathbb R$ be a local $L^5(\nu_\ab)$ function supported on sites in $\Lambda_{\ell_0}$ such that $\tilde f(\ab)=0$ and $\nabla \tilde f(\ab)=0$.  Then, there exists a constant $C = C(\ab,\ell_0,n)$ 
such that, for $\ell \geq  \ell_0$, and $h:\T_N\rightarrow  {\mathbb R}$,
\begin{eqnarray*}
&&\E_{\nu_\ab}\Big[\sup_{0\leq t\leq T}\Big(\int_0^t\sum_{x\in{\T_N}}\bar h_{c,s}(x)\Big\{\E_{\nu_{\ab}}\big[f(\tau_x\BF \alpha_s)|\BF \alpha_s^{(\ell)}(x)\big]\\
&&\ \ \ \ \ \ \ 
-\frac{1}{2}\sum_{j,k}\partial_{a^j}\partial_{a^k}\tilde f(\ab) \Big((\alpha_s^{j,(\ell)}(x) - a^j_0)(\alpha_s^{k,(\ell)}-a^k_0) - \frac{\Gamma_{j,k} (\ab) }{2\ell +1}\Big)\Big\} ds \Big)^2\Big] \\
&&\ \ \ \ \ \ \ \ \ \leq \ CT^2\|f\|^2_{L^5(\nu_\ab)}\frac{N^2}{\ell^{3}}\Big(\frac{1}{N}\sum_{x\in \T_N}|h(x)|\Big)^2\end{eqnarray*}
On the other hand, when only $\tilde f(\ab)=0$ is known,
\begin{eqnarray*}
&&\E_{\nu_\ab}\Big[\sup_{0\leq t\leq T}\Big(\int_0^t\sum_{x\in{\T_N}} \bar h_{c,s}(x)\Big\{\E_{\nu_{\ab}}[f(\tau_x\BF\alpha_s)|\BF\alpha_s^{(\ell)}(x)]\\
&&\ \ \ \ \ \ \ \ \ \ \ \ \ \ \ \ \ \ \ \ \ \ \ \ \ \ \ \ \ \ \ \  \ \ \ \ \ \ \ -\sum_i\partial_{a^i}\tilde f(\ab)\big(\alpha_s^{i,(\ell)}(x)-a^i_0\big) \Big\} ds\Big)^2\Big] \\
&&\ \ \ \ \ \ \ \ \ \leq \ CT^2\|f\|^2_{L^5(\nu_\ab)}\frac{N^2}{\ell^{2}}\Big(\frac{1}{N}\sum_{x\in \T_N}|h(x)|\Big)^2.\end{eqnarray*}
\end{lem}

\begin{proof} By squaring and using stationarity and translation-invariance of $\nu_\ab$, the left-hand side of each display is bounded by
$$\E_{\nu_\ab} \Big[\Big(\int_0^T \sum_{x\in \T_N}|\bar h_{c,s}(x)| |\psi(x,\BF \alpha_s)| ds\Big)^2\Big] \ \leq \ T \int_0^T E_{\nu_\ab}\Big[\Big(\sum_{x\in \T_N} |\bar h_{c,s}(x)||\psi(x,\BF \alpha)|\Big)^2\Big] ds$$
where $\psi(x,\BF \alpha)$ is the expression in curly braces. 
Now, for each $\BF \alpha$, by Schwarz inequality,
$$\Big(\sum_{x\in \T_N} |\bar h_{c,s}(x)|\psi(x, \BF \alpha)\Big)^2 \ \leq \ \Big(\sum_{x\in \T_N} |\bar h_{c,s}(x)|\Big)\sum_{x\in \T_N} |\bar h_{c,s}(x)|\psi^2(x, \BF \alpha).$$
Since $\nu_\ab$ is translation-invariant, $E_{\nu_\ab}[\psi^2(x,\BF \alpha)] = E_{\nu_\ab}[\psi^2(0,\BF \alpha)]$. Moreover, $\sum_x | \bar h_{c,s}(x)| =\sum_x |h(x)|$.  The desired bound now follows, noting the form of $\psi(0,\BF \alpha)$, from Theorem \ref{2EE}. 
\end{proof}

\medskip
\noindent {\BF{Proof of Theorem \ref{gbg_L2}}.}  With the above ingredients in place, the estimate follows now by first applying Lemma \ref{globalone_block} with $\ell = \ell_0$, and then Lemmas \ref{globaltwo-blocks} and \ref{EE_1block}.
\qed

\subsection{Proof of Theorem \ref{2EE}: Equivalence of ensembles}
\label{ee_section}

We will adapt the proof of Proposition 5.1 in \cite{GJS} to the multi-species setting.
We prove the first display in Theorem \ref{2EE} as the second statement, following the same scheme, has a simpler argument.  At the expense of the constant, we need only to consider all large $\ell>\ell_0$.  To simplify expressions, we will make averages over $\Lambda^+_{m} = \{x: 1\leq x\leq m\}$, instead of the two-sided block $\Lambda_m$, but we will keep similar notations for the averages.  We recall that the density $\ab$ is generic and need not satisfy the Frame condition (FC).

The argument follows in several steps.

\medskip
{\it Step 1.}
Recall the definition of the covariance matrix $\Gamma:=\Gamma(\z)$ where $\Gamma_{i,j}(\z) = E_{\nu_\z}[(\alpha^i(0)-z_i)(\alpha^j(0)-z_j)]$ for $1\leq i,j\leq n$ and $\z \in (0, \infty)^n \cap \ocirc{{\rm{Dom}}_\Phi}$.
 Note also the canonical expectation $E_{\nu_\z}[f (\BF \alpha) |{\BF \alpha}^{(\ell)} = \y]$, does not depend on $\z$, and that we are free to choose it as desired. Develop
\begin{eqnarray*}
E_{\nu_\z}[f (\BF \alpha) |{\BF \alpha}^{(\ell)} = \y] &=&  E_{\nu_{\y+{\ab}}}\Big[ f({\BF \alpha})\,  \Big| \, \tfrac{1}{\ell} \sum_{x\in \Lambda^+_\ell}{\BF \alpha} (x)-{\ab} = \y\Big]\\
&=& \frac{E_{\nu_{\y+{\ab}}}\big[f(\a)1(\frac{1}{\ell}\sum_{x\in \Lambda^+_\ell} \a(x)-{\ab} = \y)\big]}{\nu_{\y+{\ab}}\big(\frac{1}{\ell}\sum_{x\in \Lambda^+_\ell}\a(x)-{\ab} = \y\big)}.
\end{eqnarray*}
Define for $m\ge 1$,  $\theta_m(\z) = {m}^{n/2}\nu_{\y+{\ab}}\big(\sum_{x\in \Lambda^+_m}(\alpha(x)-\ab-\y)=\z\big)$,
and since $f$ is a function supported only on sites in $\Lambda^+_{\ell_0}$, write the last expression as
$$E_{\nu_{\y+{\ab}}}\left [ f(\BF \alpha) \frac{{\ell}^{n/2}\theta_{\ell-\ell_0}(-\sum_{x\in \Lambda^+_{\ell_0}} (\BF \alpha (x)-\y-{\ab}))}{(\ell-\ell_0)^{n/2}\theta_\ell(0)}\right].$$
The goal will be now to expand $\theta_{\ell-\ell_0}(\z)$ to recover the main terms approximating the conditional expectation $E_{\nu_{\ab}}[f (\BF \alpha) |{\BF \alpha}^{(\ell)} = \y]$ 
when $\|\y \|$ is small.  We will treat the case when $\|\y \|$ is bounded away from $0$ afterwards in Step 8.

\medskip
{\it Step 2.} To expand $\theta_m(\z)$,
let $\psi_\y: \t \in \R^n \to \psi_\y (\t)= E_{\nu_{\y+{\ab}}}[e^{i\t\cdot(\BF \alpha (0)-{\ab}-\y)}]$ be the characteristic function of $\BF \alpha (0) -\ab -\y$ under $\nu_{\y +\ab}$.  Then, one can write
\begin{eqnarray*}
\theta_m(\z) &=& \frac{{m}^{n/2}}{(2\pi)^n} \int_{[-\pi,\pi]^n} e^{-i\t\cdot \z}\psi_\y^m(\t)d\t\\
&=& \frac{1}{(2\pi)^n}\int_{[-\pi\sqrt{m}, \pi\sqrt{m}]^n} e^{-i\t\cdot \z/\sqrt{m}}\psi^m_\y(\t/\sqrt{m})d\t.
\end{eqnarray*}
By Taylor expansion,
\begin{equation}
\label{Taylor}
\begin{split}
(2\pi)^n \theta_m(\z) & = \int_{[-\pi \sqrt{m}, \pi \sqrt{m}]^n} \psi_\y^m(\t/\sqrt{m}) d\t \\
&\ \ \   - \int_{[-\pi \sqrt{m}, \pi \sqrt{m}]^n} \frac{i\t\cdot \z}{\sqrt{m}} \psi^m_\y(\t/\sqrt{m}) d\t \\
& \ \ \  - \frac{1}{2}\int_{[-\pi \sqrt{m},\pi \sqrt{m}]^n} \frac{(\t\cdot\z)^2}{m} \psi_\y^m(\t/\sqrt{m}) d\t \\
&\ \ \ + O\Big(\frac{|\z|^3}{m^{3/2}}\Big) \int_{[-\pi \sqrt{m},\pi \sqrt{m}]^n} |\t|^3 |\psi_\y^m(\t/\sqrt{m})|d\t.
\end{split}
\end{equation}
\medskip

{\it Step 3.}
Let $\delta>0$ be sufficiently small such that the ball with radius $\delta$ around $\ab$ is contained in the allowable densities $\ocirc{{\rm{Dom}}}_\Phi\cap(0,\infty)^n$. Let also $0<\varepsilon\leq \pi$ and ${\BF 1}= (1,1,\ldots, 1)\in \R^n$.
First, for $\varepsilon>0$,
$$\sup_{|\y|\leq \delta, \varepsilon\leq |t|\leq \pi}|\psi_\y^m(\t)|< C_0^m$$ where $C_0<1$. To prove this, write 
\begin{eqnarray*}
&&|\psi_\y(\t)| \ \leq \  |\nu_{\y+{\ab}}(\BF \alpha (0)={\BF 0}) + e^{it\cdot {\BF 1}}\nu_{\y+{\ab}}(\BF \alpha (0)={\BF 1})| + \sum_{k\neq \BF 0, {\BF 1}}\nu_{\y+{\ab}}(\BF \alpha (0)=\k)\\
&&\leq \Big(A^2 -2\nu_{\y+{\ab}}(\BF \alpha (0)={\BF 0})\nu_{\y+{\ab}}(\BF \alpha (0)={\BF 1})[1-\cos(\t\cdot{\BF 1})]\Big)^{1/2}
 + 1-A
\end{eqnarray*}
where $A = \nu_{\y+{\ab}}(\a(0)={\BF 0}) + \nu_{\y+{\ab}}(\a(0)={\BF 1})$.
By continuity of $\nu_{\y+{\ab}}(\BF \alpha (0)=\k)$ in $\y$, we have $0<\nu_{\y+{\ab}}(\BF \alpha (0)=\k)<1$ for $\k={\BF 0},{\BF 1}$ uniformly for $|\y|\leq \delta$.  Hence,
uniformly over $\varepsilon\leq |\t|\leq \pi$ and $|\y|\leq \delta$, the right hand side of the display above is strictly bounded by a constant $C_0<1$.

Second, for $0\leq |\t/\sqrt{m}|<\varepsilon$ and $|\y|\leq \delta$,
$$\psi_\y^m(\t/\sqrt{m}) \ = \ \Big[1-\frac{1}{2m}\sum_{j,k} t_jt_k \Gamma_{j k} (\ab+\y) + O\big(C(\ab,\delta)|\t|^3m^{-3/2})\Big]^m.$$
Then, by using the smoothness of $\y \to \Gamma (\ab +\y)$,
$|\psi_\y^m(\t/\sqrt{m})|\leq e^{-C_1(\delta,\varepsilon)\t\cdot \Gamma(\ab)\t}$.  Here, $C_1>0$ and we recall the covariance matrix $\Gamma(\ab)$ is positive-definite, symmetric, and invertible.

Last, by the classical local limit theorem, $\lim_{m\uparrow\infty}\theta_m({\BF 0}) = (2\pi)^{-n/2} \sqrt{ {\rm det} \, \Gamma^{-1}({\ab} +\y)}$.

\medskip

{\it Step 4.}
We now observe, for $|\y|\leq \delta$ and $m\geq 1$, as a consequence of the estimates in Step 3, that the integral in the last term in \eqref{Taylor} is uniformly bounded:  Split the integral over the ranges $|\t/\sqrt{m}|<\varepsilon$ and $|\t/\sqrt{m}|\geq \varepsilon$ and bound each part separately.  Hence, the last term in \eqref{Taylor} is of order $O(|\z|^3/m^{3/2})$. 

Similarly, we split the second integral in \eqref{Taylor}, when $|\y|\leq \delta$, over ranges $|\t/\sqrt{m}|\geq \varepsilon$ and $|\t/\sqrt{m}|< \varepsilon$.  On the first range, the restricted integral exponentially decays, and on the range $|\t/\sqrt{m}|<\varepsilon$, 
the integrand is almost an odd function as
$$\psi_\y^m(\t/\sqrt{m}) \ = \ \left(1-\frac{\t\cdot \Gamma(\y+{\ab})\t }{2m}\right)^m \Big[1+ O(C(\delta)|\t|^3m^{-1/2})\Big].$$
Then, we conclude that the second term in \eqref{Taylor} is of order $O(|\z|/m^{1/2})$.

\medskip
{\it Step 5.}  
Then, for $|\y|\leq \delta$, we have
\begin{eqnarray*}
&&E_{\nu_{{\ab}}}[f(\BF\alpha)|{\BF \alpha}^{(\ell)} = \y] = \kappa_0 (\ell) \, E_{\nu_{\y+{\ab}}}[f(\BF \alpha)] \\
&&+ \frac{1}{\sqrt{\ell-\ell_0}}\sum_{j=1}^n\kappa^j_1 (\ell)\, E_{\nu_{\y+{\ab}}}\big[ f(\BF \alpha)\big(\sum_{\z\in \Lambda^+_{\ell_0}}\alpha^j(\z)-a^j_0-y^j\big)\big]\\
&& + \frac{1}{\ell-\ell_0}\sum_{j,k=1}^n \kappa^{j,k}_2 (\ell) \, E_{\nu_{\y+{\ab}}}\big[ f(\BF \alpha)\big(\sum_{\z\in \Lambda^+_{\ell_0}}\alpha^j(\z)-a^j_0-y^j\big)\big(\sum_{\z\in \Lambda^+_{\ell_0}}\alpha^k(\z)-a^k_0-y^k\big)\big] \\
&&+ \varepsilon_f(\ell),
\end{eqnarray*}
where 
$$|\varepsilon_f(\ell)| \leq C({\ab},\ell_0,\delta,n)\|f\|_{L^1 (\nu_{\ab +\y})}\ell^{-3/2} \le C^\prime ({\ab},\ell_0,\delta,n)\|f\|_{L^2 (\nu_{\ab})}\ell^{-3/2}$$ 
and $\kappa_p:=\kappa_p (\ell)$ for $p=0,1,2$ are explicit expressions. Indeed, from the discussion in Step 4, one observes
\begin{equation*}
\begin{split}
&\kappa_0(\ell) = \frac{{\ell}^{n/2}}{(\ell-\ell_0)^{n/2}}\frac{\theta_{\ell-\ell_0}({\BF 0})}{\theta_\ell({\BF 0})}  \ = \ 1 + O(\ell^{-1/2}), \\
&\kappa^j_1(\ell) = \frac{\ell^{n/2}}{\theta_\ell({\BF 0})(\ell-\ell_0)^{n/2}}\frac{1}{(2\pi)^{n}}\int_{[-\pi\sqrt{\ell-\ell_0}, \pi\sqrt{\ell-\ell_0}]^n} i \, t^j \, \psi_\y^{\ell-\ell_0}\Big(\frac{\t}{\sqrt{\ell-\ell_0}}\Big)d\t \ = \ O(\ell^{-1/2}), \\
&\kappa^{j,k}_2(\ell) = \frac{-{\ell}^{n/2}}{2\theta_\ell({\BF 0}){(\ell-\ell_0)}^{n/2}}\frac{1}{(2\pi)^{n}}\int_{[-\pi\sqrt{\ell-\ell_0}, \pi\sqrt{\ell-\ell_0}]^n} \, t^j \, t^k\, \psi_\y^{\ell-\ell_0}\Big(\frac{\t}{\sqrt{\ell-\ell_0}}\Big)d\t.
\end{split}
\end{equation*}
When $\ell \to \infty$ we may evaluate $-\kappa_2^{j,k} (\ell)$ as
$$(2\pi)^{n/2}\left(\sqrt{{\rm det}\; \Gamma^{-1}(\y + \ab)}\right)^{-1}\frac{1}{(2\pi)^n}\int_{\R^n} \,  t^j \, t^k\,  e^{-\t\cdot \Gamma(\y+\ab)\t}d\t \; +\;  O(\ell^{-1/2}).$$
Since $\Gamma = \big(\Gamma^{-1}\big)^{-1}$ and  
$$(2\pi)^{-n/2}\left(\sqrt{{\rm det}\, \Gamma^{-1} (\y +\ab) } \right)^{-1}\int_{\R^n} \,  t^j \, t^k\,  e^{-\t\cdot \Gamma(\y+\ab)\t}d\t = \left[\Gamma^{-1}(\y+\ab) \right]_{j k},$$
 manipulating the above display, we have
\begin{equation}
\label{eq:kappa2jr}
\kappa_2^{j,k} (\ell) = -\left[ \Gamma^{-1}(\y+\ab)\right]_{j k} + O(\ell^{-1/2}).
\end{equation}

{\it Step 6.} We now develop expansions of $E_{\nu_{\y+{\ab}}}[h (\BF \alpha)]$ for a local $L^2(\nu_{\ab})$ function $h$ supported on coordinates in $\Lambda^+_{\ell_0}$.
Through `tilting', we have 
$$E_{\nu_{\y+{\ab}}}[h (\BF \alpha) ] \ = \ E_{\nu_{\ab}}\left[ h(\BF \alpha)
  \frac{e^{\BF \lambda (\y+{\ab})\cdot \sum_{x\in
        \Lambda^+_{\ell_0}}(\BF \alpha(x)-{\ab})}}{M^{\ell_0}(\BF \lambda(\y+{\ab}))}\right],$$
where the chemical potential or `tilt' $\BF \lambda (\y+{\ab})$ is chosen to change the density to
$\y+{\ab}$ and $M(\BF \lambda) = E_{\nu_{\ab}}[e^{\BF \lambda \cdot (\BF \alpha(0)-{\ab})}]$. 
%
Consider now the gradient and Hessian of $E_{\nu_{\y + \ab}}[h]$:
\begin{equation}
\label{eq:pda1}
\partial_{y^j} E_{\nu_{\y+\ab}}[h (\BF \alpha)] = \sum_{r=1}^n \partial_{a^j} \lambda^r(\y+\ab) \; E_{\nu_{\y+\ab}}\Big[(h(\BF \alpha)-E_{\nu_{\y+\ab}}[h])(\sum_{x\in \Lambda^+_\ell}(\alpha^r(x) - y^r-a^r_0))\Big],
\end{equation}
and
\begin{equation}
\label{eq:pda2}
\begin{split}
& \partial_{y^k}\partial_{y^i} E_{\nu_{\y+\ab}}[h (\BF \alpha )] \\
&= \sum_{r=1}^n \partial_{a^k}\partial_{a^i}\lambda^r(\y+\ab)E_{\nu_{\y+\ab}}[(h(\a)-E_{\nu_{\y+\ab}}[h])(\sum_{x\in \Lambda^+_\ell}(\alpha^r(x) - y^r-a^r_0))]\\
&\ \ + \sum_{j,r}\partial_{a^i}\lambda^j(\y+\ab) \partial_{a^k}\lambda^r (\y+\ab)\\
&\ \ \ \ \ \ \ \times E_{\nu_{\y+\ab}}[(h(\a)-E_{\nu_{\y+\ab}}[h])(\sum_{x\in \Lambda^+_\ell}(\BF \alpha^j(x) - y^j-a^j_0))(\sum_{x\in \Lambda^+_\ell}(\alpha^r(x) - y^r-a^r_0))].
\end{split}
\end{equation}

The third partial derivatives can also be computed as in the single species case (cf. step 6 in proof of Proposition 5.1 in \cite{GJS}). Recall from Lemma \ref{lem:2.1} that 
\begin{equation*}
\nabla \BF \lambda (\y + \ab) = \Gamma^{-1}  (\y + \ab).
\end{equation*}

Suppose now all partial derivatives vanish, $\partial_{y^j}E_{\nu_{\y+\ab}}[h (\BF \alpha)]=0$ for $1\leq j\leq n$.  Then, by \eqref{eq:pda1}, with
$${\BF w} = \Big ( E_{\nu_{\y+\ab}}\big[\, (h(\BF \alpha)-E_{\nu_{\y+\ab}}[h])(\sum_{x\in \Lambda^+_\ell} (\alpha^r(x)-y^r-a^r_0)) \, \big]\Big)_{r=1}^n,$$
thought of as a row vector, we will have ${\BF w}\Gamma^{-1}(\y+\ab)={\BF 0}$, and so ${\BF w}={\BF 0}$.  Hence, in this case, the formula \eqref{eq:pda2} for the second partial derivatives $\partial_{y_k}\partial_{y_j} E_{\nu_{\y+\ab}}[h]$ simplifies, the first term vanishing.

Moreover, when $\tilde h(\ab) = 0$ and $\nabla \tilde h(\ab)={\BF 0}$, we see by \eqref{eq:pda1} that
\begin{equation}
\label{step6_line}
\partial_{a^j} \tilde h({\ab}) \ = \ \sum_{r=1}^n \partial_{a^j} \lambda^r(\ab) E_{\nu_{\ab}}\Big[h(\BF \alpha)(\sum_{x\in \Lambda^+_\ell}(\alpha^r(x) -a^r_0)) \Big]
\end{equation}
and by \eqref{eq:pda2} that
\begin{equation*}
\partial_{a^k}\partial_{a^i} \tilde h(\ab) = \sum_{j,r}\partial_{a^i}\lambda^j(\ab) \partial_{a^k}\lambda^r(\ab)  E_{\nu_{\ab}} \Big[h(\BF \alpha)(\sum_{x\in \Lambda^+_\ell}(\alpha^j(x) -a^j_0))(\sum_{x\in \Lambda^+_\ell}(\alpha^r(x) -a^r_0)) \Big].
\end{equation*}

Finally, for $|\y|\leq \delta$, when $\tilde h({\ab})=0$ and $\nabla \tilde h({\ab})={\BF 0}$, we may expand around $\y={\BF 0}$:
\begin{eqnarray*}
&&E_{\nu_{\y+{\ab}}}[h(\BF \alpha)]  \\
&& =   \frac{1}{2}\sum_{k,i} \Big[\sum_{j,r} \partial_{a^i} \lambda^j(\ab) \partial_{a^k} \lambda^r (\ab)E_{\nu_{\ab}}[h(\BF \alpha)\big(\sum_{x\in\Lambda^+_{\ell_0}}\alpha^j(x)-a^j_0\big)\big(\sum_{x\in\Lambda^+_{\ell_0}}\alpha^r(x)-a^r_0\big)\big]\Big]\; y^i y^k\\
&&\ \ \ \ \   + \|\y\|^3r({\ab},\delta,h).
\end{eqnarray*}
When, only $\tilde h({\ab})=0$ is known,
\begin{align*}
E_{\nu_{\y+\ab}}[h(\a)] &= \sum_{j=1}^n \left[
\sum_{r=1}^n \partial_{a^j} \lambda^r(\ab) E_{\nu_{\ab}} [(h(\a)-E_{\nu_{\ab}}[h])(\sum_{\z\in \Lambda^+_\ell}(\alpha^r(\z) - a^r_0))]\right] y^j\\
 &\ \ \ \ \ \ + \|\y\|^2r({\ab},\delta,h). 
 \end{align*}
When possibly $\tilde h({\ab})\neq 0$,
  \begin{eqnarray*}
E_{\nu_{\y+{\ab}}}[h(\BF \alpha)]&=& E_{\nu_{\ab}}[h(\BF \alpha)] + \|\y\| r({\ab}, \delta, h).
 \end{eqnarray*}
Here, as the first, second and third partial derivatives of $E_{\nu_{\y+\ab}}[h]$ are bounded for $|\y|\leq \delta$,
we may conclude that the remainders $|r({\ab},\delta,h)| \leq C({\ab}, \delta)\|h\|_{L^2(\nu_{\ab})}$.

\medskip

{\it Step 7.}
Consider the expansion of $E_{\nu_{\ab}}[f (\BF \alpha) |{\BF \alpha}^{(\ell)} = \y]$ in Step 5 when $|\y|\leq \delta$.  With $h(\BF \alpha)$ equal to variously $f (\BF \alpha)$, $f(\BF \alpha)\big(\sum_{x\in\Lambda_{\ell_0}^+}(\alpha^j (x)-y^j - a^j_0)\big)$, and $f(\BF \alpha)\big(\sum_{x\in\Lambda_{\ell_0}^+}(\alpha^j (x)-y^j - a^j_0)\big)\big(\sum_{x\in\Lambda_{\ell_0}^+}(\alpha^k(x)-y^k - a^k_0)\big)$, we may write $E_{\nu_{\ab}}[f (\BF \alpha) | {\BF \alpha}^{(\ell)} = \y]$
equal to
\begin{equation*}
\begin{split}
 &\frac{\kappa_0 (\ell) }{2}\sum_{k,i}\Big[\sum_{j,r} \partial_{a^i} \lambda^j \, \partial_{a^k} \lambda^r \, E_{\nu_{\ab}}\big[f(\BF \alpha)\big(\sum_{x\in \Lambda_{\ell_0}^+}(\alpha^j(x)-a^j_0)\big)\big(\sum_{x\in \Lambda_{\ell_0}^+}(\alpha^r(x)-a^r_0)\big)\big]\Big]y^i y^k \\
 &\ \ \ \ \ \ \ \ \ \ + \kappa_0 (\ell) \|\y\|^3 r(f)\\
&\ + \frac{1}{\sqrt{\ell-\ell_0}}\sum_{j=1}^n \kappa^j_1 (\ell)  \left\{\sum_{r=1}^n \partial_{a^k} \lambda^r\,  y^r \, E_{\nu_{\ab}}\big[f(\BF \alpha)\big(\sum_{x\in \Lambda_{\ell_0}^+}(\alpha^j(x)-a^j_0)\big)\big(\sum_{x\in \Lambda_{\ell_0}^+}(\alpha^r(x)-a^r_0)\big)\big]\right\}\\
&\ \ \ \ \ \ \ \ \ + \sum_{j=1}^n\frac{\kappa^j_1 (\ell) }{\sqrt{\ell-\ell_0}} \|\y\|^2r(f)\\
&\ + \frac{1}{\ell-\ell_0}\sum_{j,r}\kappa^{j,r}_2 (\ell) \,  E_{\nu_{\ab}}\big[f(\BF \alpha)\big(\sum_{x\in \Lambda^+_{\ell_0}}(\alpha^j(x)-a^j_0)\big)\big(\sum_{x\in \Lambda^+_{\ell_0}}(\alpha^r(x)-a^r_0)\big)\big]\\
&\ \ \ \ \ \ \ \ \  + \sum_{j,r}\frac{\kappa^{j,r}_2 (\ell)}{\ell-\ell_0}\|\y\| r(f) + \varepsilon_f(\ell),
\end{split}
\end{equation*}
where $|r(f)|\leq C({\ab}, \ell_0, \delta)\|f\|^2_{L^2(\nu_{\ab})}$.  

To group more the above expressions, note that $[\Gamma^{-1}]^* = \Gamma^{-1}$ as $\Gamma$ is symmetric.  Then, noting Lemma \ref{lem:2.1} and recalling \eqref{eq:kappa2jr},
\begin{align*}
-\kappa^{j,r}_2 (\ell) &= \Gamma^{-1}_{j r} (\y + \ab)  + O(\ell^{-1/2})\\
&=\big(\Gamma^{-1}\Gamma[\Gamma^{-1}]^*\big)_{j r} \, (\y + \ab) + O(\ell^{-1/2})\\
&= \sum_{k,i} \Big[\partial_{a^i} \lambda^j \, \partial_{a^k} \, \lambda ^r \, \Gamma_{i k} \Big](\y + \ab) + O(\ell^{-1/2}).
\end{align*}

Hence, with the assumptions $\tilde f({\ab})=0$ and $\nabla \tilde f(\ab)={\BF 0}$, the estimates and relations on $\kappa_i$ for $i=0,1,2$,  \eqref{step6_line}, and
$E_{\nu_{\ab}}\big[[y^j]^{2p}\big] = O(\ell^{-p})$ so that each $y^j$ factor is $O(\ell^{-1/2})$, and by Taylor expansion around $\ab$ to have $-\kappa^{j,r}_2 (\ell) = \sum_{k,i} \Big[\partial_{a^i} \lambda^j \, \partial_{a^k} \, \lambda ^r \, \Gamma_{i k} \Big](\ab) + O(\ell^{-1/2})$,  we can group the dominant terms to arrive at
\begin{equation*}
\begin{split}
&E_{\nu_{\ab}} \left[ 1(\|\y\|\leq \delta)\Big(E_{\nu_{\ab}}[f(\BF \alpha)|{\BF \alpha}^{(\ell)} = \y] -  \sum_{i,k}\Big\{y^i y^k - \tfrac{\Gamma_{ik} (\ab)}{\ell}\Big\}\frac{\partial_{a^i}\partial_{a^k }\tilde f({\ab})}{2}\Big)^4\right] \\ 
&\leq \ C\|f\|^4_{L^2(\nu_{\ab})}\ell^{-6}.
\end{split}
\end{equation*}

\medskip
{\it Step 8.}
On the other hand,
by say large deviations estimates, we bound
\begin{equation*}
\begin{split}
&E_{\nu_{\ab}} \left[ 1(\|\y\|> \delta)\Big(E_{\nu_{\ab}}[f(\BF \alpha)|{\BF \alpha}^{(\ell)} =\y] -  \sum_{i,k}\Big\{y^i y^k - \tfrac{\Gamma_{ik} (\ab)}{\ell}\Big\}\frac{\partial_{a^i}\partial_{a^k }\tilde f({\ab})}{2}\Big)^4\right]\\
&\leq C\|f\|^4_{L^5(\nu_{\ab})}O(\ell^{-6})
\end{split}
 \end{equation*}
to finish the proof.
\qed

\section{Trilinear condition}
\label{trilinear}

In this section, we show that the coupled KPZ-Burgers equation obtained in the
limit satisfies the so-called trilinear condition \eqref{eq:7.tri} stated below.
Recall that, under the Frame condition (FC)
for the density, the equation \eqref{gen_OU_sec}
has the form:
\begin{equation}  \label{eq:7.1}
\partial_t \Y^i= \tfrac{\la}2 \De \Y^i + c \sum_{j,k=1}^n \ga_{j, k}^i \nabla (\Y^j\Y^k)
+ q_i \nabla\dot{\W}_t^i, \quad 1 \le i \le n,
\end{equation}
where $\la = \partial_{a^i}\tilde g_i(\BF a)\Big|_{\ab}$,
$\ga_{j,k}^i = \partial_{a^j}\partial_{a^k}(\tilde g_i( \BF a))\Big|_{\ab}$ and 
$q_i = \sqrt{\tilde g_i(\BF a)}\Big|_{\ab}.$  
By the change of time and magnitude defined by 
$\bar \Y_t^i := \tfrac{\sqrt{\la}}{q_i} \Y_{\la^{-1}t}^i$, the 
equation \eqref{eq:7.1} is rewritten as
\begin{equation}  \label{eq:6.2-b}
\partial_t \bar \Y^i= \tfrac12 \De \bar \Y^i + \tfrac{c}{\la^{3/2}}
\sum_{j,k=1}^n \tfrac{q_jq_k}{q_i}  \ga_{j k}^i \nabla(\bar \Y^j \bar \Y^k)
+ \nabla\dot{\W}_t^i, \quad 1 \le i \le n,
\end{equation}
in a canonical form in the sense of law.

\begin{prop}  \label{prop:6.1-b}
The coupling constants
\begin{equation} 
 \label{eq:6.3-b}
\Ga_{j \ell}^i := \tfrac{c}{\la^{3/2}}\tfrac{q_j q_\ell}{q_i}  \ga_{j \ell}^i,
\end{equation}
satisfy the trilinear condition:
\begin{equation} 
 \label{eq:7.tri}
\Ga_{j \ell}^i = \Ga_{\ell j}^i = \Ga_{i \ell}^j,
\end{equation}
for every $1\le i,j,\ell\le n$.
\end{prop}

\begin{rem}\rm
If the coupling constants 
$(\Ga_{j \ell}^i)_{i,j, \ell}$ satisfy the trilinear condition,
the coupled KPZ-Burgers equation has the distribution of the white noise
as its invariant measure, see \cite{FH17}. One can expect that the converse would be also true, see \cite{F18}.
\end{rem}

\begin{proof}
Recall that $\boldsymbol\fa( \BF a) = (\tilde g_i( \BF a))_{i=1}^n$ is defined by the relation
$$
a^i=R_i(\boldsymbol\fa) = \frac1{Z_{\boldsymbol\fa}}\sum_\Bk k^i \frac{\boldsymbol\fa^\Bk}{\Bg!(\Bk)},
\quad 1\le i \le n.
$$
Consider the notation 
$$\langle f(\Bk)\rangle_{\bar\nu_{\boldsymbol\fa}}:= \frac{1}{Z_{\boldsymbol\fa}}\sum_{\Bk} f(\Bk)\frac{\boldsymbol\fa^{\Bk }}{\Bg!(\Bk)}.$$
It will be convenient now to revisit the proof of Lemma \ref{lem:2.1}:  Write
\begin{align}  
\label{eq:7.first}
\partial_{\varphi^j} a^i 
& = \frac1{Z_{\boldsymbol\fa}}\sum_\Bk k^i k^j\frac{\boldsymbol\fa^{\Bk-\de_j}}{\Bg!(\Bk)}
- \frac1{Z_{\boldsymbol\fa}^2}\sum_\Bk k^i \frac{\boldsymbol\fa^\Bk}{\Bg!(\Bk)}
\; \sum_\Bk k^j \frac{\boldsymbol\fa^{\Bk-\de_j}}{\Bg!(\Bk)}  \\
& = \frac1{\fa^j} \left( \lan k^ik^j\ran_{\bar\nu_{\boldsymbol\fa}}
- \lan k^i\ran_{\bar\nu_{\boldsymbol\fa}}
\lan k^j\ran_{\bar\nu_{\boldsymbol\fa}}\right) 
\notag  \\
& = \frac1{\fa^j}  \Big({\rm cov}(\bar\nu_{\boldsymbol\fa})\big)_{ij},
\notag
\end{align}
where $\de_j\in \PZ^n$ is defined by $(\de_j)_i=\de_{ij}$ for $1\le i \le n$.
Then, 
\begin{equation}  \label{eq:7.3}
\nabla \BF a
\equiv \left({\partial_{\fa^j} \, a^i} \right)_{ij}
= {\rm cov}(\bar\nu_{\boldsymbol\fa})
\; {\rm diag}\Big(\Big(\frac1{\fa^j}\Big)_{1\le j\le n}\Big),
\end{equation}
so that
\begin{equation}  \label{eq:7.4}
\nabla \BF \varphi 
\equiv \left({\partial_{a^j}\,  \fa^i} \right)_{ij}
= \Big({\rm diag} \big(\tilde g_i\big)_{1\le i\le n}\Big) \;  {\rm cov}(\bar\nu_{\boldsymbol\fa})^{-1},
\end{equation}
or equivalently, as in Lemma \ref{lem:2.1},
\begin{equation}
\label{eq:6.8-a}
{\partial_{a^j} \, \tilde g_i}= \tilde g_i \, \big({\rm cov}(\bar\nu_{\boldsymbol\fa})^{-1}\big)_{ij}.
\end{equation}
Taking a further derivative of  \eqref{eq:7.4}, we have
\begin{align}  \label{eq:7.5}
\partial_{a^\ell} \nabla \BF \fa
& = {\partial_{a^\ell}} \Big( {\rm diag}\, (\tilde g_i)_{1\le i\le n}\Big) \, {\rm cov}(\bar\nu_{\boldsymbol\fa})^{-1}
+ \Big({\rm diag}\, ((\tilde g_i)_{1\le i\le n})\Big) \, \partial_{a^\ell}\Big({\rm cov}(\bar\nu_{\boldsymbol\fa})^{-1}\Big).
\end{align}
By \eqref{eq:7.4} again, the first term in the right hand side of \eqref{eq:7.5} is equal to
$$
{\rm diag} \Big(\big(\tilde g_i  \big( {\rm cov}(\bar\nu_{\boldsymbol\fa})^{-1} \big)_{i\ell}\big)_{1\le i\le n}  \Big)
\; {\rm cov}(\bar\nu_{\boldsymbol\fa})^{-1},
$$
whose $ij$-component is
\begin{equation}  \label{eq:7.6}
\tilde g_i  \big( {\rm cov}(\bar\nu_{\boldsymbol\fa})^{-1} \big)_{i\ell}
\big({\rm cov}(\bar\nu_{\boldsymbol\fa})^{-1}\big)_{ij}.
\end{equation}
By Proposition \ref{prop:4.2}, when $\BF a= \ab$, ${\rm cov}(\bar\nu_{\boldsymbol\fa})$ and therefore
${\rm cov}(\bar\nu_{\boldsymbol\fa})^{-1}$ is a diagonal matrix and, after multiplying by
$\frac{q_jq_\ell}{q_i}$, noting $\tilde g_i = q_i^2$,  \eqref{eq:7.6} becomes
$$
q_iq_jq_\ell \cdot \de_{i\ell}\de_{ij} \big( {\rm cov}(\bar\nu_{\boldsymbol\fa})^{-1} \big)_{ii}^2.
$$
This term is invariant under the permutations of $i,j,\ell$.

The final task is to show the invariance of the second term of \eqref{eq:7.5} after multiplying
$\frac{q_jq_\ell}{q_i}$.  We first note that in general for $n\times n$ non-degenerate matrix
$A(\BF a)$ which is smooth in $\BF a$, we have
$$
\partial_{a^\ell} A^{-1}(\BF a) = - A^{-1}(\BF a) \, \big[\partial_{a^\ell} { A(\BF a)} \big]\,  A^{-1}(\BF a).
$$
Indeed, this follows from
$$
0 = \partial_{a^\ell} \left( A(\BF a) A^{-1}(\BF a)\right)
= \partial_{a^\ell} A(\BF a) \;  A^{-1}(\BF a)+ A(\BF a)\;  \partial_{a^\ell} A^{-1}(\BF a).
$$
In particular, using again (\ref{eq:7.4}), we have
\begin{align*}
&\big(\partial_{a^\ell} \, {\rm cov}(\bar\nu_{\boldsymbol\fa})^{-1}\big)_{ij}
 = - \sum_{p,q=1}^n \big({\rm cov}(\bar\nu_{\boldsymbol\fa})^{-1}\big)_{ip}
\; \partial_{a^\ell} \big({\rm cov}(\bar\nu_{\boldsymbol\fa})\big)_{pq} \; 
\big({\rm cov}(\bar\nu_{\boldsymbol\fa})^{-1}\big)_{qj}  \\
& = - \sum_{m,p,q=1}^n \big({\rm cov}(\bar\nu_{\boldsymbol\fa})^{-1}\big)_{ip}
\; \partial_{a^\ell} {\fa^m} \, {\partial_{\fa^m}  \big({\rm cov}(\bar\nu_{\boldsymbol\fa})\big)_{pq}}\;
\big({\rm cov}(\bar\nu_{\boldsymbol\fa})^{-1}\big)_{qj}  \\
& = - \sum_{m,p,q=1}^n  \big({\rm cov}(\bar\nu_{\boldsymbol\fa})^{-1}\big)_{ip}
\; \fa_m 
\; \big({\rm cov}(\bar\nu_{\boldsymbol\fa})^{-1}\big)_{m\ell}
\; {\partial_{\fa_m} \big({\rm cov}(\bar\nu_{\boldsymbol\fa})\big)_{pq}}
\big({\rm cov}(\bar\nu_{\boldsymbol\fa})^{-1}\big)_{qj}.
\end{align*}
Here, we need to compute and have similarly to \eqref{eq:7.first} that
\begin{align*}
&\Big({\partial_{\fa^m}\;   {\rm cov}(\bar\nu_{\boldsymbol\fa})\Big)_{ij}} 
 =
\partial_{\fa^m} \, \left( \lan k^i k^j\ran_{\bar\nu_{\boldsymbol\fa}}
- \lan k^i\ran_{\bar\nu_{\boldsymbol\fa}}
\lan k^j\ran_{\bar\nu_{\boldsymbol\fa}}\right) \\
& = \frac1{\fa_m}  \left( \lan k^i k^j k^m\ran_{\bar\nu_{\boldsymbol\fa}}
- \lan k^i k^j \ran_{\bar\nu_{\boldsymbol\fa}}
\lan k^m \ran_{\bar\nu_{\boldsymbol\fa}}\right)
- \frac1{\fa^m}  \left( \lan k^i k^m \ran_{\bar\nu_{\boldsymbol\fa}}
- \lan k^i\ran_{\bar\nu_{\boldsymbol\fa}}
\lan k^m\ran_{\bar\nu_{\boldsymbol\fa}}\right)\lan k^j\ran_{\bar\nu_{\boldsymbol\fa}}\\
& \hskip 40mm
- \frac1{\fa^m}  \left( \lan k^j k^m\ran_{\bar\nu_{\boldsymbol\fa}}
- \lan k^j\ran_{\bar\nu_{\boldsymbol\fa}}
\lan k^m\ran_{\bar\nu_{\boldsymbol\fa}}\right)\lan k^i\ran_{\bar\nu_{\boldsymbol\fa}} \\
& = \frac1{\fa^m}  \left( \lan k^i k^j k^m\ran_{\bar\nu_{\boldsymbol\fa}}
- \lan k^i k^j \ran_{\bar\nu_{\boldsymbol\fa}} \lan k^m\ran_{\bar\nu_{\boldsymbol\fa}}
- \lan k^j k^m\ran_{\bar\nu_{\boldsymbol\fa}} \lan k^i\ran_{\bar\nu_{\boldsymbol\fa}}  \right. \\
& \hskip 40mm \left.  
- \lan k^m k^i\ran_{\bar\nu_{\boldsymbol\fa}} \lan k^j\ran_{\bar\nu_{\boldsymbol\fa}}  
+ 2 \lan k^i\ran_{\bar\nu_{\boldsymbol\fa}}
\lan k^j\ran_{\bar\nu_{\boldsymbol\fa}}
\lan k^m\ran_{\bar\nu_{\boldsymbol\fa}}\right).
\end{align*}
Note that except $\frac1{\fa^m}$, which actually cancels with $\fa^m$ in the last formula,
the right hand side is invariant under the permutations of $i,j,m$.
Since the $ij$-component of the second term of \eqref{eq:7.5} is equal to
$$
\tilde g_i \; \big( \partial_{a^\ell} \, {\rm cov}(\bar\nu_{\boldsymbol\fa})^{-1}\big)_{ij}
$$
after multiplying by $\frac{q_jq_\ell}{q_i}$, this is invariant under the permutations of $i,j,\ell$.
This completes the proof of the trilinearity of $\Ga$.
\end{proof}

\section{Multi-colored systems}
\label{example_section}

We now discuss the special case of multi-colored zero-range processes.  Here, the dynamics is
determined as follows.  Let the jump rate $g:\PZ\to (0,\infty)$ of color-blind particles be
given and define the jump rates $g_i(\Bk)$ of $i$th colored particles by
\begin{equation}  \label{eq:7.1-b}
g_i(\Bk) = g(|\Bk|) \frac{k^i}{|\Bk|}, \quad 1 \le i \le n,
\end{equation}
where we recall $|\Bk| = k^1+\cdots+k^n$ for $\Bk=(k^1,\ldots,k^n)$.  In other words, at time $t=0$,
we paint every particle in $n$-different colors.  They evolve due to the color-blind particles
dynamics and the evolution of colors is determined simply by choosing a particle with equal
probability when a jump is going to happen.  This gives the factor $\frac{k^i}{|\Bk|}$.
The jump rates $g_i(\Bk)$ given by \eqref{eq:7.1-b} satisfy all necessary conditions.

Now, the invariant measures are given by
$$
\nu_{\BF a} (\Bk) = \nu_\rho(|\Bk|) \frac{|\Bk|!}{k^1!\cdots k^n!} \prod_{i=1}^n
\left( \frac{a^i}\rho\right)^{k^i}.
$$
for $\Bk= (k^1,\ldots,k^n)$, $\BF a= (a^1,\ldots,a^n)$, $\rho = a^1+\cdots + a^n$, and
$$
\nu_\rho(k) \equiv \bar\nu_{\varphi(\rho)}(k)=
\frac1{Z_{\varphi (\rho)}} \frac{(\varphi (\rho))^k}{g(k)!}
$$
is the distribution of the color-blind particles. Here the $1:1$ correspondance $\rho \to \varphi (\rho)$ is similar to the correspondance described in Subsection \ref{invariant_section}.  

\subsection{Frame condition}

We show in Proposition \ref{prop:7.4-a} below that the Frame condition (FC)
holds exactly for $\ab$ when $\rho_0 = a_0^1+\ldots+a_0^n$ satisfies
\begin{equation}
\label{eq:Seth}
(\partial_\rho\fa)(\rho_0) = \frac{\fa(\rho_0)}{\rho_0}.
\end{equation}
Since  $\rho(\varphi) = Z_\varphi^{-1}\sum k\varphi^k/g(k)!$ and so $(\partial_\rho\fa)(\rho_0) = \fa(\rho_0)/{\sigma^2(\rho_0)}$, equivalent to \eqref{eq:Seth} is the condition 
\begin{equation}
\label{balance}
\sigma^2(\rho_0)=\rho_0
\end{equation}
where $\sigma^2(\rho)$ is the variance of the law $\nu_\rho$.

In the $n$-color setting, the conditions given
in Proposition \ref{prop:4.2} are read as
\begin{equation}  \label{eq:5.2}
\V_{ij}(\ab) =0, \quad \tilde g_i (\ab) \V_{jj} (\ab) =\tilde g_j (\ab) \V_{ii} (\ab) ,
\end{equation}
for every $i\not= j$.  

We now develop a series of computations leading to Proposition \ref{prop:7.4-a}.  We will use the notation $\langle f(\Bk)\rangle_{\nu_{\BF a}} = E_{\nu_{\BF a}}[f(\Bk)]$ in the sequel.

\begin{lem}  \label{lem:5.1}
We have
\begin{align*}
& \lan k^i\ran_{\nu_{\BF a}} = a^i, \\
& \lan k^i k^j \ran_{\nu_{\BF a}} = \frac{a^ia^j}{\rho^2} \lan k(k-1)\ran_{\nu_\rho},
\quad \text{ if } \;  i\not= j, \\
& \lan [k^i]^2\ran_{\nu_{\BF a}} = \frac{(a^i)^2}{\rho^2} \lan k(k-1)\ran_{\nu_\rho} + a^i,  \\
& \tilde g_i( {\BF a}) = \frac{\fa(\rho) a^i}\rho.
\end{align*}
\end{lem}

\begin{proof}
We have
\begin{align*}
\lan k^i\ran_{\nu_{\BF a}}
& = \sum_k \nu_\rho(k) k \frac{a^i}\rho
\sum {}'
\frac{(k-1)!}{k^1!\cdots (k^i-1)! \cdots k^n!} \prod_{j=1}^n
\left( \frac{a^j}\rho\right)^{k^j- \de_{ij}}\\
& = \frac{a^i}\rho \sum_k k \nu_\rho(k)
= a^i,
\end{align*}
where sum $\sum'$ is over $\Bk$ such that
$k^1+\cdots + (k^i-1)+\cdots+k^n = k-1$.
Also, for $i\neq j$,
\begin{align*}
\lan k^i k^j \ran_{\nu_{\BF a}}
& = \sum_k \nu_\rho(k) k(k-1) \frac{a^i}\rho\frac{a^j}\rho  \\
& \qquad \times
\sum{}''
\frac{(k-2)!}{k^1!\cdots (k^i-1)! \cdots (k^j-1)! \cdots k^n!} \prod_{\ell=1}^n
\left( \frac{a^\ell}\rho\right)^{k_\ell- \de_{\ell i} - \de_{\ell j}}\\
& = \frac{a^i a^j}{\rho^2}\lan k(k-1)\ran_{\nu_\rho}, 
\end{align*}
where the sum $\sum''$ is over $\Bk$ such that $k^1+\cdots + (k^i-1)+\cdots+(k^j-1)+ \cdots +k^n = k-2$.
Moreover,
\begin{align*}
\lan [k^i]^2 \ran_{\nu_{\BF a}}
& = \lan k^i(k^i-1) \ran_{\nu_{\BF a}} + \lan k_i \ran_{\nu_{\BF a}}  \\
& = \sum_k \nu_\rho(k) k(k-1) \left(\frac{a^i}\rho\right)^2
+ \lan k^i \ran_{\nu_a}  \\
& = \frac{(a^i)^2}{\rho^2}\lan k(k-1)\ran_{\nu_\rho} + a^i.
\end{align*}
The last identity for $\tilde g_i( \BF a)$ is obtained by noting 
$$
g_i(\Bk) = g(k) \frac{k^i}k,
$$
in the present setting.
\end{proof}

\begin{lem} \label{lem:5.2}
We have
\begin{align}
 & \la \cfrac{d Z_\la}{d\lambda}  = \rho Z_\la,   \label{eq:Z1}  \\
 & \la^2 \cfrac{d^2 Z_\la}{d\lambda^2} = \lan k(k-1)\ran_{\nu_\rho} Z_\la,  \label{eq:Z2}
\end{align}
where $\rho$ and $\la$ are related by
$\fa(\rho) = \lan g(k)\ran_{\nu_\la} = \la.$
\end{lem}

\begin{proof}
Recall
$$
Z_\la = \sum_k \frac{\la^k}{g(k)!}
$$
and take the derivatives in $\la$.
\end{proof}

\begin{lem} \label{lem:5.3}
We have $\V_{ij} (\ab) =0$ for $i\not= j$ is equivalent to \eqref{eq:Seth}.
\end{lem}

\begin{proof} 
From Lemma \ref{lem:5.1}, we have
\begin{equation}  \label{eq:5.5}
\V_{ij} (\BF a) = \frac{a^i a^j}{\rho^2} \left(\lan k(k-1)\ran_{\nu_{\rho}} -\rho^2\right),
\end{equation}
for $i\not= j$.  On the other hand, \eqref{eq:Z1} shows
$$
\cfrac{d}{d\lambda}\, \log Z_\la= \frac{\rho}\la = \frac{\rho}{\fa(\rho)}.
$$
Since
$$
\cfrac{d^2}{d\lambda^2} \, \log Z_\la = \frac{\tfrac{d^2 Z_\la}{d\la^2} }{Z_\la} - \left( \frac{\tfrac{d Z_\la}{d\la} }{Z_\la} \right)^2,
$$
from \eqref{eq:Z2},
$$
\lan k(k-1)\ran_{\nu_\rho} = \la^2 \left\{\cfrac{d}{d\la} \,  \left(\frac{\rho}{\fa(\rho)}\right)
+ \left(\frac{\rho}{\fa(\rho)}\right)^2 \right\}.
$$
Since
$$
\frac{d}{d\la} = \frac{d\rho}{d\la}\frac{d}{d\rho}
= \frac1{\fa'(\rho)}\frac{d}{d\rho},
$$
we have
\begin{align}  \label{eq:5.6}
\lan k(k-1)\ran_{\nu_\rho} 
& = \fa(\rho)^2 
\left\{ \frac1{\fa'(\rho)} \cfrac{d}{d\rho}  \left(\frac{\rho}{\fa(\rho)}\right)
+ \left(\frac{\rho}{\fa(\rho)}\right)^2 \right\}  \\
& =  \frac1{\fa'(\rho)}  \left(\fa(\rho) -\rho\fa'(\rho)\right)
+ \rho^2,  \notag
\end{align}
where $'=\frac{d}{d\rho}$.  Thus, from \eqref{eq:5.5}, we have
\begin{equation*}
\V_{ij} (\BF a)  = \frac{a^ia^j}{\rho^2\fa'(\rho)}  \left(\fa(\rho) -\rho\fa'(\rho)\right).
\end{equation*}
This implies the conclusion.
\end{proof}

\begin{prop}  \label{prop:7.4-a}
In the present setting, the conditions in Proposition \ref{prop:4.2} (or \eqref{eq:5.2}) hold for $\ab$ if and only if \eqref{eq:Seth} is satisfied for the corresponding $\rho_0$.
\end{prop}

\begin{proof}
Lemma \ref{lem:5.3} shows that \eqref{eq:5.2} implies \eqref{eq:Seth}.
Therefore, we need only prove that \eqref{eq:Seth} implies $\tilde g_i (\ab) \V_{jj} (\ab) = \tilde g_j (\ab) \V_{ii} (\ab)$ for $i\not= j$.  Note that from Lemma \ref{lem:5.1}
\begin{align*}
\V_{ii} (\BF a) & = \lan [k^i]^2\ran_{\nu_{\BF a}} - \lan k^i\ran_{\nu_{\BF a}}^2 \\
& = \frac{(a^i)^2}{\rho^2} \lan k(k-1)\ran_{\nu_\rho}+a^i - (a^i)^2.
\end{align*}
But, if \eqref{eq:Seth} holds, $\lan k(k-1)\ran_{\nu_{\rho_0}}= \rho_0^2$ from \eqref{eq:5.6}.
Therefore $\V_{ii} (\ab) =a_0^i$.  Thus we obtain $\tilde g_i (\ab) \V_{jj} (\ab) = \tilde g_j (\ab) \V_{ii} (\ab)$ recalling Lemma \ref{lem:5.1} for $\tilde g_i$.
\end{proof}

\subsection{Multi-colored coupled KPZ-Burgers equations}
\label{multi_color_kpz}

We assume the condition \eqref{eq:Seth} is satisfied and we compute now the constants appearing in the KPZ-Burgers equation (\ref{eq:7.1}). By the fourth item of Lemma \ref{lem:5.1}, we have that
\begin{equation*}
\partial_{a^i} {\tilde g}_i (\ab) = \tfrac{\fa (\rho_0)}{\rho_0} + \partial_\rho \left( \tfrac{\fa (\rho)}{\rho} \right)\Big|_{\rho_0} \, a_0^i = \tfrac{\fa (\rho_0)}{\rho_0} 
\end{equation*}
and
\begin{equation*}
\partial_{a^j} \partial_{a^k} {\tilde g}_i (\ab) = [\delta_{ki}+\delta_{ji}]   \partial_\rho \left( \tfrac{\fa (\rho)}{\rho} \right)\Big|_{\rho_0} \, + \, a_0^i \partial^2_\rho \left( \tfrac{\fa (\rho)}{\rho} \right)\Big|_{\rho_0}= a_0^i \tfrac{\fa''(\rho_0)}{\rho_0}.
\end{equation*}
Hence we get
\begin{align}  
\label{eq:5.2_reduced}
\partial_t {\mathcal Y}_t^i = c_1 \De {\mathcal Y}_t^i + c_2a_0^i \nabla \left(\sum_{j=1}^n{\mathcal Y}_t^j\right)^2
+ c_3 \sqrt{a_0^i} \nabla \dot{\mathcal W}_t^i,
\end{align}
where $c_1= \frac12\tfrac{\fa(\rho_0)}{\rho_0}$,
$c_2 = c\tfrac{\fa''(\rho_0)}{\rho_0}$,
$c_3 = \sqrt{\tfrac{\fa(\rho_0)}{\rho_0}}$.
Or, if we write at KPZ level, that is when ${\mathcal Y}_t^i = \partial_u h_t^i$, we have
\begin{align}  \label{eq:5.3}
\partial_t h_t^i = c_1 \De h_t^i + c_2a_0^i \left(\nabla \sum_{j=1}^nh_t^j\right)^2
+ c_3 \sqrt{a_0^i} \dot{\mathcal W}_t^i, \quad i=1,2,\ldots, n.
\end{align}
One can actually {\textit{decouple}} the coupled KPZ equation \eqref{eq:5.3} as follows.
The process $H:= \sum_{i=1}^n h^i$, which corresponds to the color-blind system,
satisfies the scalar KPZ equation:
\begin{align*}
\partial_t H_t = c_1 \De H_t + c_2 \Big(\sum_{i=1}^n a_0^i\Big)  \;  (\nabla H_t)^2
+ c_3 \dot{\mathcal W}_t, \quad \dot{\mathcal W}_t := \sum_{i=1}^n \sqrt{a_0^i} \dot{\mathcal W}_t^i.
\end{align*}
On the other hand, $H^{ij}_t := a_0^j h^i - a_0^i h^j$ are Ornstein-Uhlenbeck processes 
satisfying
\begin{align*}
\partial_t H_t^{ij} = c_1 \De H_t^{ij} 
+ c_3 \dot{\mathcal W}_t^{ij}, \quad \dot{\mathcal W}_t^{ij} := \sqrt{a_0^i}a_0^j \dot{\mathcal W}_t^i - \sqrt{a_0^j}a_0^i \dot{\mathcal W}_t^j
\end{align*}
Moreover, it turns out that noises $\{\dot{\mathcal W}_t\}_{t\ge 0}$ and $\{\dot{\mathcal W}_t^{ij}\}_{t\ge 0}$ are independent, since (by multiplying by a test function)
$$
\E\left[\left(\sum_{k=1}^n \sqrt{a_0^k} w_t^k\right)
\left(\sqrt{a_0^i} a_0^j w_t^i - \sqrt{a_0^j}a_0^i w_t^j\right)\right]
= a_0^i a_0^j t - a_0^j a_0^i t =0.
$$

Now, the uniqueness of the (stationary) energy solutions of our coupled KPZ-Burgers equation
follows from the uniqueness of (stationary) scalar-valued KPZ-Burgers equation and Ornstein-Uhlenbeck
processes, respectively, and independence of these processes.  Since, such uniqueness also holds for systems on $\R$, as shown in \cite{GP18} (see also Subsection 5.3 \cite{FH17}), in this multi-color situation, we can formulate 
Theorem \ref{secondorder_thm} in the infinite volume on $\R$, as mentioned in Remark \ref{KPZ_rmk}.

Next, we show in the multi-color setting that the limit SPDE 
\eqref{gen_OU_sec} or \eqref{eq:7.1} is genuinely nonlinear. The previous discussion shows that the nonlinearity is equivalent to the condition $c_2 \ne 0$, namely there exist
a jump rate $g$ of color-blind particles and a density $\rho_0$, for which both
$\fa''(\rho_0) \not=0$ and the condition \eqref{eq:Seth} holds.
To simplify notation we write $\langle \cdot \rangle$ instead of $\langle\cdot \rangle_{\nu_a}$, where $\nu_a$ is the equilibrium distribution of the color-blind particles system. By \eqref{eq:7.5}, \eqref{eq:7.4} and computations made after \eqref{eq:7.5} with $n=1$, we have
\begin{align*}
\tilde g''(a) & = \fa''(a) = {\partial_{a} \fa} \; {\rm cov}(\bar\nu_\fa)^{-1}
+ \fa \; {\partial_a} \, {\rm cov}(\bar\nu_\fa)^{-1} \\
& = \fa \Big({\rm cov}(\bar\nu_\fa)^{-1}\big)^2 - \fa \Big({\rm cov}(\bar\nu_\fa)^{-1}\big)^3
\big( \lan k^3\ran - 3\lan k^2\ran\lan k\ran +2\lan k\ran^3\big) \\
& = - \fa \Big({\rm cov}(\bar\nu_\fa)^{-1}\big)^3
\big( \lan k^3\ran - 3\lan k^2\ran\lan k\ran +2\lan k\ran^3
-\lan k^2\ran + \lan k\ran^2 \big).
\end{align*}
Therefore, the condition $\fa ''(\rho_0)\not=0$ is equivalent to
\begin{align}  \label{eq:4.g-2}
\fa (\rho_0) \not=0 \quad \text{ and } \quad
\mathcal C (\rho_0) := \lan k^3\ran - 3\lan k^2\ran\lan k\ran +2\lan k\ran^3
-\lan k^2\ran + \lan k\ran^2 \not=0.
\end{align}
Thus our goal is to show that there exists a rate $g$ and a density $\rho_0$ such that \eqref{eq:4.g-2}
and the condition \eqref{balance}, that is equivalently $\si^2(\rho_0) = \rho_0$ or \eqref{eq:Seth} or
\begin{align} 
\label{eq:4.var-2}
\lan k^2\ran -\lan k\ran^2 = \lan k\ran
\end{align}
hold.

For example, in the simplest case with $g(k)=k$, $\bar\nu_\fa$ is a Poisson distribution and
\begin{align*}
& Z_\fa = \sum_{k=0}^\infty \frac{\fa^k}{k!} = e^\fa, \\
& \lan k\ran (=\rho) = \tfrac1{Z_\fa} \fa Z_\fa' = \fa, \\
& \lan k(k-1)\ran = \tfrac1{Z_\fa} \fa^2 Z_\fa'' = \fa^2, \\
& \lan k(k-1)(k-2)\ran = \tfrac1{Z_\fa} \fa^3 Z_\fa''' = \fa^3,
\end{align*}
which imply $\lan k^2\ran = \fa^2+\fa$ and $\lan k^3\ran = \fa^3+3\fa^2+\fa$.
In particular, $\rho=\fa$ and $\si^2(\rho) = \lan k^2\ran -\lan k\ran^2 = \fa = \rho$
so that the condition \eqref{eq:4.var-2} holds for all $\rho$.
On the other hand, one easily see
${\mathcal C}(\rho) = (\fa^3+3\fa^2+\fa) -3(\fa^2+\fa)\fa + 2 \fa^3 -(\fa^2+\fa)+\fa^2 =0$
so that  $\tilde g''(\rho)=0$ for all $\rho$.
Therefore, the nonlinearity in the limit SPDE becomes trivial if we take $g(k)=k$.

We believe however that \eqref{eq:4.g-2} holds for generic $g$.  To see this, let us consider
a perturbation $g(k)=k(1+h(k))$ of $g(k)=k$ with a function $h$ on $\PZ$ having 
a compact support.  Then, since $g(k)!= \frac{k!}{H(k)}$ with
$$
H(k) = \frac1{(1+h(k))!} = \frac1{(1+h(k))(1+h(k-1))\cdots (1+h(0))},
$$
we see
$$
\nu_h(k) \equiv \nu_{\fa,h}(k) = \frac1{Z_{\fa,h}}\frac{\fa^k}{g(k)!} = \frac{\fa^k H(k)}{Z_{\fa,h} k!},
$$
where
$$
Z_{\fa,h} = \sum_{k=0}^\infty \frac{\fa^k H(k)}{k!}.
$$
We denote now the expectation with respect to $\nu_h$ by $\lan \cdot\ran_h$; in particular,
$\lan\cdot\ran_0$ is the average under the Poisson distribution with parameter $\fa$.

\begin{lem}  \label{lem:4.7-a}
We have
\begin{align*}
& \lan k\ran_h = c_1\fa, \\
& \lan k(k-1) \ran_h = c_2\fa^2, \\
& \lan k(k-1)(k-2) \ran_h = c_3 \fa^3,
\end{align*}
where
$$
c_i = \frac{\lan H(k+i)\ran_0}{\lan H(k)\ran_0}, \quad i=1,2,3.
$$
In particular, the condition \eqref{eq:4.var-2} is equivalent to the statement that $c_1^2=c_2$ or $\fa=0$.
\end{lem}

\begin{proof}
Compared with the case of $h\equiv 0$ (i.e., $H\equiv 1$), we have 
\begin{align*}
& Z_{\fa,h} = Z_{\fa,0}\lan H(k)\ran_0, \\
& Z_{\fa,h}' = \sum_{k=1}^\infty \frac{\fa^{k-1}H(k)}{(k-1)!} = Z_{\fa,0}\lan H(k+1)\ran_0.
\end{align*}
On the other hand,
$$
Z_{\fa,h}' = \sum_{k=1}^\infty \frac{k \fa^{k-1}H(k)}{k!} = \fa^{-1}\lan k\ran_h Z_{\fa,h}.
$$
Therefore,
$$
\lan k\ran_h = \fa \frac{Z_{\fa,h}'}{Z_{\fa,h}} = \fa \frac{\lan H(k+1)\ran_0.}{\lan H(k)\ran_0}.
$$
The rest is similar.
\end{proof}

\begin{exa}
\rm
Let $h(0)=\frac1a-1, h(1)= \frac{a}b-1, h(2) = \frac{b}c-1, h(3) = h(4)=\cdots =0$
with $a,b,c\not=0$.  Then, $H(0)=a, H(1)=b, H(2) =H(3)=H(4)=\cdots=c$.  Thus,
\begin{align*}
\lan H(k)\ran_0\quad &=e^{-\fa} (a+b\fa + \frac{c}{2}\fa^2+ \frac{c}{3!}\fa^3+\cdots) \\
& = e^{-\fa} (a+b\fa - c-c\fa+ce^\fa), \\
\lan H(k+1)\ran_0&=e^{-\fa} (b+c\fa + \frac{c}{2}\fa^2+ \cdots) \\
& = e^{-\fa} (b-c+ce^\fa), \\
\lan H(k+2)\ran_0&= \lan H(k+3)\ran_0=\cdots = c,
\end{align*}
and we set
$$
c_1= \frac{(b-c)+ce^\fa}{(a-c) +(b-c) \fa +ce^\fa}, \quad
c_2= \frac{ce^\fa}{(a-c) +(b-c) \fa +ce^\fa}.
$$
The condition \eqref{eq:4.var-2} is equivalent to \lq\lq$c_1^2=c_2$ or $\fa=0$" as we pointed out.
On the other hand, by Lemma \ref{lem:4.7-a} and using this relation, the condition \eqref{eq:4.g-2} is:
$\fa\not=0$ and 
\begin{align*}
0 &\not= \lan k^3\ran - 3\lan k^2\ran\lan k\ran +2\lan k\ran^3
-\lan k^2\ran + \lan k\ran^2 \\
& = \lan k(k-1)(k-2)\ran + 2\lan k(k-1)\ran - 3\lan k(k-1)\ran \lan k\ran
-2\lan k\ran^2 + 2\lan k\ran^3 \\
& = c_2\fa^3 + 2c_2\fa^2 - 3c_2\fa^2\cdot c_1\fa -2c_1^2\fa^2 + 2c_1^3\fa^3 \\
& = c_1^2(1-c_1) \fa^3.
\end{align*}
Therefore, our condition for both \eqref{eq:4.g-2} and \eqref{eq:4.var-2} to hold is summarized as 
\lq\lq$c_1^2=c_2, c_1\not=1, \fa\not=0$",
which is realized under a proper choice of $a, b, c$ and $\fa$. Indeed, let $\varphi=1$ and then the condition $c_1^2=c_2$  reduces to 
\begin{equation*}
(b-c)^2 + (b-c)ce = (a-c)ce,
\end{equation*}
which can be realized if we take $a=1/2$, $b=1/4$ for some $c<1/4$.
\end{exa}

\section{Non-decoupleable model}
\label{ND_section}

We show that one can construct a model which satisfies all
necessary conditions (ND), (LG), (INV), (ORI), (SG), (LB)
and the Frame condition (FC)
(cf.\ Proposition \ref{prop:4.2})
such that some nonlinear coupling terms appearing in the system of KPZ-Burgers equation are non zero and remain such even after any suitable change of variables, i.e. the  KPZ-Burgers system cannot be fully decoupled.

Let us first define what means for us to have a coupled system of KPZ-Burgers equation.  We consider the stochastic PDE \eqref{eq:7.1} in a canonical form,
 that is, in the form \eqref{eq:6.2-b} (especially, the noise term of the form
$\nabla\dot\W_t$) and define the coupling constants $\Ga_{j\ell}^i$ as in
\eqref{eq:6.3-b}.  Then, noting that $\si \dot{\mathcal{W}}_t$ remains to be an
$\R^n$-valued space-time white noise (in law) if $\si$ is an 
orthogonal matrix, we consider the corresponding transform: 
$\tilde{\mathcal{Y}}_t := \si\bar{\mathcal{Y}}_t$.  Then $(\Ga_{j\ell}^i)_{i,j,\ell}$
is transformed as
\begin{equation}
\label{eq:sigmagamma}
(\si\circ\Ga)_{j\ell}^i := \sum_{i',j',\ell'=1}^n  \si_{i i'} \Ga_{j'\ell'}^{i'}
\si_{j j' } \si_{\ell \ell'}.
\end{equation}

\begin{defn}
\rm
We say that the KPZ-Burgers system \eqref{eq:6.2-b} is fully decoupleable if there exists an orthogonal matrix $\sigma$ such that for any $i \in \{1,\ldots, n\}$, the coupling constants $(\si\circ\Ga)_{j\ell}^i$ are zero for any $(j,\ell) \ne (i,i)$. Otherwise we say it cannot be fully decoupled. 

We say that it is partially decoupleable if there exists an orthogonal matrix $\sigma$ such that there exists $i \in \{1,\ldots, n\}$ for which the coupling constants $(\si\circ\Ga)_{j\ell}^i$ are zero for any $(j,\ell) \ne (i,i)$. Otherwise we say it cannot be partially decoupled. 
\end{defn}

Observe that a KPZ-Burgers system that can be partially decoupleable without being fully decoupleable and that if a system is not fully decoupleable then its evolution is non-trivial. The article \cite{FH17} discusses the example of \cite{EK} for which one can find
a matrix $\tau$ such that $(\tau\circ\Ga)_{j\ell}^i =0$ for any $i$ and $(j,\ell) \ne (i,i)$ but the noise term $\tau\dot\W_t$ has correlated components since $\tau$ is not an orthogonal matrix. Hence, according to the previous definition, this does not mean that this system is fully decoupleable.  

A natural question is to know under which conditions a multi-species zero range process satisfying the conditions (ND), (LG), (INV), (ORI), (SG), (LB) and the Frame condition (FC)
(cf. Proposition \ref{prop:4.2}) gives rise to a fully decoupleable (resp. partially decoupleable) KPZ-Burgers system. We believe that generically this is not the case but this seems to be a non-trivial question in full generality. 

In this section, we restrict however to the case $n=2$, and give an example of a multi-species zero range process which is not fully decoupleable. In fact we cannot hope better for $2$-components models since the next proposition shows that any $2$-components KPZ-Burgers system satisfying the trilinear condition \eqref{eq:7.tri} is partially decoupleable.    

\begin{prop}
Let $n=2$ and consider a KPZ-Burgers system (\ref{eq:6.2-b}) with the coupling constants (\ref{eq:6.3-b}) satisfying the trilinear condition (\ref{eq:7.tri}). Then, there exists an orthogonal matrix which partially decouples the KPZ-Burgers system. 
\end{prop}

\begin{proof}
For $\psi \in [0, 2\pi]$ let us consider the orthogonal matrix $\sigma_{\psi}$ given by 
\begin{equation*}
\sigma_\psi = 
\left(
\begin{array}{cc}
\cos \psi & -\sin  \psi\\
\sin  \psi & \cos  \psi
\end{array}
\right). 
\end{equation*} 
The condition $(\sigma_\psi \circ \Gamma)_{12}^1 = 0$ is equivalent to 
\begin{align*}
F(\psi)=&\cos^2  \psi \sin  \psi \; \Gamma_{11}^1 +\sin^2  \psi \cos  \psi \; \Gamma_{22}^2\\ 
&+(\cos^3  \psi - 2 \cos  \psi \sin^2  \psi )\; \Gamma_{11}^2 + (\sin^3  \psi - 2 \sin  \psi \cos^2  \psi)\;  \Gamma_{22}^1 \; = \; 0. 
\end{align*}
The continuous function $F:[0,2\pi] \to \R$ satisfies 
$$F(0) = \Gamma_{11}^2, \quad F(\pi) = -\Gamma_{11}^2.$$
Hence by the intermediate value theorem there exists $\psi \in [0,2\pi]$ such that  $F(\psi)=0$, which concludes the proof.   
\end{proof}

\subsection{Perturbation of models keeping the compatibility condition (INV)} 
 \label{sec:8.2}

The condition (INV) is equivalent to that there exists a function
$G:\Z_+^n \to \R$ such that $G({\BF 0})=0$ (normalization) and
$$
\log g_i(\Bk) = G(\Bk)-G(\Bk^i,k^i-1),
$$
see \cite{GS03}.  We perturb $G$ as
$$
G^\la(\Bk) = G(\Bk) - \log \la(\Bk),
$$
with some $\la=\{\la(\Bk)>0\}$ such that $\la({\BF 0})=1$.  Then, introducing the superscript $\lambda$ to signify the perturbation,
\begin{align*}
\log g_i^\la(\Bk) &= G^\la(\Bk)-G^\la(\Bk^i,k^i-1)\\
&= \log g_i(\Bk) - \log \la(\Bk) + \log \la(\Bk^i,k^i-1),
\end{align*}
and therefore
$$
g_i^\la(\Bk) = \frac{\la(\Bk^i,k^i-1)}{\la(\Bk)} g_i(\Bk).
$$
In particular, we have
$$
g^\la!(\Bk) = \prod  \frac{\la(\Bk^i,k^i-1)}{\la(\Bk)} g_i(\Bk),
$$
where the product is taken along a path connecting ${\BF 0}$ and $\Bk$ and, in fact,
\begin{align}  \label{eq:8.1}
g^\la!(\Bk) = \frac{g!(\Bk)}{\la(\Bk)}.
\end{align}
Then, by \eqref{eq:8.1}, we have
\begin{align*}
& Z_{\boldsymbol\fa}^\la = \sum_\Bk \frac{\boldsymbol\fa^\Bk}{g^\la!(\Bk)}
= e^{\fa^1+\fa^2} + \sum_{|\Bk|\le k_0} \left(\la(\Bk)-1\right) \frac{[\fa^1]^{k^1} [\fa^2]^{k^2}}{k^1!k^2!},\\
& \bar\nu_{\boldsymbol\fa}^\la(\Bk)
= \frac1{Z_{\boldsymbol\fa}^\la} \la(\Bk)\frac{[\fa^1]^{k^1}[\fa^2]^{k^2}}{k^1!k^2!},
\intertext{and}
& E_{\bar\nu_{\boldsymbol\fa}^\la}[k^1(k^1-1)\cdots(k^1-m+1)k^2(k^2-1)\cdots(k^2-n+1)] \\
& \hskip 15mm 
= \frac1{Z_{\boldsymbol\fa}^\la} \sum_{\Bk =(k^1,k^2)} \la(\Bk)\frac{[\fa^1]^{k^1} [\fa^2]^{k^2}}{(k^1-m)!(k^2-n)!},
\qquad m, n \ge 1.
\end{align*}

\subsection {Example of a non fully decoupleable perturbation}

We take $n=2$ and consider a perturbation of independent
random walks (or another candidate would be the multi-color system).  Let us consider 
the situation that all particles of
two different species perform totally independent random walks with
same speeds.  In other words, we take jump rates
\begin{align}  \label{eq:8.2}
g_1(\Bk)= g_1(k^1)=k^1!, \quad g_2(\Bk)= g_2(k^2)=k^2!
\end{align}
This satisfies (INV). Note that $Z_{\boldsymbol\fa} =Z_{\fa^1}Z_{\fa^2} =
e^{\fa^1}e^{\fa^2}$ for $\boldsymbol\fa =(\fa^1,\fa^2) \in (0,+\infty)^2$
and $\bar\nu_{\boldsymbol\fa} = \bar\nu_{\fa^1}\otimes\bar\nu_{\fa^2}$
is a product of two Poisson measures.  The Frame condition (FC)
is satisfied for all $\boldsymbol\fa$, since $\Ga_{12}=\Ga_{21}=0$ and
$\Ga_{ii}/\tilde{g}_i= \fa^i/\fa^i=1$ for $i=1,2$.

We perturb this system by changing finitely many $G(\Bk)$'s as $G^\la(\Bk)$ as we mentioned 
in Section \ref{sec:8.2}.  In this way, the number of parameters for the perturbation can be taken 
as many as we need and we will see that $2$ will be sufficient. The perturbation is chosen as $\lambda (\mathbf k) =1+x$ (resp. $1+y$)  with $x>-1$ (resp. $y>-1$) if $\mathbf k = (1,0)$ (resp. $\mathbf k= (0,1)$), and $\lambda (\mathbf k)=1$ otherwise. We have then 
\begin{equation*}
Z_{\boldsymbol\fa}^\la= e^{\varphi^1 + \varphi^2} + x \varphi^1 + y \varphi^2.
\end{equation*}
A simple computation shows that
\begin{align*}
E_{\bar\nu_{\boldsymbol\fa}^\la} (k^1) &= \tfrac{1}{Z_{\boldsymbol\fa}^\la} \left[ \varphi^1 e^{\varphi^1 +\varphi^2} + x \varphi^1 \right],\\
E_{\bar\nu_{\boldsymbol\fa}^\la} (k^2) &= \tfrac{1}{Z_{\boldsymbol\fa}^\la} \left[ \varphi^2 e^{\varphi^1 +\varphi^2} + y \varphi^2 \right],\\
E_{\bar\nu_{\boldsymbol\fa}^\la} (k^1 k^2) &= \tfrac{1}{Z_{\boldsymbol\fa}^\la} \, \varphi^1 \varphi^2 e^{\varphi^1 +\varphi^2},\\
E_{\bar\nu_{\boldsymbol\fa}^\la} ([k^1]^2) &= \tfrac{1}{Z_{\boldsymbol\fa}^\la} \left[ (\varphi^1 +[\varphi^1]^2) e^{\varphi^1 +\varphi^2} + x \varphi^1 \right],\\
E_{\bar\nu_{\boldsymbol\fa}^\la} ([k^2]^2) &= \tfrac{1}{Z_{\boldsymbol\fa}^\la} \left[ (\varphi^2 +[\varphi^2]^2) e^{\varphi^1 +\varphi^2} + y \varphi^2 \right].
\end{align*}
Later we will also need some third moments estimates:
\begin{align*}
E_{\bar\nu_{\boldsymbol\fa}^\la} ([k^1]^3) &=  \tfrac{1}{Z_{\boldsymbol\fa}^\la} \left[ ([\varphi^1]^3 + 3 [\varphi^1]^2 + \varphi^1) e^{\varphi^1 +\varphi^2} + x \varphi^1 \right],\\
E_{\bar\nu_{\boldsymbol\fa}^\la} ([k^2]^3) &=  \tfrac{1}{Z_{\boldsymbol\fa}^\la} \left[ ([\varphi^2]^3 + 3 [\varphi^2]^2 + \varphi^2) e^{\varphi^1 +\varphi^2} + y \varphi^2 \right],\\
E_{\bar\nu_{\boldsymbol\fa}^\la} ([k^1]^2 k^2) &=\tfrac{1}{Z_{\boldsymbol\fa}^\la} \varphi^2 \varphi^1 ( \varphi^1 +1) e^{\varphi^1 +\varphi^2},\\
E_{\bar\nu_{\boldsymbol\fa}^\la} ([k^2]^2 k^1) &=\tfrac{1}{Z_{\boldsymbol\fa}^\la}\varphi^2 \varphi^1 ( \varphi^2 +1) e^{\varphi^1 +\varphi^2}.
\end{align*}
It follows that 
\begin{align*}
{\rm cov}\Big(\bar\nu_{\boldsymbol\fa}^\la\Big)= \cfrac{1}{[Z_{\boldsymbol\fa}^\la]^2}
\left(
\begin{array}{cc}
N_{11}
&N_{12} \\
&\\
N_{21} 
&
N_{22}
\end{array}
\right)
\end{align*}
with
\begin{align*}
N_{11} &= \varphi^1 \{e^{2(\varphi^1+\varphi^2)} + x e^{\varphi^1 + \varphi^2} ([\varphi^1]^2 -\varphi^1 +1) + y e^{\varphi^1 + \varphi^2} (1+\varphi^1) \varphi^2 + xy \varphi^2\},\\
N_{12}&=N_{21}=x e^{\varphi^1 + \varphi^2} \varphi^1 \varphi^2 (\varphi^1 -1) +  y e^{\varphi^1 + \varphi^2} \varphi^1 \varphi^2 (\varphi^2 -1) - xy \varphi^{1}\varphi^2,\\
N_{22}&=\varphi^2 \{e^{2(\varphi^1+\varphi^2)} + y e^{\varphi^1 + \varphi^2} ([\varphi^2]^2 -\varphi^2 +1)+ x e^{\varphi^1 + \varphi^2} (1+\varphi^2) \varphi^1 + xy \varphi^1\}.
\end{align*}

The Frame condition (FC)
is satisfied if 
\begin{align}
\label{eq:choicephi2}
&x (\varphi^1 -1) + y (\varphi^2 -1) =xy e^{-(\varphi^1 +\varphi^2)},\\
&x[(\varphi^1 -1)^2 -\varphi^1 \varphi^2] + y [\varphi^1 \varphi^2 - (\varphi^2 -1)^2] = xy (\varphi^1- \varphi^2) e^{-(\varphi^1+\varphi^2)}.
\end{align}
This implies 
\begin{equation*}
x=y=(\varphi^1+\varphi^2 -2) e^{\varphi^1 + \varphi^2}
\end{equation*}
or 
\begin{equation*}
\varphi^1 + \varphi^2 =1, \quad x (\varphi^1 -1) -y \varphi^1 -e^{-1} xy =0.
\end{equation*}
We make the choice 
\begin{equation*}
\varphi^1 + \varphi^2 =1, \quad x (\varphi^1 -1) -y \varphi^1 -e^{-1} xy =0.
\end{equation*}

We can choose $\varphi^1 \in (0,1)$ and $x>-1$ arbitrarily but the condition $y>-1$ has to be satisfied which gives some restrictions. Assuming this, we get then that 
\begin{align*}
& Z_{\boldsymbol\fa}^\la = e+x \varphi^1 + y \varphi^2
\end{align*}
and
\begin{align*}
{\rm cov}\Big(\bar\nu_{\boldsymbol\fa}^\la\Big)=
\left(
\begin{array}{cc}
e\varphi^1 &0\\
0&e \varphi^2 
\end{array}
\right)
\end{align*}
To simplify notations we denote by $M:=M_{\boldsymbol\fa}^\la$ the matrix 
\begin{align*}
M&=\left[ {\rm cov}\Big(\bar\nu_{\boldsymbol\fa}^\la\Big)  \right]^{-1}= e^{-1} \, 
\left(
\begin{array}{cc}
1/\varphi^1 &0\\
0&1/\varphi^2 
\end{array}
\right).
\end{align*}

With this choice, we compute now the coupling constants $\Gamma_{j\ell}^i$ by using the computations performed in the proof of Proposition  \ref{prop:6.1-b}, which give
\begin{equation*}
\Gamma_{j\ell}^i =c\, \cfrac{\sqrt{\varphi^i \varphi^j \varphi^\ell }}{(\varphi^i M_{ii})^{3/2}} \left\{ M_{i\ell} M_{ij} - \sum_{m,p,q=1}^2 M_{ip} M_{qj} M_{m\ell} \, \kappa (k^p,k^q,k^\ell) \right\}
\end{equation*}
where $\kappa (X,Y,Z)$ denotes the joint cumulant of the random variables $X,Y,Z$ under the probability measure ${\bar \nu}_{\boldsymbol\fa}^\la$. We obtain
\begin{align}
\label{eq:Gamma-exp}
\Gamma_{11}^1 &= c \sqrt{M_{11}} \left\{1- M_{11}\,  \kappa (k^1,k^1,k^1)\right\}=\cfrac{c}{\sqrt{ e \varphi_1}} \left\{1- \cfrac{\kappa (k^1,k^1,k^1)}{e \varphi^1} \right\} ,\\
\Gamma_{22}^2 &= c \sqrt{M_{22}} \left\{1- M_{22} \, \kappa (k^2,k^2,k^2)\right\} = \cfrac{c}{\sqrt{ e \varphi^2}} \left\{1- \cfrac{\kappa (k^2,k^2,k^2)}{e \varphi^2} \right\} ,\\
\Gamma_{11}^2 &= - c \cfrac{\varphi^1}{\varphi^2} \cfrac{M_{11}^2}{\sqrt{M_{22}}}\,  \kappa (k^1, k^1,k^2) = - c \cfrac{1}{e^{3/2} [\varphi^2]^{1/2} \varphi^1} \, \kappa (k^1,k^1, k^2),\\
\Gamma_{22}^1 &= - c \cfrac{\varphi^2}{\varphi^1} \cfrac{M_{22}^2}{\sqrt{M_{11}}}\,  \kappa (k^2, k^2,k^1)= - c \cfrac{1}{e^{3/2} [\varphi^1]^{1/2} \varphi^2} \, \kappa (k^2,k^2, k^1).
\end{align}
We get the following expressions for the following moments 
\begin{align*}
E_{\bar\nu_{\boldsymbol\fa}^\la} (k^1) &= \tfrac{1}{Z_{\boldsymbol\fa}^\la} \left[ \varphi^1 e^{\varphi^1 +\varphi^2} + x \varphi^1 \right]=\cfrac{\varphi^1 (e +x) }{e+x \varphi^1 + y \varphi^2}\\
E_{\bar\nu_{\boldsymbol\fa}^\la} (k^2) &= \tfrac{1}{Z_{\boldsymbol\fa}^\la} \left[ \varphi^2 e^{\varphi^1 +\varphi^2} + y \varphi^2 \right]=\cfrac{\varphi^2 (e +y) }{e+x \varphi^1 + y \varphi^2},\\
E_{\bar\nu_{\boldsymbol\fa}^\la} (k^1 k^2) &= \tfrac{1}{Z_{\boldsymbol\fa}^\la} \varphi^1 \varphi^2 e^{\varphi^1 +\varphi^2} =\cfrac{e \varphi^1 \varphi^2 }{e+x \varphi^1 + y \varphi^2}\\
E_{\bar\nu_{\boldsymbol\fa}^\la} ([k^1]^2) &= \tfrac{1}{Z_{\boldsymbol\fa}^\la} \left[ (\varphi^1 +[\varphi^1]^2) e^{\varphi^1 +\varphi^2} + x \varphi^1 \right]= \cfrac{(\varphi^1 + [\varphi^1]^2)  e + x \varphi^1 }{e+x \varphi^1 + y \varphi^2 },\\
E_{\bar\nu_{\boldsymbol\fa}^\la} ([k^2]^2) &= \tfrac{1}{Z_{\boldsymbol\fa}^\la} \left[ (\varphi^2 +[\varphi^2]^2) e^{\varphi^1 +\varphi^2} + y \varphi^2 \right]= \cfrac{(\varphi^2 + [\varphi^2]^2)  e + y \varphi^2 }{e+x \varphi^1 + y \varphi^2}\\
E_{\bar\nu_{\boldsymbol\fa}^\la} ([k^1]^3) &=  \tfrac{1}{Z_{\boldsymbol\fa}^\la} \left[ ([\varphi^1]^3 + 3 [\varphi^1]^2 + \varphi^1) e^{\varphi^1 +\varphi^2} + x \varphi^1 \right]=\cfrac{([\varphi^1]^3 + 3 [\varphi^1]^2 + \varphi^1) e  + x \varphi^1}{e+x \varphi^1 + y \varphi^2}\\
E_{\bar\nu_{\boldsymbol\fa}^\la} ([k^2]^3) &=\tfrac{1}{Z_{\boldsymbol\fa}^\la} \left[( [\varphi^2]^3 + 3[\varphi^2]^2 +\varphi^2) e^{\varphi^1 +\varphi^2}+y \varphi^2 \right]= \cfrac{([\varphi^2]^3 + 3 [\varphi^2]^2 + \varphi^2) e  + y \varphi^2}{e+x\varphi^1 + y \varphi^2},\\
E_{\bar\nu_{\boldsymbol\fa}^\la} ([k^1]^2 k^2) &=\tfrac{1}{Z_{\boldsymbol\fa}^\la} \varphi^2 \varphi^1 ( \varphi^1 +1) e^{\varphi^1 +\varphi^2}= \cfrac{\varphi^2 \varphi^1 ( \varphi^1 +1) e}{e+x\varphi^1 + y \varphi^2},\\
E_{\bar\nu_{\boldsymbol\fa}^\la} ([k^2]^2 k^1) &=\tfrac{1}{Z_{\boldsymbol\fa}^\la}\varphi^2 \varphi^1 ( \varphi^2 +1) e^{\varphi^1 +\varphi^2}=\cfrac{\varphi^2 \varphi^1 ( \varphi^2 +1) e}{e+x\varphi^1 + y \varphi^2}.
\end{align*}

Recall the discussion at the beginning of Section \ref{eq:sigmagamma}. Since $n=2$ the system is fully decoupleable  if and only if there exists an orthogonal matrix $\sigma$ such that 
$$(\sigma \circ \Gamma)_{12}^1 = (\sigma\circ \Gamma)_{12}^{2} = 0. $$ 
We recall that any orthogonal matrix $\sigma$ is in the form 
\begin{equation*}
\sigma_\psi = 
\left(
\begin{array}{cc}
\cos \psi & -\sin  \psi\\
\sin  \psi & \cos  \psi
\end{array}
\right) 
\quad \text{or} \quad 
{\tilde \sigma}_\psi = 
\left(
\begin{array}{cc}
-\cos  \psi & \sin  \psi\\
\sin  \psi & \cos  \psi
\end{array}
\right), 
\end{equation*}
where $ \psi \in [0,2\pi]$.
 
\begin{center}
\begin{figure}[!h]
\begin{center}
\includegraphics[scale=0.5]{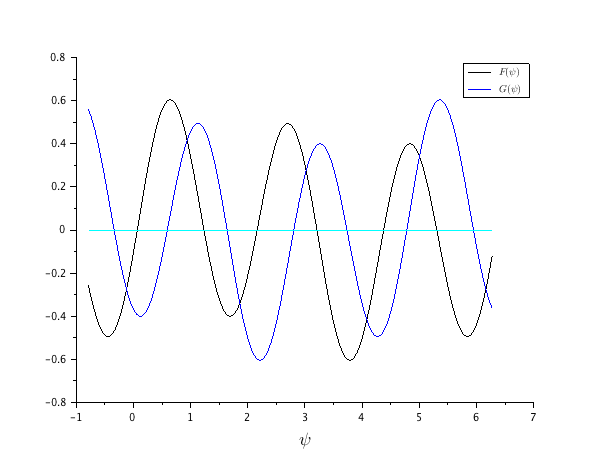}
\end{center}
\caption{Graphs of the functions $F,G$ when $\varphi_1=0.49$, $x=3$ (and therefore $y=-0.96>-1$).}
\label{figure1}
\end{figure}
\end{center}
Therefore, the condition $(\sigma_\psi \circ \Gamma)_{12}^1 = 0$  is equivalent to
\begin{equation}
\begin{split}
F(\psi)&=\cos^2  \psi \sin  \psi \; \Gamma_{11}^1 +\sin^2  \psi \cos  \psi \; \Gamma_{22}^2\\ 
&+(\cos^3  \psi - 2 \cos  \psi \sin^2  \psi )\; \Gamma_{11}^2 + (\sin^3  \psi - 2 \sin  \psi \cos^2  \psi)\;  \Gamma_{22}^1 \; = \; 0. 
\end{split}
\end{equation}
while the condition $({\tilde \sigma}_\psi \circ \Gamma)_{12}^1=0$ is equivalent to the same condition.
Similarly the condition $(\sigma_\psi \circ \Gamma)_{12}^2 = 0$ is equivalent to 
\begin{equation}
\label{eq:finpsi2}
\begin{split}
G(\psi)&=\sin^2  \psi \cos  \psi \; \Gamma_{11}^1 -\cos^2  \psi \sin  \psi \; \Gamma_{22}^2\\ 
& +(2 \sin  \psi \cos^2  \psi -\sin^3 \psi)\;  \Gamma_{11}^2+(\cos^3  \psi - 2 \cos  \psi \sin^2  \psi )\; \Gamma_{22}^1 \; = \; 0,
\end{split}
\end{equation}
while the condition $({\tilde \sigma}_\psi \circ \Gamma)_{12}^2=0$ is equivalent to the same condition.


A numerical simulation (see Figure \ref{figure1}) shows that if we choose $x=3$ and $\varphi_1=0.49$, then the functions $F$ and $G$ (resp. $F$ and $H$) never vanish simultaneously. For these values, we have that $\Gamma_{11}^1, \Gamma_{22}^2, \Gamma_{11}^2, \Gamma_{22}^1$ are all non zero. 


%


\section*{Acknowledgements}
The work of C. Bernardin has been supported by the projects EDNHS ANR-14- CE25-0011, LSD ANR-15-CE40-0020-01 of the French National Research Agency (ANR), and funding from the European Research Council (ERC) under  the European Union's Horizon 2020 research and innovative programme (grant agreement No 715734).  T. Funaki was supported in part by JSPS KAKENHI, Grant-in-Aid for Scientific Researches (A) 18H03672 and (S) 16H06338.
S. Sethuraman was supported by grant ARO W911NF-181-0311, a Simons Foundation Sabbatical grant, and by a Japan Society for the Promotion of Science Fellowship. 

\end{document}